\documentclass[final,1p]{elsarticle} 
\usepackage[english]{babel}
\usepackage[latin1]{inputenc}
\usepackage{amsfonts}
\usepackage{amssymb}
\usepackage{amsmath}
\usepackage{amsthm}
\usepackage{graphicx}
\usepackage{epsfig}
\usepackage{booktabs}
\usepackage{extarrows} 
\usepackage{float}
\usepackage{mathtools}
\usepackage[all]{xy}
\usepackage{parskip}
\usepackage{blindtext}
\usepackage{hyperref}
\usepackage{bbm}
\usepackage{mathrsfs} 
\usepackage[all]{xy}
\usepackage{cleveref} 

\usepackage[usenames,dvipsnames]{color}

\usepackage[symbol]{footmisc}

\DeclareMathOperator{\tr}{tr}
\DeclareMathOperator{\Tr}{Tr}

\DeclareMathOperator{\End}{End}
\DeclareMathOperator{\SO}{SO}
\DeclareMathOperator{\U}{U}
\DeclareMathOperator{\SU}{SU}
\DeclareMathOperator{\M}{M}
\DeclareMathOperator{\GL}{GL}
\DeclareMathOperator{\SL}{SL}
\DeclareMathOperator{\Ad}{Ad}
\DeclareMathOperator{\Real}{Re}
\DeclareMathOperator{\Symp}{Sp}
\DeclareMathOperator{\Sp}{Sp}
\DeclareMathOperator{\diag}{diag}
\DeclareMathOperator{\loc}{loc}
\DeclareMathOperator{\Lie}{Lie}
\DeclareMathOperator{\Co}{Co}

\newcommand{\C}{\mathbb{C}}
\newcommand{\R}{\mathbb{R}}
\newcommand{\N}{\mathbb{N}}
\newcommand{\Z}{\mathbb{Z}}

\newcommand{\KK}{\mathbb{K}}
\newcommand{\so}{\mathfrak{so}}
\newcommand{\Ind}{\mathbbm{1}}
\newcommand{\wX}{\widetilde{X}}
\newcommand{\wY}{\widetilde{Y}}
\newcommand{\p}[1]{\frac{\partial}{\partial #1}}
\newcommand{\pp}[2]{\frac{\partial^2}{\partial #1 \partial #2}}
\newcommand{\inpro}[1]{\left\langle #1 \right\rangle}
\newcommand{\V}{\mathbf{V}}
\newcommand{\bSON}{\beta_{\so(N,\R)}}
\newcommand{\dSON}{d_{\SO(N,\R)}}
\newcommand{\LSON}{\Delta_{\SO(N,\R)}}
\newcommand{\K}{\mathcal{K}}
\newcommand{\wL}{\widetilde{\mathcal{L}}}
\newcommand{\Cuuv}{\C[u,u^{-1};\mathbf{v}]}
\newcommand{\Cmuuv}{\C_m[u,u^{-1};\mathbf{v}]}
\newcommand{\Cuu}{\C[u,u^{-1}]}
\newcommand{\ASONst}{\mathcal{A}_{s,\tau}^{\SO(N,\C)}} 
\newcommand{\BSONst}{\mathbf{B}_{s,\tau}^{\SO(N,\R)}} 
\newcommand{\Gst}{\mathscr{G}_{s,\tau}} 
\newcommand{\Hst}{\mathscr{H}_{s,\tau}} 
\newcommand{\MNC}{\M_N(\C)}

\newcommand{\DNst}{\widetilde{\mathcal{D}}^{s,\tau}_N}
\newcommand{\Lzst}{\widetilde{\mathcal{L}}_0^{s,\tau}}
\newcommand{\LmuSONst}{{L^2(\mu_{s,\tau}^{\SO(N,\C)})}}
\newcommand{\exptDN}{e^{\frac{\tau}{2}\mathcal{D}_N^{(1)}}}
\newcommand{\exptLz}{e^{\frac{\tau}{2}\mathcal{L}_0}}
\newcommand{\RN}{R^{(N)}}
\newcommand{\dP}{\frac{\partial P}{\partial v_\varepsilon}}
\newcommand{\ddP}{\frac{\partial^2 P}{\partial v_\varepsilon\partial v_\delta}}
\newcommand{\ejkzero}{{\varepsilon_{j,k}^0}}
\newcommand{\ejkone}{{\varepsilon_{j,k}^1}}
\newcommand{\ejktwo}{{\varepsilon_{j,k}^2}}
\newcommand{\ejkthree}{{\varepsilon_{j,k}^3}}

\newcommand{\symp}{\mathfrak{sp}}
\newcommand{\LSpN}{\Delta_{\Symp(N)}}
\newcommand{\HH}{\mathbb{H}}
\newcommand{\idH}{\mathbf{1}}
\newcommand{\ii}{\mathbf{i}}
\newcommand{\jj}{\mathbf{j}}
\newcommand{\kk}{\mathbf{k}}
\newcommand{\bSpN}{\beta_{\symp(N)}}
\newcommand{\ASpst}{\mathcal{A}_{s,\tau}^{\Sp(N,\C)}}
\newcommand{\wSpN}{\widetilde{\Sp}(N)}
\newcommand{\wsympN}{\widetilde{\symp}(N)}

\newcommand{\wtr}{\widetilde{\tr}}
\newcommand{\LieK}{\mathfrak{k}}
\newcommand{\su}{\mathfrak{su}}
\newcommand{\uu}{\mathfrak{u}}

\newcommand{\Fabg}{F_{a,b}^\gamma}
\newcommand{\FabidH}{F_{a,b}^{\idH}}
\newcommand{\Fab}{F_{a,b}}
\newcommand{\Fbag}{F_{b,a}^\gamma}
\newcommand{\FbaidH}{F_{b,a}^{\idH}}

\newcommand{\DNbeta}{\mathcal{D}_N^{(\beta)}}
\newcommand{\Lzerobeta}{\mathcal{L}_0}
\newcommand{\Lonebeta}{\mathcal{L}_1^{(\beta)}}
\newcommand{\Ltwobeta}{\mathcal{L}_2^{(\beta)}}
\newcommand{\rhosNbeta}{\rho_s^{N,(\beta)}}
\newcommand{\mustNbeta}{\mu_{s,\tau}^{N,(\beta)}}

\newcommand{\LL}{{\wL_0^{s,\tau}+\frac{1}{N}\wL_1^{s,\tau}}}

\theoremstyle{plain}
\newtheorem{theorem}{Theorem}[section]
\newtheorem*{theorem*}{Theorem}
\newtheorem{lemma}[theorem]{Lemma}
\newtheorem*{lemma*}{Lemma}
\newtheorem{prop}[theorem]{Proposition}
\newtheorem*{prop*}{Proposition}
\newtheorem{cor}[theorem]{Corollary}
\newtheorem*{cor*}{Corollary}
\newtheorem*{claim*}{Claim}

\theoremstyle{definition}
\newtheorem{definition}[theorem]{Definition}
\newtheorem*{definition*}{Definition}
\newtheorem{remark}[theorem]{Remark}
\newtheorem*{remark*}{Remark}
\newtheorem{example}[theorem]{Example}
\newtheorem*{example*}{Example}

\newtheorem*{exer*}{Exercise}
\newtheorem{notation}[theorem]{Notation}

\makeatletter
\def\thm@space@setup{%
	\thm@preskip=\parskip \thm@postskip=0pt
}
\makeatother

\numberwithin{equation}{section}

\journal{Journal of Functional Analysis}

\begin{document}	

\begin{frontmatter}
	
	\title{The Segal-Bargmann Transform on Classical Matrix Lie Groups}
	
	\author[1]{Alice Z. Chan\footnote{Supported by NSF GRFP under Grant No. DGE-1650112}}
	\ead{azchan@ucsd.edu}

	\address[1]{Department of Mathematics,  University of California, San Diego,  La Jolla, CA 92093}
	
	\begin{abstract}
		We study the complex-time Segal-Bargmann transform $\mathbf{B}_{s,\tau}^{K_N}$ on a compact type Lie group $K_N$, where $K_N$ is one of the following classical matrix Lie groups: the special orthogonal group $\SO(N,\R)$, the special unitary group $\SU(N)$, or the compact symplectic group $\Sp(N)$. Our work complements and extends the results of Driver, Hall, and Kemp on the Segal-Bargman transform for the unitary group $\U(N)$. We provide an effective method of computing the action of the Segal-Bargmann transform on \emph{trace polynomials}, which comprise a subspace of smooth functions on $K_N$ extending the polynomial functional calculus. Using these results, we show that as $N\to\infty$, the finite-dimensional transform $\mathbf{B}_{s,\tau}^{K_N}$ has a meaningful limit $\mathscr{G}_{s,\tau}$ which can be identified as an operator on the space of complex Laurent polynomials.
		
	\end{abstract}
	
\end{frontmatter}

\tableofcontents

\section{Introduction}

\subsection{The classical Segal-Bargmann transform} \label{sec:SBclassical}
In this paper, we consider a complex-time generalization of the classical Segal-Bargmann transform on the matrix Lie groups $\SO(N,\R)$, $\SU(N)$, and $\Sp(N)$, and analyze its limit as $N\to\infty$. We begin with a brief discussion of the classical Segal-Bargmann transform and its generalizations to Lie groups of compact type. For $t>0$, let $\rho_t$ denote the heat kernel with variance $t$ on $\R^d$: $$\rho_t(x)=(2\pi t)^{-d/2}\exp(-|x|^2/2t).$$ This has an entire analytic continuation to $\C^d$, given by $$(\rho_t)_{\C}(z)=(2\pi t)^{-d/2}\exp\left(-\frac{z\cdot z}{2t}\right).$$ If $f\in L^1_{\loc}(\R^d)$ is of sufficiently slow growth, we can define 
\begin{equation}
	(B_tf)(z):=\int_{\R^d}(\rho_t)_\C (z-y)f(y)\, dy. \label{eq:SBclassical}
\end{equation}
The map $f\mapsto B_tf$ is called the classical {\bf Segal-Bargmann transform}, due to the work of the eponymous authors in \cite{Bargmann:Hilbertremarks, Bargmann:Hilbertspace, Segal:complexwave, Segal:mathchar, Segal:mathproblems}. If we let $\mu_t$ denote the heat kernel on $\C^d$, now with variance $\frac{t}{2}$: 
\begin{equation}\label{eq:heatkernelCd}
	\mu_t(z)=(\pi t)^{-d}\exp(-|z|^2/t),
\end{equation}
then the main result is that the map $$B_t:L^2(\R^d,\rho_t)\to\mathcal{H}L^2(\C^d,\mu_t)$$ is a unitary isomorphism, where $\mathcal{H}L^2(\C^d,\mu_t)$ is the space of square-integrable holomorphic functions on $\C^d$.

\subsection{Main definitions and theorems}

In \cite{Hall:SB}, Hall showed that the Segal-Bargmann transform can be extended to compact Lie groups, and in \cite{D:SBH}, the transform was further extended by Driver to Lie groups of compact type. Later, the behavior of the transform on the unitary group $\U(N)$ as $N\to\infty$ was analyzed independently by C\'{e}bron in \cite{Ceb:transform} and Driver, Hall, and Kemp in \cite{DHK:U(N)}. The authors of the latter paper subsequently introduced a complex-time generalization of the Segal-Bargmann transform for all compact type Lie groups in \cite{DHK:complextimeSB}. We use this version of the transform (see Definition \ref{def:boostedSB}) in our analysis of the special orthogonal group $\SO(N,\R)$, the special unitary group $\SU(N)$, and the compact symplectic group $\Sp(N)$. We now outline the main results of this paper, deferring a fuller discussion of the requisite background and definitions to Section \ref{sec:background}. 

Let $K$ denote a Lie group of compact type with Lie algebra $\mathfrak{k}$. Then $K$ possesses a left-invariant metric whose associated Laplacian $\Delta_K$ is bi-invariant (cf. \eqref{eq:Laplace}). For each $s>0$, there is an associated heat kernel $\rho_s^K\in C^\infty (K,(0,\infty))$ satisfying
\begin{equation} \label{eq:heatkernel}
	\left(e^{\frac{s}{2}\Delta_K}f\right)(x)=\int_K f(xk)\rho_s^K(k)\ dk=\int_K f(k)\rho_s^K(xk^{-1})\ dk\hspace{.5cm}\mbox{for all $f\in L^2(K)$}.
\end{equation}

Let $K_\C$ denote the complexification of $K$ and $\C_+=\{\tau=t+iu\, :\, t>0,u\in\R\}$ denote the right half-plane. As shown in \cite[Theorem 1.2]{DHK:complextimeSB}, the heat kernel $(\rho_s^K)_{s>0}$ possesses a unique entire analytic continuation $$\rho_{\C}^K:\C_+\times K_\C\to\C$$ such that $\rho_{\C}^K(s,x)=\rho_s(x)$ for all $s>0$ and $x\in K\subseteq K_\C$. Thus we can proceed as in the classical case and define the Segal-Bargmann transform for compact type Lie groups using a group convolution formula similar to \eqref{eq:SBclassical} (see \cite{D:SBH} for details). 

\begin{definition}[Complex-time Segal-Bargmann transform] \label{def:boostedSB}
	Let $s>0$ and $\tau = t+i\theta\in\mathbb{D}(s,s)$, and let $\rho_s^K$ and $\rho_{\C}^K$ be as above. For $z\in K_\C$ and $f\in L^2(K,\rho_s^K)$, define
	\begin{equation}
	(B_{s,\tau}^K f)(z):=\int_K\rho_\C^K(\tau,zk^{-1})f(k)\, dk\hspace{.5cm}\mbox{for }z\in K_\C.
	\end{equation}
\end{definition}

The complex-time Segal-Bargmann transform $B_{s,\tau}^K$ is an isometric isomorphism between certain $L^2$ and holomorphic $L^2$ spaces on $K$ and $K_\C$:

\begin{theorem}[Driver, Hall, and Kemp, {\cite[Theorem 1.6]{DHK:complextimeSB}}]
	Let $K$ be a connected, compact-type Lie group, and set $s>0$ and $\tau\in\mathbb{D}(s,s)$. Let $\mu_{s,\tau}^{K_\C}$ denote the (time-rescaled) heat kernel on $K_\C$ defined by \eqref{eq:heatKC}. Then $B_{s,\tau}^K f$ is holomorphic on $K_\C$ for each $f\in L^2(K,\rho_s^K)$. Moreover,
	\begin{equation}
	B_{s,\tau}^K:L^2(K,\rho_s^K)\to\mathcal{H}L^2(K_\C,\mu_{s,\tau}^{K_\C})
	\end{equation}
	is a unitary isomorphism. 
\end{theorem}

If we set $K=\R^d$ and $s=\tau >0$, we recover the main result for the classical Segal-Bargmann transform from Section \ref{sec:SBclassical}. 

Every compact Lie group $K$ has a faithful representation. Hence, viewing $K$ as a matrix group, we can extend $B_{s,\tau}^K$ to matrix-valued functions. Let $\M_N(\R)$ denote the algebra of $N\times N$ real matrices with unit $I_N$, and let $\M_N(\C)$ denote the algebra of $N\times N$ complex matrices. Let $\GL_N(\C)$ denote the group of all invertible matrices in $\M_N(\C)$. Define a scaled Hilbert-Schmidt norm on $\MNC$ by
\begin{equation} \label{eq:MNCnorm}
	\|A\|^2_{\MNC}:=\frac{1}{N}\Tr(AA^*),\hspace{.5cm}A\in\MNC,
\end{equation}
where $\Tr$ denotes the usual trace. 

\begin{definition}[Boosted Segal-Bargmann transform]
	Let $K\subseteq\GL_N(\C)$ be a compact matrix Lie group with complexification $K_\C$. Let $F$ be an $\MNC$ valued function on $K$ or $K_\C$, and let $\|F\|_{\MNC}$ denote the scalar-valued function $A\mapsto\|F(A)\|_{\MNC}$. Fix $s>0$ and $\tau\in\mathbb{D}(s,s)$, and let
	\begin{align*}
	L^2(K,\rho_s^K;\MNC)&=\{F:K\to\MNC;\,\|F\|_{\MNC}\in L^2(K,\rho_s^K)\},\hspace{.5cm}\mbox{and}\\
	L^2(K_\C,\mu_{s,\tau}^{K_\C};\MNC)&=\{F:K_\C\to\MNC;\,\|F\|_{\MNC}\in L^2(K_\C,\mu_{s,\tau}^{K_\C})\}.
	\end{align*}
	
	The norms defined by
	\begin{align}
	\|F\|^2_{L^2(K,\rho_s^K;\MNC)}&:=\int_{K}\|F(A)\|^2_{\MNC}\,\rho_s^K(A)\, dA\\
	\|F\|^2_{L^2(K_\C,\mu_{s,\tau}^{K_\C};\MNC)}&:=\int_{K_\C}\|F(A)\|^2_{\MNC}\,\mu_{s,\tau}^{K_\C}(A)\, dA \label{eq:LtmuNstnorm}
	\end{align}
	make these spaces into Hilbert spaces. Let $\mathcal{H}L^2(K_\C,\mu_{s,\tau}^{K_\C};\MNC)\subseteq L^2(K_\C,\mu_{s,\tau}^{K_\C};\MNC)$ denote the subspace of matrix-valued holomorphic functions. 
	
	The {\bf boosted Segal-Bargmann transform} $$\mathbf{B}_{s,\tau}^K:L^2(K,\rho_s^K;\MNC)\to \mathcal{H}L^2(K_\C,\mu_{s,\tau}^{K_\C};\MNC)$$
	is the unitary isomorphism determined by applying $B_{s,\tau}^N$ componentwise; i.e., $$\mathbf{B}_{s,\tau}^K(f\cdot A)=B_{s,\tau}^Kf\cdot A\hspace{.5cm}\mbox{for }f\in L^2(K,\rho_s^K)\mbox{ and }A\in\MNC.$$
\end{definition}

In this paper, we consider the case when $K=K_N$ is one of $\SO(N,\R)$, $\SU(N)$, or $\Sp(N)$. We provide an effective method of computing $\mathbf{B}_{s,\tau}^{K_N}$ on the space of polynomials, and show that in this setting, the transform has a meaningful limit as $N\to\infty$. However, the boosted Segal-Bargmann transform typically does not preserve the space of polynomials on $K$. For example, consider the polynomial $P_N(A)=A^2$ on $\SO(N,\R)$. Let $\Tr$ denote the usual trace and set $\tr = \frac{1}{N}\Tr$. From Example \ref{ex:BSONcomputation}, we see that
\begin{equation} \label{eq:tracepolyex}
(\mathbf{B}_{s,t}^{\SO(N,\R)}P_N)(A)=\frac{1}{N}(1-e^{-t})I_N+\frac{1}{2}e^{-t}(1+e^{\frac{2t}{N}})A^2-\frac{N}{2}e^{-t}(-1+e^{\frac{2t}{N}})A\tr A,
\end{equation}
This example highlights that in general, $\mathbf{B}_{s,\tau}^K$ maps the space of polynomial functions on $K$ into the larger space of {\bf trace polynomials}, first introduced in \cite[Definition 1.7]{DHK:U(N)}, which we now briefly describe (see Section \ref{sec:intertwineSON} for the full details).

Let $\Cuuv$ denote the algebra of polynomials in the commuting variables $u$, $u^{-1}$, and $\mathbf{v}=\{v_{\pm 1},v_{\pm 2},\hdots \}$. For each $P\in\Cuuv$, there is an associated trace polynomial $P_N^{(1)}$ on $\SO(N,\R)$. For now, it is enough to think of $P_N^{(1)}(A)$ as a certain ``polynomial'' function of $A$ and $\tr(A)$. If $P\in\Cuuv$ is a function of the variable $u$ alone, then $P_N^{(1)}$ is simply the corresponding polynomial on $\SO(N,\R)$. For example, the trace polynomial on $\SO(N,\R)$ associated with $P(u;\mathbf{v})=u^5$ is $P_N^{(1)}(A)=A^5$. On the other hand, \eqref{eq:tracepolyex} is the trace polynomial on $\SO(N,\R)$ associated with $$f(u;\mathbf{v})=\frac{1}{N}(1-e^{-t})+\frac{1}{2}e^{-t}(1+e^{\frac{2t}{N}})u^2-\frac{N}{2}e^{-t}(-1+e^{\frac{2t}{N}})uv_1\in\Cuuv .$$ In addition, for $P\in\Cuuv$, there are analogous trace polynomials $P_N^{(2)}$ and $P_N^{(4)}$ on $\SU(N)$, and $\Sp(N)$, resp. (see Definition \ref{def:tracepolys} for precise definitions). 

To simplify the statement of our results, we introduce the following notation: for $N\in\N$, let
\begin{equation}
K_N^{(1)}=\SO(N,\R),\hspace{.5cm}K_N^{(2)}=\SU(N),\hspace{.5cm}K_N^{(2')}=\U(N),\hspace{.5cm}K_N^{(4)}=\Sp(N), 
\end{equation}
and set $K_{N,\C}^{(\beta)}=(K_N^{(\beta)})_\C$ and $\Delta_N^{(\beta)}=\Delta_{K_N^{(\beta)}}$ for $\beta=1,2,2',4$. In addition, we write $\rhosNbeta=\rho_s^{K_N^{(\beta)}}$ and $\mustNbeta=\mu_{s,\tau}^{K_{N,\C}^{(\beta)}}$ for the heat kernels on $K_N^{(\beta)}$ and $K_{N,\C}^{(\beta)}$, respectively.

Our first theorem is an intertwining result showing that $\Delta_N^{(\beta)} P_N^{(\beta)}$ can be computed by applying a certain pseudodifferential operator on $\Cuuv$ to $P$ and then considering the corresponding trace polynomial. 

\begin{theorem}[Intertwining formulas for $\Delta_{\SO(N,\R)}$, $\Delta_{\SU(N)}$, and $\Delta_{\Sp(N)}$] \label{thm:intertwine1}
	For $\beta\in\{1,2,4\}$ and $N\in\N$, there are pseudodifferential operators $\mathcal{D}_N^{(\beta)}$ acting on $\Cuuv$ (see Definition \ref{def:Doperators}) such that, for each $P\in\Cuuv$, 
	\begin{equation} \label{eq:intertwine1}
	\Delta_N^{(\beta)}P_N^{(\beta)}=[\mathcal{D}_N^{(\beta)}P]_N. 
	\end{equation}
	
	Moreover, for all $\tau\in\C$ and $P\in\Cuuv$,
	\begin{equation} \label{eq:intertwine1exp}
	e^{\frac{\tau}{2}\Delta_N^{(\beta)}}P_N^{(\beta)}=[e^{\frac{\tau}{2}\mathcal{D}_N^{(\beta)}}P]_N
	\end{equation}
	For each such $\beta$, $\DNbeta$ takes the form 
	\begin{equation}
	\DNbeta=\Lzerobeta +\frac{1}{N}\Lonebeta+\frac{1}{N^2}\Ltwobeta,
	\end{equation}
	where $\Lzerobeta,\Lonebeta$ are first order pseudodifferential operators and $\Ltwobeta$ is a second order differential operator, all independent of $N$. 
\end{theorem}

The exponential $e^{\frac{\tau}{2}\mathcal{D}_N^{(\beta)}}$ is well defined on $\Cuuv$ (see Corollary \ref{cor:DNinvariant}). The proof of Theorem \ref{thm:intertwine1} for $\SO(N,\R)$  and $\Sp(N)$ can be found on p. \pageref{pf:thm:intertwine1} and p. \pageref{pf:thm:intertwine1SpN}, resp. We do not provide a full proof of the $\SU(N)$ version of the theorem, as it is very similar to the $\SO(N,\R)$ and $\Sp(N)$ cases; the proof is discussed on p. \pageref{pf:thm:intertwine1SUN}. 

Our next result shows that we can use $e^{\frac{\tau}{2}\mathcal{D}_N^{(\beta)}}$ to explicitly compute the Segal-Bargmann transform on trace polynomials.

\begin{theorem}\label{thm:SBintertwine} 
	Let $P\in\Cuuv$, $N\in\N$, $s>0$ and $\tau\in\mathbb{D}(s,s)$. For $\beta\in\{1,2,4\}$, the Segal-Bargmann transform on the trace polynomial $P_N^{(\beta)}$ can be computed as
	\begin{equation}
	\mathbf{B}_{s,\tau}^{K_N^{(\beta)}} P_N^{(\beta)}=[e^{\frac{\tau}{2}\mathcal{D}_N^{(\beta)}}P]_N.
	\end{equation}
\end{theorem}
The proof of Theorem \ref{thm:SBintertwine} for $\SO(N,\R)$ and $\Sp(N)$ can be found on p. \pageref{pf:thm:SBintertwine} and p. \pageref{pf:thm:SBintertwineSpN}, resp. The proof of this theorem for the $\SU(N)$ case is similar, and is discussed on p. \pageref{pf:thm:SBintertwineSUN}. 

Our main result concerns the limit of $\mathbf{B}_{s,\tau}^{K_N^{(\beta)}}$ on polynomials as $N\to\infty$. While the range of the boosted Segal-Bargmann transform is the space of trace polynomials, as $N\to\infty$ the image concentrates back onto the smaller subspace of Laurent polynomials. 

\begin{theorem} \label{thm:freeSBlimit} 
	Let $s>0$, $\tau\in\mathbb{D}(s,s)$, and $f\in \C[u,u^{-1}]$. Then there exist unique $g_{s,\tau},h_{s,\tau}\in\C [u,u^{-1}]$ such that
	\begin{align}
		\|\mathbf{B}_{s,\tau}^{K_N^{(\beta)}} f_N-[g_{s,\tau}]_N\|^2_{L^2(K_{N,\C}^{(\beta)},\mustNbeta;\MNC)}&=O\left(\frac{1}{N^2}\right),\label{eq:freeSBlimit}\\
		\|(\mathbf{B}_{s,\tau}^{K_N^{(\beta)}})^{-1}f_N-[h_{s,\tau} ]_N\|^2_{L^2(K_N^{(\beta)},\rhosNbeta;\MNC)}&=O\left(\frac{1}{N^2}\right). \label{eq:freeinverseSBlimit}
	\end{align}
	for $\beta\in\{1,2,4\}$. 
\end{theorem}

\begin{definition}\label{def:freeSB}
	The map $\Gst:\Cuu\to\Cuu$ given by $f\mapsto g_{s,\tau}$ is called the {\bf free Segal-Bargmann transform}, and the map $\Hst:\Cuu\to\Cuu$ given by $f\mapsto h_{s,\tau}$ is called the {\bf free inverse Segal-Bargmann transform}. Explicit formulas for $\Gst$ and $\Hst$ are given in \eqref{eq:Gst} and \eqref{eq:Hst}, respectively. 
\end{definition}

The proof of Theorem \ref{thm:freeSBlimit} for $\SO(N,\R)$ can be found on p. \pageref{pf:thm:freeSBlimit}. The proof of the theorem for the $\SU(N)$ and $\Sp(N)$ cases is entirely analogous to the $\SO(N,\R)$ case, so we do not provide the full details; the proofs for $\SU(N)$ and $\Sp(N)$ are discussed on pp. \pageref{pf:thm:freeSBlimitSUN} and \pageref{pf:thm:freeSBlimitSpN}. 

\begin{remark}
	The fact that the Segal-Bargmann transform on matrices in $\M_N(\C)$ has a meaningful limit $\mathscr{G}^t$ as $N\to\infty$ is due to Biane \cite[Proposition~13]{Biane:SB}. For the special case when $s=t>0$, the free Segal-Bargmann transform $\mathscr{G}_{t,t}$ of this paper is equivalent to $\mathscr{G}^t$, which Biane called the \emph{free Hall transform}. A two-parameter generalization of this transform, called the \emph{free unitary Segal-Bargmann transform}, $\mathscr{G}_{s,t}$, where $s>t/2>0$, was introduced in \cite{DHK:U(N)}. This version of the free Segal-Bargmann transform was also investigated by Ho in \cite{Ho:SB}; in particular, he gave an integral kernel for the large-$N$ limit of the Segal-Bargmann transform over $\U(N)$. In this paper, we study a slight generalization of $\mathscr{G}_{s,t}$ by allowing the time parameter $\tau$ to be complex, incorporating results from \cite{DHK:complextimeSB} on the complex-time Segal-Bargmann transform. When $\tau=t>0$, the free Segal-Bargmann transform $\mathscr{G}_{s,\tau}$ of Definition \ref{def:freeSB} is precisely the same as the operator $\mathscr{G}_{s,t}$ of \cite{DHK:U(N)}. 	
\end{remark}

\begin{remark}
	Many of the results in this paper on the Segal-Bargmann transform for $\SO(N,\R)$, $\SU(N)$, and $\Sp(N)$ correspond to work done by Driver, Hall, and Kemp on the Segal-Bargmann transform on $\U(N)$. The $\U(N)$ analogue of our intertwining formulas, Theorem \ref{thm:intertwine1}, is contained in \cite[Theorem 1.18]{DHK:U(N)}, where it is shown that $\Delta_{\U(N)}$ has the intertwining formula $$\mathcal{D}_N^{(2')}=\mathcal{L}_0+\frac{1}{N}\mathcal{L}_1^{(2')}+\frac{1}{N^2}\mathcal{L}_2^{(2')}$$ for certain pseudodifferential operators $\mathcal{L}_0, \mathcal{L}_1^{(2')}$, and $\mathcal{L}_2^{(2')}$ on the trace polynomial space $\Cuuv$ (cf. Remark \ref{rem:intertwineU(N)}), where the operator $\mathcal{L}_1^{(2')}$ is identically zero (as it is for $\mathcal{L}_1^{(2)}$ in the $\SU(N)$ case). 
	
	The \emph{same} operator $\mathcal{L}_0$ appears in the intertwining formulas for $\SO(N,\R)$, $\U(N)$, $\SU(N)$, and $\Sp(N)$. As $N\to\infty$, it is the operator $\Lzerobeta$ that drives the large-$N$ behavior of the transform; this is the key to our main result on the free Segal-Bargmann transform, Theorem \ref{thm:freeSBlimit}, which states that the large-$N$ limit of the Segal-Bargmann transform on $\SO(N,\R)$, $\SU(N)$, and $\Sp(N)$ is the same operator $\mathscr{G}_{s,\tau}$. Moreover, we emphasize that the rate of convergence in \emph{all} cases is $O(1/N^2)$, as it is in the analogous result for the $\U(N)$ case, \cite[Theorem 1.11]{DHK:U(N)}. This is somewhat surprising, since there is a nontrivial term of order $1/N$ in the intertwining formulas for $\SO(N,\R)$ and $\Sp(N)$ that does not appear in the $\U(N)$ and $\SU(N)$ cases. That we still obtain $O(1/N^2)$ convergence for $\SO(N,\R)$ and $\Sp(N)$ is due to the fact that in the scalar version of their intertwining formulas (cf. Theorems \ref{thm:intertwiningII} and \ref{thm:intertwiningIIpolySpN}), the term of order $1/N$ is a first order differential operator. As a result, the $1/N$ terms are ``washed away'' in the covariance estimates (cf. the proof of Theorem \ref{thm:tracelimit} in the $\SO(N,\R)$ case for a more precise explanation; the $\Sp(N)$ case is similar). 
\end{remark}


\section{Background} \label{sec:background}

In this section, we expand on the background required to prove our main results. In particular, we provide a brief overview of the heat kernel results used in constructing the Segal-Bargmann transform, and conclude by setting some notation that will be used for the remainder of the paper. 

\subsection{Heat kernels on Lie groups}

\begin{definition}
	Let $K$ be a Lie group with Lie algebra $\mathfrak{k}$. An inner product $\inpro{\cdot ,\cdot}_{\mathfrak{k}}$ on $\mathfrak{k}$ is {\bf $\Ad$-invariant} if, for all $X_1,X_2\in\mathfrak{k}$ and all $k\in K$, 
	\begin{equation*}
	\inpro{\Ad_k X_1,\Ad_k X_2}_{\mathfrak{k}}=\inpro{X_1,X_2}_{\mathfrak{k}}.
	\end{equation*}
	A group whose Lie algebra possesses an $\Ad$-invariant inner product is called {\bf compact type}. 
\end{definition}

The Lie groups $\SO(N,\R)$, $\SU(N)$, and $\Sp(N)$ studied in this paper are compact type (and are, in fact, compact). 

Compact type Lie groups have the following property. 

\begin{prop}[{[\cite{Milnor:curvatures}, Lemma 7.5]}]
	If $K$ is a compact type Lie group with a fixed $\Ad$-invariant inner product, the $K$ is isometrically isomorphic to a direct product group, i.e. $K\cong K_0\times \R^d$ for some compact Lie group $K_0$ and some nonnegative integer $d$.
\end{prop}

Let $K$ be a compact type Lie group with Lie algebra $\LieK$, and fix an $\Ad$-invariant inner product on $\mathfrak{k}$. Choose a right Haar measure $\lambda$ on $K$; we will write $dx$ for $\lambda(dx)$ and $L^2(K)$ for $L^2(K,\lambda)$. If $\beta_\LieK$ is an orthonormal basis for $\LieK$, we define the Laplace operator $\Delta_K$ on $K$ by 
\begin{equation} \label{eq:Laplace}
\Delta_K=\sum_{X\in\beta_\LieK}\wX ^2,
\end{equation}
where for any $X\in\LieK$, $\wX$ is the left-invariant vector field given by
\begin{equation}
(\wX f)(k)=\frac{d}{dt}f(ke^{tX})\bigg|_{t=0}
\end{equation}
for any smooth real or complex-valued function $f$ on $K$. The operator is independent of orthonormal basis chosen. For a detailed overview of left-invariant Laplacian operators, see \cite{Milnor:curvatures}; see also \cite{D:SBH, DGS:holo, DHK:complextimeSB} for further background on heat kernels on compact type Lie groups.  

The operator $\Delta_K$ is left-invariant for any Lie group $K$. When $K$ is compact type, $\Delta_K$ is also bi-invariant. In addition, it is well known that $\Delta_K$ on $K$ is elliptic and essentially self-adjoint on $L^2(K)$ with respect to any right invariant Haar measure (see \cite[Section 3.2]{DHK:complextimeSB}). As a result, there exists an associated heat kernel $\rho_t^K\in C^\infty (K,(0,\infty))$ satisfying \eqref{eq:heatkernel}. 

\begin{definition}
	Let $s>0$ and $\tau = t+i\theta\in\mathbb{D}(s,s)$. We define a second order left-invariant differential operator $\mathcal{A}^{K_\C}_{s,\tau}$ on $K_\C$ by
	\begin{equation}
	\mathcal{A}^{K_\C}_{s,\tau}=\sum_{j=1}^d\left[\left(s-\frac{t}{2}\right)\wX^2_j+\frac{t}{2}\wY^2_j-\theta\wX_j\wY_j\right],
	\end{equation}
	where $\{X_j\}_{j=1}^d$ is any orthonormal basis of $\LieK$, and $Y_j=JX_j$ where $J$ is the operation of multiplication of $i$ on $\LieK_\C=\Lie(K_\C)$. 
\end{definition}

The operator $\mathcal{A}^{K_\C}_{s,\tau}|_{C_c^\infty (K_\C)}$ is essentially self-adjoint on $L^2(K_\C)$ with respect to any right Haar measure (see \cite[Section 3.2]{DHK:complextimeSB}). There exists a corresponding heat kernel density $\mu_{s,\tau}^{K_\C}\in C^{\infty}(K_\C,(0,\infty))$ such that
\begin{equation} \label{eq:heatKC}
\left(e^{\frac{1}{2}\mathcal{A}^{K_\C}_{s,\tau}}f\right)(z)=\int_{K_\C}\mu_{s,\tau}^{K_\C}(w)f(zw)\, dz\hspace{.5cm}\mbox{for all }f\in L^2(K_\C).
\end{equation}

\subsection{Notation and definitions}

The classical Segal-Bargmann transform is an isometric isomorphism from $L^2(\R^d)$ onto $\mathcal{H}L^2(\C^d)$. In order to extend the Segal-Bargmann transform to Lie groups, we note that every connected Lie group has a \emph{complexification} defined by a certain representation-theoretic universal property, mimicking the relationship between $\R^d$ and $\C^d$ (see \cite[Section 2]{DHK:complextimeSB}). For the present purposes, it is enough to know the concrete complexifications of $\SO(N,\R)$, $\SU(N)$, and $\Sp(N)$, which we describe below. 

Let $\SO(N,\R)$ denote the {\bf special orthogonal group}, defined by  $\SO(N,\R)=\{A\in\M_N(\R)\,:\,AA^\intercal=I_N,\,\det(A)=1\}.$ Its complexification, $\SO(N,\R)_\C$, is given by $\SO(N,\C)=\{A\in\MNC\,:\,AA^\intercal=I_N,\, \det(A)=1\}.$ The Lie group $\SO(N,\R)$ has Lie algebra $\so(N,\R)=\{X\in\M_N(\R)\,:\, X^\intercal=-X\},$ while $\SO(N,\C)$ has Lie algebra $\so(N,\C)=\{X\in\MNC\,:\, X^\intercal=-X\}.$

The {\bf unitary group}, $\U(N)$, is defined by $\U(N)=\{A\in\MNC\, :\,AA^*=I_N\}$. The Lie algebra of $\U(N)$ is $\uu(N)=\{X\in\MNC\, :\, X^*=-X\}$. The {\bf special unitary group} $\SU(N)$ is the subgroup of $\U(N)$ consisting of matrices with determinant $1$. The complexification of $\SU(N)$ is the special linear group $\SL(N,\C)=\{A\in\GL_N(\C)\, :\, \det(A)=1\}$. The Lie algebra of $\SU(N)$ is $\mathfrak{su}(N)=\{X\in\MNC\, :\, X^*=-X,\Tr(X)=0\}$, and the Lie algebra of $\SL(N,\C)$ is $\mathfrak{sl}(N,\C)=\{X\in\MNC\, :\, \Tr(X)=0\}$. 

There are two standard realizations of the compact symplectic group $\Sp(N)$, first as a group of $N\times N$  matrices over the quaternions, and second as a group of $2N\times 2N$ matrices over $\C$. It is more convenient for us to work with the latter realization, which we introduce here; we address the former realization in Section \ref{sec:SpNquaternion}. Our choice of realization means that we will often have to add a normalization factor of $\frac{1}{2}$ when working with $\Sp(N)\subseteq\M_{2N}(\C)$. 

Set $$\Omega_0=\left[
\begin{array}{rr}
0 & 1 \\
-1 & 0 \\
\end{array}
\right]\in\M_2(\C).$$
For $N\in\N$, we define $\Omega=\Omega_N:=\diag(\underbrace{\Omega_0,\Omega_0,\hdots ,\Omega_0}_{N\text{ times}})\in \M_{2N}(\C).$
We define the {\bf complex symplectic group} by $\Sp(N,\C)=\{A\in\M_{2N}(\C)\, :\, A^\intercal \Omega A=\Omega\}.$ The Lie algebra of $\Sp(N,\C)$ is $
\symp(N,\C)=\{X\in M_{2N}(\C)\, :\, \Omega X^\intercal\Omega=X\}.$ The {\bf compact symplectic group} $\Sp(N)$ consists of the elements of $\Sp(N,\C)$ which are also unitary, i.e. 
\begin{equation} \label{def:SpN}
	\Sp(N)=\Sp(N,\C)\cap \U(2N)=\{A\in\M_{2N}(\C)\, :\, A^\intercal \Omega A=\Omega, A^*A=I_{2N}\}.
\end{equation}
The Lie algebra of $\Sp(N)$ is $
\symp(N)=\{X\in M_{2N}(\C)\, :\, \Omega X^\intercal\Omega=X, X^*=-X\}.$
The complexification of $\Sp(N)$ is $\Sp(N,\C)$. 

The groups $\SO(N,\R)$, $\SU(N)$, and $\Sp(N)$ can be thought of as (special) unitary groups over $\R$, $\C$, and $\HH$, resp. For notational convenience, throughout the paper we associate the parameters $\beta=1,2,2',4$ with $\SO(N,\R)$, $\SU(N)$, $\U(N)$, and $\Sp(N)$, resp., where $\beta$ refers to the dimension of $\R$, $\C$, and $\HH$ as associative algebras over $\R$.

In order for the large-$N$ limit of the Segal-Bargmann transform to converge in a meaningful way, we require the following normalizations of the trace and Hilbert-Schmidt inner products.
\begin{notation}
	For $A\in\MNC$, define
	\begin{align}
	\tr(A)=\frac{1}{N}\Tr(A),\\
	\wtr (A)=\frac{1}{2N}\Tr(A),
	\end{align}
	where $\Tr$ denotes the usual trace. We use $\tr$ when applying the Segal-Bargmann transform on $\SO(N,\R)$ and $\SU(N)$, and $\wtr$ when applying the Segal-Bargmann transform on $\Sp(N)$. 
	
	We also define inner products on $\so(N,\R)$, $\mathfrak{su}(N)$, and $\symp(N)$ by
	\begin{align}
	\inpro{X,Y}_{\so(N,\R)}&=\frac{N}{2}\Tr(XY^*)=\frac{N^2}{2}\tr(XY^*),\hspace{.5cm}X,Y\in\so(N,\R),\hspace{.5cm}\\
	\inpro{X,Y}_{\su(N)}&=N\Tr(XY^*)=N^2\tr(XY^*),\hspace{.5cm}X,Y\in\su(N),\label{eq:inproSUN}\\
	\inpro{X,Y}_{\symp(N)}&=N\Tr(XY^*)=2N^2\wtr(XY^*),\hspace{.5cm}X,Y\in\symp(N). \label{eq:inproSpNcomplex}
	\end{align}
\end{notation}


\section{The Segal-Bargmann transform on $\SO(N,\R)$} 

In this section and the next, we prove Theorem \ref{thm:intertwine1}, Theorem \ref{thm:SBintertwine}, and Theorem \ref{thm:freeSBlimit} for the Segal-Bargmann transform on $\SO(N,\R)$. We begin with a set of ``magic formulas'' which are the key ingredient to proving the $\SO(N,\R)$ version of the intertwining formula for $\Delta_{\SO(N,\R)}$ of Theorem \ref{thm:intertwine1}. These formulas, which give simple expressions for certain quadratic matrix sums, are the analogues of the ``magic formulas'' of \cite[Proposition 3.1]{DHK:U(N)} for $\U(N)$. They allow us to compute the Segal-Bargmann transform for trace polynomials on $\SO(N,\R)$ (\ref{thm:SBintertwine} for $\SO(N,\R)$). We then prove a second intertwining formula for the operator $\mathcal{A}_{s,\tau}^{\SO(N,\C)}$ which is used in proving the limit theorems of Section \ref{sec:limittheoremsSONR}.

\subsection{Magic formulas and derivative formulas}

\begin{lemma}[Elementary matrix identities]\label{lem:elementary}
	Let $E_{i,j}\in \MNC$ denote the $N\times N$ matrix with a $1$ in the $(i,j)$th entry and zeros elsewhere. For $1\leq i,j,k,\ell\leq N$ and $A=(A_{i,j})\in\MNC$, we have
	\begin{align}
	E_{i,j}E_{k,\ell}&=\delta_{j,k}E_{i,\ell} \label{lem:elem1}\\
	E_{i,j}AE_{i,j}&=A_{j,i}E_{i,j}\label{lem:elem2} \\
	E_{i,j}AE_{j,i}&=A_{j,j}E_{i,i} \label{lem:elem3}\\
	\Tr(AE_{i,j})&=A_{j,i} \label{lem:elem4}
	\end{align}
\end{lemma}

\begin{proof}
	For $1\leq a,b,\leq N$,  $$[E_{i,j}E_{k,\ell}]_{a,b}=\sum_{h=1}^N [E_{i,j}]_{a,h}[E_{k,\ell}]_{h,b}.$$ Note that $[E_{i,j}]_{a,h}[E_{k,\ell}]_{h,b}$ equals $1$ precisely when $a=i$, $b=\ell$, and $j=k$, and $h=j=k$, and equals $0$ otherwise, so $[E_{i,j}E_{k,\ell}]_{a,b}=\delta_{j,k}\delta_{a,i}\delta_{b,\ell}$, which proves (\ref{lem:elem1}).
	
	We now make the following observations: for any $N\times N$ matrix $A$, $AE_{i,j}$ is the matrix which is all zeros except for the $j$th column, which is equal to the $i$th column of $A$. Hence the only (possibly) nonzero diagonal entry is the $(j,j)$th entry, which is $A_{j,i}$. This proves (\ref{lem:elem4}). Similarly, $E_{i,j}A$ is the matrix which is all zeros except for the $i$ row, which is the $j$th row of $A$. Putting these observations together yields (\ref{lem:elem2}) and (\ref{lem:elem3}). 
\end{proof}

\begin{prop}[Magic formulas for $\SO(N,\R)$] \label{prop:magic}
	For any $A,B\in M_N(\C)$, 
	\begin{align}
	\sum_{X\in\beta_{\so(N,\R)}}X^2&=-\frac{N-1}{N}I_N=-I_N+\frac{1}{N}I_N\label{eq:magic1}\\
	\sum_{X\in\beta_{\so(N,\R)}}XAX&=\frac{1}{N}A^\intercal-\tr(A)I_N\label{eq:magic2}\\
	\sum_{X\in\beta_{\so(N,\R)}}\tr (XA)X&=\frac{1}{N^2}(A^\intercal -A)\label{eq:magic3}\\
	\sum_{X\in\beta_{\so(N,\R)}}\tr(XA)\tr(XB)&=\frac{1}{N^2}(\tr(A^\intercal B)-\tr (AB))\label{eq:magic4}
	\end{align}
\end{prop}

\begin{proof}
	 As shown in Theorem 3.3 of \cite{DHK:U(N)}, the quantity $\sum_{X\in\beta_{\so(N,\R)}}XAX$ is independent of the choice of orthonormal basis. Hence to see (\ref{eq:magic2}), we may set $\beta_{\so(N,\R)}$ to be the orthonormal basis $$\beta_{\so(N,\R)}=\left\{\frac{1}{\sqrt{N}}(E_{i,j}-E_{j,i})\right\}_{1\leq i<j\leq N}.$$ By Lemma \ref{lem:elementary},
	\begin{align*}
	\sum_{i<j}(E_{i,j}-E_{j,i})A(E_{i,j}-E_{j,i})&=\sum_{i<j}[E_{i,j}AE_{i,j}-E_{i,j}AE_{j,i}-E_{j,i}AE_{i,j}+E_{j,i}AE_{j,i}]\\
	&=\sum_{i\neq j}(E_{i,j}AE_{i,j}-E_{i,j}AE_{j,i})\\
	&=\sum_{i\neq j}A_{j,i}E_{i,j}-\sum_{i\neq j}A_{j,j}E_{i,i}\\
	&=\sum_{i\neq j}A_{j,i}E_{i,j}+\sum_{i=1}^NA_{i,i}E_{i,i}-\sum_{i=1}^NA_{i,i}E_{i,i}-\sum_{i\neq j}A_{j,j}E_{i,i}\\
	&=A^\intercal-\Tr(A)I_N,
	\end{align*}
	and dividing by $N$ proves (\ref{eq:magic2}). Equation (\ref{eq:magic1}) then follows from (\ref{eq:magic2}) by setting $A=I_N$.
	
	Similarly, we compute
	\begin{align*}
	\sum_{i<j}\Tr(A(E_{i,j}-E_{j,i}))(E_{i,j}-E_{j,i})&=\sum_{i<j}[\Tr(AE_{i,j})E_{i,j} - \Tr(AE_{i,j})E_{j,i}\\
	&\hspace{1.5cm}-\Tr(AE_{j,i})E_{i,j}+\Tr(AE_{j,i})E_{j,i}]\\
	&=\sum_{i\neq j} [\Tr(AE_{i,j})E_{i,j}-\Tr(AE_{j,i})E_{i,j}]\\
	&=\sum_{i\neq j}A_{j,i}E_{i,j}-\sum_{i\neq j}A_{i,j}E_{i,j}\\
	&=\sum_{i\neq j}A_{j,i}E_{i,j}+\sum_{i=1}^NA_{i,i}E_{i,i}\\
	&\hspace{1.5cm}-\sum_{i=1}^NA_{i,i}E_{i,i}-\sum_{i\neq j}A_{i,j}E_{i,j}\\
	&=A^\intercal-A,
	\end{align*}
	which is equivalent to (\ref{eq:magic3}). Equation (\ref{eq:magic4}) now follows from (\ref{eq:magic3}) by multiplying by $B$ and taking $\tr$.
\end{proof}

\begin{remark}
	Our proof of the magic formulas for $\SO(N,\R)$ is primarily computational, relying on the elementary matrix identities of Lemma \ref{lem:elementary}. These results can also be obtained more abstractly. For $\beta\in\{1,2,2',4\}$, let $\LieK_N^{(\beta)}$ denote $\so(N,\R)$, $\mathfrak{su}(N)$, $\mathfrak{u}(N)$, and $\symp(N)$, resp., and let $\beta_{\LieK_N^{(\beta)}}$ be an orthonormal basis for $\LieK_N^{(\beta)}$. In \cite[Lemma 1.2.1]{Levy:master}, L\'{e}vy provides an explicit decomposition of the {\bf Casimir element} $C_{\LieK_N^{(\beta)}}$, defined as the tensor $$C_{\LieK_N^{(\beta)}}=\sum_{X\in\beta_{\LieK_N^{(\beta)}}}X\otimes X.$$
	Let
	\begin{equation*}
	T=\sum_{a,b=1}^N E_{a,b}\otimes E_{b,a}\mbox{\hspace{.5cm}and\hspace{.5cm}} P=\sum_{a,b=1}^N E_{a,b}\otimes E_{a,b}.
	\end{equation*}
	For $\mathbb{K}\in\{\R,\C,\HH\}$, set $\mathbf{I}(\mathbb{K})=\{\idH,\ii,\jj,\kk\}\cap\mathbb{K}$ and define two elements $\Real^{\mathbb{K}}$ and $\Co^\mathbb{K}$ of $\KK\otimes_\R\KK$ by
	\begin{equation*}
	\Real^\KK=\sum_{\gamma\in\mathbf{I}(\KK)}\gamma\otimes\gamma^{-1}\mbox{\hspace{.5cm}and\hspace{.5cm}}\Co^\KK=\sum_{\gamma\in\mathbf{I}(\KK)}\gamma\otimes\gamma.
	\end{equation*}
	Then for $\beta\in\{1,2'4\}$ (where $2'$ corresponds to the value $2$), the Casimir element $C_{\LieK_N^{(\beta)}}$ is given by
	\begin{equation*}
	C_{\LieK_N^{(\beta)}}=\frac{1}{\beta N}\left(-T\otimes\Real^\KK + P\otimes\Co^\KK\right),
	\end{equation*}
	and for $\beta=2$, $C_{\su(N)}=C_{\mathfrak{u}(N)}-\frac{1}{N^2}iI_N\otimes iI_N$. This decomposition can be used to compute any expression of the form $\sum_{X\in\beta_{\LieK_N^{(\beta)}}}B(X,X)$, where $B$ is an $\R$-bilinear map. In particular, the magic formulas for $\SO(N,\R)$, $\SU(N)$, and $\Sp(N)$ (cf. Propositions \ref{prop:magic}, \ref{prop:magicSUN}, and \ref{prop:magicSpN}) can be obtained in this way.  
\end{remark}

\begin{prop}[Derivative formulas for $\SO(N,\R)$]
	Let $X\in\beta_{\so(N,\R)}$. The following identities hold on $\SO(N,\R)$ and $\SO(N,\C)$: 
	\begin{gather}
	\widetilde{X} A^m=\sum_{j=1}^mA^jXA^{m-j},\hspace{.5cm}m\geq 0\label{eq:XAmpos}\\
	\widetilde{X} A^m=-\sum_{j=m+1}^0A^jXA^{m-j},\hspace{.5cm}m<0\label{eq:XAmneg}\\
	\widetilde{X}\tr(A^m)=m\cdot\tr (XA^m),\hspace{.5cm}m\in\Z\label{eq:XtrAm}\\
	\Delta_{\SO(N,\R)}A^m=2\Ind_{m\geq 2}\left[\frac{1}{N}\sum_{j=1}^{m-1}(m-j)A^{m-2j}-\sum_{j=1}^{m-1}(m-j)\tr(A^j)A^{m-j}\right] \nonumber\\
	\hspace{2cm}-\frac{m(N-1)}{N}A^m,\hspace{.5cm}m\geq 0\label{eq:DAmpos}\\
	\Delta_{\SO(N,\R)}A^m=2\Ind_{m\leq -2}\left[\frac{1}{N}\sum_{j=m+1}^{-1}(-m+j)A^{m-2j}-\sum_{j=m+1}^{-1}(-m+j)\tr(A^{j})A^{m-j}\right]\nonumber\\
	\hspace{1.5cm}+\frac{m(N-1)}{N}A^m,\hspace{.5cm}m<0\label{eq:DAmneg}\\
	\sum_{X\in\beta_{\so(N,\R)}}\wX\tr(A^m)\cdot\wX A^p=\frac{mp}{N^2}(A^{p-m}-A^{p+m}),\hspace{.5cm}m,p\in\Z.\label{eq:XtrAmXAp}
	\end{gather}	
\end{prop}

We emphasize that these derivative formulas hold true only in the case on $\SO(N,\R)$ and $\SO(N,\C)$, in which case $A^{-1}=A^{\intercal}$. 

\begin{proof}
	The proof of (\ref{eq:XAmpos}), (\ref{eq:XAmneg}), and (\ref{eq:XtrAm}) for the $\U(N)$ case are contained in \cite{DHK:U(N)}; the proof of these equations for the $\SO(N,\R)$ case is identical. Next, by (\ref{eq:XAmpos}), we have, for $m\geq 0$,
	\begin{align}
	\widetilde{X}^2A^m=2\Ind_{m\geq 2}\sum_{\substack{j,k\neq 0\\j+k+\ell =m}}A^jXA^kXA^\ell +\sum_{j=1}^mA^jX^2A^{m-j}.\label{eq:XXAm}
	\end{align}
	Summing over all $X\in\beta_{\SO(N,\R)}$ and using the magic formulas (\ref{eq:magic1}) and (\ref{eq:magic2}), we have
	\begin{align*}
	\Delta_{\SO(N,\R)}A^m&=2\Ind_{m\geq 2}\sum_{X\in\bSON}\sum_{\substack{k,j\neq 0\\k+j+\ell =m}}A^kXA^jXA^\ell +\sum_{X\in\bSON}\sum_{j=1}^mA^jX^2A^{m-j}\\
	&=2\Ind_{m\geq 2}\sum_{\substack{k,j\neq 0\\k+j+\ell =m}}\left[\frac{1}{N}A^k(A^j)^\intercal A^\ell-\tr(A^j)A^{k+\ell}\right]-\frac{m(N-1)}{N}A^m\\
	&=2\Ind_{m\geq 2}\left[\frac{1}{N}\sum_{j=1}^{m-1}(m-j)A^{m-2j}-\sum_{j=1}^{m-1}(m-j)\tr(A^j)A^{m-j}\right]\\
	&\hspace{.5cm}-\frac{m(N-1)}{N}A^m,
	\end{align*}
	which proves (\ref{eq:DAmpos}). 
	
	Similarly, if $m<0$, we use (\ref{eq:XAmneg}) to get
	\begin{align*}
	\widetilde{X}^2A^m=2\Ind_{m\leq -2}\sum_{\substack{k,\ell<0,j\leq 0\\j+k+\ell =m}}A^jXA^kXA^\ell +\sum_{j=m+1}^0A^jX^2A^{m-j}. \label{eq:XXAmneg}
	\end{align*}
	Summing over all $X\in\beta_{\SO(N,\R)}$ and using the magic formulas (\ref{eq:magic1}) and (\ref{eq:magic2}), we have
	\begin{align*}
	\Delta_{\SO(N,\R)}A^m&=2\Ind_{m\geq 2}\sum_{X\in\bSON}\sum_{\substack{k,\ell<0,j\leq 0\\j+k+\ell =m}}A^kXA^jXA^\ell +\sum_{X\in\bSON}\sum_{j=m+1}^0A^jX^2A^{m-j}\\
	&=2\Ind_{m\geq 2}\sum_{\substack{k,\ell<0,j\leq 0\\j+k+\ell =m}}\left[\frac{1}{N}A^k(A^j)^\intercal A^\ell-\tr(A^j)A^{k+\ell}\right]+\frac{m(N-1)}{N}A^m\\
	&=2\Ind_{m\geq 2}\left[\frac{1}{N}\sum_{j=m+1}^{-1}(-m-j)A^{m-2j}-\sum_{j=m+1}^{-1}(-m-j)\tr(A^j)A^{m-j}\right]\\
	&\hspace{.5cm}+\frac{m(N-1)}{N}A^m,
	\end{align*}
	which proves (\ref{eq:DAmneg}). 
	
	For (\ref{eq:XtrAmXAp}), we first assume $m,p\geq 0$. Using equation (\ref{eq:magic3}) and the fact that $A^\intercal=A^{-1}$ for $A\in \SO(N,\R)$, 
	\begin{align*}
	\sum_{X\in\beta_{\so(N,\R)}}\widetilde{X}\tr(A^m)\cdot\widetilde{X}A^p&=\sum_{X\in\beta_{\so(N,\R)}}\left[\tr\left(\sum_{j=1}^m A^jXA^{m-j}\right)\left(\sum_{k=1}^{p}A^k XA^{p-k}\right)\right]\\
	&=\sum_{k=1}^{p}A^k\sum_{X\in\beta_{\so(N,\R)}}\left(\sum_{j=1}^m\tr(XA^m)\right) XA^{p-k}\\
	&=\sum_{k=1}^{p}mA^k\left(\sum_{X\in\beta_{\so(N,\R)}}\tr(XA^m)X\right)A^{p-k}\\
	&=\frac{mp}{N^2}(A^{p-m}-A^{p+m}).
	\end{align*}
	
	For $m<0$ and $p\geq 0$, we have 
	\begin{align*}
	\sum_{X\in\beta_{\so(N,\R)}}\widetilde{X}\tr(A^m)\cdot\widetilde{X}A^p&=\sum_{X\in\beta_{\so(N,\R)}}\left[\tr\left(-\sum_{j=m+1}^0 A^jXA^{m-j}\right)\left(\sum_{k=1}^{p}A^k XA^{p-k}\right)\right]\\
	&=-\sum_{k=1}^{p}A^k\sum_{X\in\beta_{\so(N,\R)}}\left(\sum_{j=m+1}^0\tr(XA^m)\right) XA^{p-k}\\
	&=-\sum_{k=1}^{p}(-m)A^k\left(\sum_{X\in\beta_{\so(N,\R)}}\tr(XA^m)X\right)A^{p-k}\\
	&=\frac{mp}{N^2}(A^{p-m}-A^{p+m}).
	\end{align*}
	
	For $m\geq 0$ and $p<0$, we have
	\begin{align*}
	\sum_{X\in\beta_{\so(N,\R)}}\widetilde{X}\tr(A^m)\cdot\widetilde{X}A^p&=\sum_{X\in\beta_{\so(N,\R)}}\left[\tr\left(\sum_{j=1}^m A^jXA^{m-j}\right)\left(-\sum_{k=p+1}^{0}A^k XA^{p-k}\right)\right]\\
	&=-\sum_{k=p+1}^{0}A^k\sum_{X\in\beta_{\so(N,\R)}}\left(\sum_{j=1}^m\tr(XA^m)\right) XA^{p-k}\\
	&=-\sum_{k=p+1}^{0}mA^k\left(\sum_{X\in\beta_{\so(N,\R)}}\tr(XA^m)X\right)A^{p-k}\\
	&=\frac{mp}{N^2}(A^{p-m}-A^{p+m}).
	\end{align*}
	
	Finally, for $m,p<0$, we have 
	\begin{align*}
	\sum_{X\in\beta_{\so(N,\R)}}\widetilde{X}&\tr(A^m)\cdot\widetilde{X}A^p\\
	&=\sum_{X\in\beta_{\so(N,\R)}}\left[-\tr\left(\sum_{j=m+1}^0 A^jXA^{m-j}\right)\left(-\sum_{k=p+1}^{0}A^k XA^{p-k}\right)\right]\\
	&=\sum_{k=p+1}^{0}A^k\sum_{X\in\beta_{\so(N,\R)}}\left(\sum_{j=m+1}^0\tr(XA^m)\right) XA^{p-k}\\
	&=\sum_{k=p+1}^{0}(-m)A^k\left(\sum_{X\in\beta_{\so(N,\R)}}\tr(XA^m)X\right)A^{p-k}\\
	&=\frac{(-m)(-p)}{N^2}(A^{p-m}-A^{p+m}).
	\end{align*}
	This proves (\ref{eq:XtrAmXAp}).
\end{proof}

\subsection{Intertwining formulas for $\Delta_{\SO(N,\R)}$} \label{sec:intertwineSON}

In this section, we describe the trace polynomial functional calculus that will be necessary for the proof of Theorem \ref{thm:intertwine1}, which contains the intertwining formulas for $\Delta_{\SO(N,\R)}$, $\Delta_{\SU(N)}$, and $\Delta_{\Sp(N)}$. The space of trace polynomials was first introduced in \cite[Definition 1.7]{DHK:U(N)}, where it was used to prove analogous intertwining results for $\Delta_{\U(N)}$. Using this framework, we then prove Theorem \ref{thm:intertwine1} for the $\SO(N,\R)$ case. 

The computations in this section are related to, but not quite the same as, those used to establish the intertwining formulas for $\Delta_{\U(N)}$ in \cite[Sections 3-4]{DHK:U(N)}.

\begin{definition} \label{def:tracepolys}
	Let $\C[u,u^{-1}]$ denote the algebra of {\bf Laurent polynomials} in a single variable $u$:
	\begin{equation}
	\C[u,u^{-1}]=\left\{\sum_{k\in\Z} a_ku^k\, :\, a_k\in\C,a_k=0\mbox{ for all but finitely many }k\right\},
	\end{equation}
	with the usual polynomial multiplication. Let $\C[u]$ and $\C[u^{-1}]$ denote the subalgebras consisting of polynomials in $u$ and $u^{-1}$, respectively. 
	
	Let $\C[\mathbf{v}]$ denote the algebra of polynomials in countably many commuting variables $\mathbf{v}=\{v_{\pm 1},v_{\pm 2},\hdots \}$, and let $\C[u,u^{-1};\mathbf{v}]$ denote the algebra of polynomials in the commuting variables $\{u,u^{-1},v_{\pm 1},v_{\pm 2},\hdots\}$. 
	
	We now define the {\bf trace polynomial functional calculus}: for $P\in\Cuuv$, we have
	\begin{align*}
	P_N^{(1)}(A)&=[P^{(1)}]_N(A):=P(u;\mathbf{v})|_{u=A,v_k=\tr(A^k),k\neq 0},\hspace{.5cm}\mbox{$A\in\SO(N,\R)$ or $\SO(N,\C)$},\\
	P_N^{(2)}(A)&=[P^{(2)}]_N(A):=P(u;\mathbf{v})|_{u=A,v_k=\tr(A^k),k\neq 0},\hspace{.5cm}\mbox{$A\in\SU(N)$ or $\SL(N,\C)$},\\
	P_N^{(4)}(A)&=[P^{(4)}]_N(A):=P(u;\mathbf{v})|_{u=A,v_k=\wtr(A^k),k\neq 0},\hspace{.5cm}\mbox{$A\in\Sp(N)$ or $\Sp(N,\C)$}.
	\end{align*}
	Functions of the form $P_N^{(\beta)}$, where $P\in\Cuuv$, are called {\bf trace polynomials}. We will often suppress the superscripts for $P_N^{(\beta)}$ and write $P_N$ when it is clear from context which version of the trace polynomial functional calculus is being used. 
\end{definition}
As an illustrative example, the trace  polynomial $P_N^{(1)}=P_N$ on $\SO(N,\R)$ associated with $P(u;\mathbf{v})=v_1v_{-3}^6u^2-3v_5^2$ is $$P_N(A)=\tr(A)\tr(A^{-3})^6A^2-3\tr(A^5)^2I_N.$$

We now introduce the pseudodifferential operators on $\Cuuv$ that appear in our intertwining formulas. 

\begin{definition}
	Let $\mathcal{R}^\pm$ denote the \emph{positive} and \emph{negative projection} operators $$\mathcal{R}^+:\C[u,u^{-1};\mathbf{v}]\to\C[u;\mathbf{v}]\hspace{.5cm}\mbox{and}\hspace{.5cm}\mathcal{R}^-:\C[u,u^{-1};\mathbf{v}]\to\C[u^{-1};\mathbf{v}]$$ defined by $$\mathcal{R}^+\left(\sum_{k=-\infty}^\infty u^kq_k(\mathbf{v})\right)=\sum_{k=0}^\infty u^kq_k(\mathbf{v}),\hspace{.5cm}\mathcal{R}^-\left(\sum_{k=-\infty}^\infty u^kq_k(\mathbf{v})\right)=\sum_{k=0-\infty}^{-1} u^kq_k(\mathbf{v}).$$
	
	In addition, for any $k\in\Z$, let $\mathcal{M}_{u^k}$ denote the multiplication operator defined by $\mathcal{M}_{u^k}P(u;\mathbf{v})=u^kP(u;\mathbf{v})$. 
\end{definition}

\begin{definition}\label{def:Doperators}
	Define the following operators on $\C[u,u^{-1};\mathbf{v}]$:
	\begin{align}
	\mathcal{N}&=\mathcal{N}_0+\mathcal{N}_1=\sum_{|k|\geq 1}|k|v_k\frac{\partial}{\partial v_k}+u\frac{\partial}{\partial u}(\mathcal{R}^+-\mathcal{R}^-),\\
	\mathcal{Y}_1&=\mathcal{Y}_1^+-\mathcal{Y}_1^-=\sum_{k=1}^\infty v_ku\mathcal{R}^+\frac{\partial}{\partial u}\mathcal{M}_{u^{-k}}\mathcal{R}^+-\sum_{k=-\infty}^{-1} v_ku\mathcal{R}^-\frac{\partial}{\partial u}\mathcal{M}_{u^{-k}}\mathcal{R}^-,\\
	\mathcal{Y}_2&=\mathcal{Y}_2^+-\mathcal{Y}_2^-=\sum_{k=2}^\infty\left(\sum_{j=1}^{k-1}jv_jv_{k-j}\right)\frac{\partial}{\partial v_k}-\sum_{k=-\infty}^{-2}\left(\sum_{j=k+1}^{-1}jv_jv_{k-j}\right)\frac{\partial}{\partial v_k},\\
	\mathcal{Z}_1&=\mathcal{Z}_1^+-\mathcal{Z}_1^-=\sum_{k=2}^\infty\left(\sum_{j=1}^{k-1}(k-j)v_{k-2j}\right)\frac{\partial}{\partial v_k}-\sum_{k=-\infty}^{-2}\left(\sum_{j=k+1}^{-1}(k-j)v_{k-2j}\right)\frac{\partial}{\partial v_k},\\
	\mathcal{Z}_2&=\mathcal{Z}_2^+-\mathcal{Z}_2^-=\sum_{k=1}^\infty u^{-k+1}\mathcal{R}^+\frac{\partial}{\partial u}\mathcal{M}_{u^{-k}}\mathcal{R}^+-\sum_{k=-\infty}^{-1} u^{-k+1}\mathcal{R}^-\frac{\partial}{\partial u}\mathcal{M}_{u^{-k}}\mathcal{R}^-,\\
	\mathcal{K}_1&=\mathcal{K}_1^+-\mathcal{K}_1^-=\sum_{|k|\geq 1}ku^{-k+1}\frac{\partial^2}{\partial v_k\partial u}-\sum_{|k|\geq 1}ku^{k+1}\frac{\partial^2}{\partial v_k\partial u},\\
	\mathcal{K}_2&=\mathcal{K}_2^+-\mathcal{K}_2^-=\sum_{|j|,|k|\geq 1}jkv_{k-j}\frac{\partial^2}{\partial v_j\partial v_k}-\sum_{|j|,|k|\geq 1}jkv_{k+j}\frac{\partial^2}{\partial v_j\partial v_k},\\
	\mathcal{J}&=2\sum_{|k|\geq 1}ku\frac{\partial}{\partial u\partial v_k}+\sum_{|j|,|k|\geq 1}jkv_jv_k\frac{\partial^2}{\partial v_j\partial v_k}+\sum_{|k|\geq 1}k^2v_k\frac{\partial}{\partial v_k}+u\frac{\partial^2}{\partial u^2}+u\frac{\partial}{\partial u}.
	\end{align}
	We set
	\begin{align}
	&\mathcal{L}_0=-\mathcal{N}-2\mathcal{Y}_1-2\mathcal{Y}_2,\\
	&\mathcal{L}_1^{(1)}=\mathcal{N}+2\mathcal{Z}_1+2\mathcal{Z}_2,\hspace{.75cm}\mathcal{L}_2^{(1)}=2\mathcal{K}_1+\mathcal{K}_2,\\
	&\mathcal{L}_1^{(2)}=0,\hspace{.75cm}\mathcal{L}_2^{(2)}=-2\mathcal{K}_1^{-}-\mathcal{K}_2^{-}+\mathcal{J},\\
	&\mathcal{L}_1^{(4)}=-\frac{1}{2}\mathcal{L}_1^{(1)},\hspace{.75cm}\mathcal{L}_2^{(4)}=\frac{1}{4}\mathcal{L}_2^{(1)}.
	\end{align}
	
	For $\beta\in\{1,2,4\}$, we define
	\begin{equation}
	\DNbeta=\Lzerobeta +\frac{1}{N}\Lonebeta+\frac{1}{N^2}\Ltwobeta.
	\end{equation}
\end{definition}

\begin{notation}
	For $m\in\Z$ and $A\in\MNC$, let $W_m(A)=A^m$, $V_m(A)=\tr(A^m)$, and $\mathbf{V}(A)=\{V_m(A)\}_{|m|\geq 1}$. The functions $W_m, V_m,$ and $\mathbf{V}$ implicitly depend on $N$, but we suppress the index, which should not cause confusion. 
\end{notation}

\begin{theorem}[Partial product rule] \label{thm:partialprodrule}
	Let $\alpha,\beta\in\C$. For $P\in\C[u,u^{-1};\mathbf{v}]$ and $Q\in\C[\mathbf{v}]$, the operator $\alpha \mathcal{L}_0+\beta\mathcal{L}_1^{(1)}$ satisfies
	\begin{align}
	\left(\alpha\mathcal{L}_0+\beta\mathcal{L}_1^{(1)}\right)(PQ)&=\left(\left(\alpha\mathcal{L}_0+\beta\mathcal{L}_1^{(1)}\right)P\right)Q+P\left(\left(\alpha\mathcal{L}_0+\beta\mathcal{L}_1^{(1)}\right) Q\right).\label{eq:prodruleL0L1}
	\end{align}
\end{theorem}

\begin{proof}
	From Definition \ref{def:Doperators}, recall that $$\mathcal{L}_0=-(\mathcal{N}_0+\mathcal{N}_1)-2\mathcal{Y}_1-2\mathcal{Y}_2,\hspace{.5cm}\mathcal{L}_1^{(1)}=\mathcal{N}_0+\mathcal{N}_1+2\mathcal{Z}_1+2\mathcal{Z}_2.$$ Observe that $\mathcal{N}_0$, $\mathcal{Y}_2$, $\mathcal{Z}_1$ are first order differential operators on $\C[\mathbf{v}]$, and so satisfy the product rule on $\C[u,u^{-1};\mathbf{v}]$. On the other hand, $\mathcal{N}_1$, $\mathcal{Y}_1,$ and $\mathcal{Z}_2$ annihilate $\C[\mathbf{v}]$ and satisfy, for $P\in\C[u,u^{-1};\mathbf{v}]$ and $Q\in\C[\mathbf{v}],$
	\begin{align*}
	\mathcal{N}_1(PQ)&=(\mathcal{N}_1P)Q,\hspace{.5cm}\mathcal{Y}_1(PQ)=(\mathcal{Y}_1P)Q,\hspace{.5cm}\mathcal{Z}_2(PQ)=(\mathcal{Z}_2P)Q.
	\end{align*}
	Putting this together shows (\ref{eq:prodruleL0L1}).
\end{proof}

\begin{definition}
	The {\bf trace degree} of a monomial in $\Cuuv$ is defined to be
	\begin{equation}
	\deg\left(u^{k_0}v_1^{k_1}v_{-1}^{k_{-1}}\cdots v_m^{k_m}v_{-m}^{k_{-m}}\right)=|k_0|+\sum_{1\leq |j|\leq m}|j|k_j.
	\end{equation}
	We define the trace degree of any element in $\Cuuv$ to be the highest trace degree of any of its monomial terms. For $m\geq 0$, we define the subspace $\C_m[u,u^{-1};\mathbf{v}]\subseteq\Cuuv$ to be
	\begin{equation}
	\C_m[u,u^{-1};\mathbf{v}]=\{P\in\Cuuv\,:\,\deg P\leq m\},
	\end{equation}
	consisting of polynomials of trace degree $\leq m$. Note that $\C_m[u,u^{-1};\mathbf{v}]$ is finite-dimensional and $\Cuuv=\bigcup_{m\geq 0}\C_m[u,u^{-1};\mathbf{v}]$. 
\end{definition}

\begin{cor} \label{cor:DNinvariant}
	Let $m,N\in\N$ and $\alpha,\beta\in\C$. The operators  $\alpha\mathcal{L}_0+\beta\mathcal{L}_1^{(1)}$, $\mathcal{D}_N^{(1)}$, $\mathcal{D}_N^{(2)}$, and $\mathcal{D}_N^{(4)}$ preserve the finite dimensional subspace $\C_m[u,u^{-1};\mathbf{v}]\subseteq\Cuuv$. It follows that  $e^{\alpha\mathcal{L}_0+\beta\mathcal{L}_1^{(1)}}$, $e^{\alpha\mathcal{D}_N^{(1)}}$, $e^{\alpha\mathcal{D}_N^{(2)}}$, and $e^{\alpha\mathcal{D}_N^{(4)}}$ are well-defined operators on each $\C_m[u,u^{-1};\mathbf{v}]$ (via power series, cf. Remark \ref{rem:SBpowerseries}) and hence are well-defined on all of $\Cuuv$. 
\end{cor}

\begin{proof}
	Recall that $\alpha\mathcal{L}_0+\beta\mathcal{L}_1^{(1)}$, $\mathcal{D}_N^{(1)}$, $\mathcal{D}_N^{(2)}$, and $\mathcal{D}_N^{(4)}$ are linear combinations of the operators $\mathcal{J}$, $\mathcal{N}$, $\mathcal{Y}_1$, $\mathcal{Y}_2$, $\mathcal{Z}_1$, $\mathcal{Z}_2$, $\mathcal{K}_1$, and $\mathcal{K}_2$ described in Definition \ref{def:Doperators}. Each of these operators, in turn, are linear combinations of compositions of multiplication, differentiation, and projection operators. The multiplication and differentiation operators raise and lower trace degree, respectively, while projection operators $\mathcal{R}^+$ and $\mathcal{R}^-$ leaves $\C_m[u,u^{-1};\mathbf{v}]$ invariant. Keeping track of multiplication and differentiation operators in $\mathcal{J}$, $\mathcal{N}$, $\mathcal{Y}_1$, $\mathcal{Y}_2$, $\mathcal{Z}_1$, $\mathcal{Z}_2$, $\mathcal{K}_1$, and $\mathcal{K}_2$ shows that these operators preserve trace degree. Hence  $\alpha\mathcal{L}_0+\beta\mathcal{L}_1^{(1)}$, $\mathcal{D}_N^{(1)}$, $\mathcal{D}_N^{(2)}$, and $\mathcal{D}_N^{(4)}$ preserve the finite dimensional subspace $\C_m[u,u^{-1};\mathbf{v}]\subseteq\Cuuv$, and the result for  $e^{\alpha\mathcal{L}_0+\beta\mathcal{L}_1^{(1)}}$, $e^{\alpha\mathcal{D}_N^{(1)}}$, $e^{\alpha\mathcal{D}_N^{(2)}}$, and $e^{\alpha\mathcal{D}_N^{(4)}}$ follows. 
\end{proof}

\begin{cor} \label{cor:exppartialprodrule}
	For any $P\in\Cuuv$, $Q\in\C[\mathbf{v}]$, and $\alpha,\beta,\tau\in\C$,
	\begin{align}
	e^{\frac{\tau}{2}(\alpha\mathcal{L}_0+\beta\mathcal{L}_1^{(1)})}(PQ)&=e^{\frac{\tau}{2}(\alpha\mathcal{L}_0+\beta\mathcal{L}_1^{(1)})}P\cdot e^{\frac{\tau}{2}(\alpha\mathcal{L}_0+\beta\mathcal{L}_1^{(1)})}Q.\label{eq:expprodruleL0L1}
	\end{align}
\end{cor}

\begin{proof}
	Suppose $\mathcal{L}$ is a linear operator on $\Cuuv$ which leaves $\C[\mathbf{v}]$ and $\C_m[u,u^{-1};\mathbf{v}]$ invariant and satisfies the partial product rule (\ref{eq:prodruleL0L1}). Applying (\ref{eq:prodruleL0L1}) repeatedly, we get
	\begin{align*}
	\sum_{k=0}^\infty \frac{1}{k!}\mathcal{L}^k(PQ)&=\sum_{k=0}^\infty\frac{1}{k!}\sum_{\ell=0}^{k}\binom{k}{\ell}(\mathcal{L}^\ell P)(\mathcal{L}^{k-\ell}Q)=\sum_{k=0}^\infty\sum_{\ell=0}^{k}\frac{1}{\ell!(k-\ell)!}(\mathcal{L}^\ell P)(\mathcal{L}^{k-\ell}Q)\\
	&=\sum_{j,k=0}^\infty \frac{1}{j!k!}(\mathcal{L}^j P)(\mathcal{L}^{k}Q)=\left(\sum_{j=0}^\infty\frac{1}{j!}\mathcal{L}^j P\right)\left(\sum_{k=0}^\infty\frac{1}{k!}\mathcal{L}^k Q\right)\\
	&=e^{\mathcal{L}}P\cdot e^{\mathcal{L}}Q.
	\end{align*}
	By Definition \ref{def:Doperators}, it is clear that $\frac{\tau}{2}\left(\alpha\mathcal{L}_0+\beta\mathcal{L}_1^{(1)}\right)$ leaves $\C[\mathbf{v}]$ and $\C_m[u,u^{-1};\mathbf{v}]$ invariant, so setting $\mathcal{L}=\frac{\tau}{2}\left(\alpha\mathcal{L}_0+\beta\mathcal{L}_1^{(1)}\right)$ concludes the proof. 
\end{proof}

We can now prove our main intertwining result for $\Delta_{\SO(N,\R)}$. 

\begin{proof}[Proof of Theorem \ref{thm:intertwine1} for $\SO(N,\R)$] \label{pf:thm:intertwine1}
	It suffices to show that $\Delta_{\SO(N,\R)}P_N=[\mathcal{D}_NP]_N$ holds for $P(u;\mathbf{v})=u^mq(\mathbf{v})$, where $m\in\Z\setminus\{0\}$ and $q\in\C[\mathbf{v}]$. Then $P_N=W_m\cdot q(\mathbf{V})$, and by the product rule, 
	\begin{align}
	\Delta_{\SO(N,\R)}P_N&=\sum_{X\in\beta_{\so(N,\R)}}\wX^2(W_m\cdot q(\mathbf{V}))\nonumber\\
	&=\sum_{X\in\beta_{\so(N,\R)}}\wX\left((\wX W_m)\cdot q(\mathbf{V})+W_m\cdot\wX q(\mathbf{V})\right) \nonumber\\
	&=\sum_{X\in\beta_{\so(N,\R)}} (\wX^2 W_m)\cdot q(\V)+2\wX W_m\cdot \wX q(\V)+W_m\cdot \wX^2 q(\V) \nonumber\\
	&=(\Delta_{\SO(N,\R)}W_m)\cdot q(\V)\nonumber\\
	&\hspace{1cm}+2\sum_{X\in\beta_{\so(N,\R)}}\wX W_m\cdot \wX q(\V)+W_m\cdot (\Delta_{\SO(N,\R)}q(\V)) \label{eq:LaplacianPN}
	\end{align}
	
	For the first term in (\ref{eq:LaplacianPN}), we consider the cases $m\geq 0$ and $m<0$ separately. For $m\geq 0$, we apply (\ref{eq:DAmpos}) to get
	\begin{align}
	(\LSON & W_m)\cdot q(\V)\nonumber\\
	&=-\frac{m(N-1)}{N}W_m\cdot q(\V)\nonumber\\
	&\hspace{.5cm}+2\Ind_{m\geq 2}\left[\frac{1}{N}\sum_{k=1}^{m-1}(m-k)W_{m-2k}-\sum_{k=1}^{m-1}(m-k)V_kW_{m-k}\right]q(\V)\nonumber\\
	&=-\frac{N-1}{N}\left[u\p{u}\mathcal{R}^+P\right]_N+\frac{2}{N}\left[\left(\sum_{k=1}^\infty u^{-k+1}\mathcal{R}^+\p{u}\mathcal{M}_{u^{-k}}\mathcal{R}^+P\right)\right]_N \nonumber\\
	&\hspace{1.5cm}-2\left[\left(\sum_{k=1}^\infty v_ku\mathcal{R}^+\p{u}\mathcal{M}_{u^{-k}}\mathcal{R}^+\right)P\right]_N\nonumber \nonumber\\
	&=-\frac{N-1}{N}\left[u\p{u}\mathcal{R}^+P\right]_N+\frac{2}{N}[\mathcal{Z}_2^+P]_N-2[\mathcal{Y}_1^+P]_N.\label{eq:DWmqVpos}
	\end{align}
	A similar computation, using (\ref{eq:DAmneg}), shows that for $m<0$, 
	\begin{align*}
	(\LSON W_m)\cdot q(\V)&=\frac{N-1}{N}\left[u\p{u}\mathcal{R}^-P\right]_N-\frac{2}{N}[\mathcal{Z}_2^-P]_N+2[\mathcal{Y}_1^-P]_N.
	\end{align*}
	We observe that $\mathcal{Y}_1^+,\mathcal{Z}_2^+$ annihilates $\C[u^{-1}]$ while $\mathcal{Y}_1^-,\mathcal{Z}_2^-$ annihilates $\C[u]$, so for all $m\in\Z$, 
	\begin{align}
	(\LSON W_m)\cdot q(\V)&=-\frac{N-1}{N}[\mathcal{N}_1P]_N+\frac{2}{N}[\mathcal{Z}_2P]_N-2[\mathcal{Y}_1P]_N.\label{eq:DWmqV}
	\end{align}
	
	Next, by the chain rule and (\ref{eq:XtrAmXAp}), the middle term in (\ref{eq:LaplacianPN}) is
	\begin{align}
	\sum_{X\in\bSON}\wX& W_m\cdot\wX q(\V)\nonumber\\
	&=\sum_{X\in\bSON}\wX W_m\cdot\sum_{|k|\geq 1}\left(\p{v_k}q\right)(\V)\cdot\wX V_k\nonumber\\
	&=\sum_{|k|\geq 1}\left(\sum_{X\in\bSON} \wX W_m\cdot\wX V_k\right)\left(\p{v_k}q\right)(\V)\nonumber\\
	&=\sum_{|k|\geq 1}\frac{mk}{N^2}(W_{m-k}-W_{m+k})\left(\p{v_k}q\right)(\V)\nonumber\\
	&=\frac{1}{N^2}\sum_{|k|\geq 1}k\left(W_{-k+1}\left[\p{u}u^m\right]_N-W_{k+1}\left[\p{u}u^m\right]_N\right)\left(\p{v_k}q\right)(\V)\label{eq:mWmj}\\
	&=\frac{1}{N^2}\left[\left(\sum_{|k|\geq 1} k(u^{-k+1}-u^{k+1})\pp{u}{v_k}\right) P\right]_N\nonumber\\
	&=\frac{1}{N^2}[\K_1 P]_N.\label{eq:XWmXqV}
	\end{align}
	where we have used the fact that $mW_{m+j}=W_{j+1}\left[\p{u}u^m\right]_N$ in line (\ref{eq:mWmj}).
	
	Finally, for the last term in (\ref{eq:LaplacianPN}), we have, for each $X\in\bSON$, 
	\begin{align}
	\wX^2 q(\V)&=\wX\left(\sum_{|k|\geq 1}\left(\p{v_k}q\right)(\V)\cdot \wX V_k\right)\nonumber\\
	&=\sum_{|k|\geq 1}\left(\p{v_k}q\right)(\V)\cdot\wX^2 V_k+\sum_{|j|,|k|\geq 1}\left(\pp{v_j}{v_k}q\right)(\V)\cdot \wX V_j\cdot \wX V_k.\label{eq:X2qV}
	\end{align}
	
	In summing the first term in (\ref{eq:X2qV}) over $X\in\bSON$, we break up the sum for positive and negative $k$. Using (\ref{eq:DAmpos}) and (\ref{eq:DAmneg}), we have
	\begin{align}
	\sum_{X\in\bSON}&\sum_{k=1}^\infty\left(\p{v_k}q\right)(\V)\cdot\wX^2 V_k\nonumber\\
	&=-\frac{N-1}{N}\sum_{k=1}^\infty kV_k\left(\p{v_k}q\right)(\V)+\frac{2}{N}\sum_{k=2}^\infty\sum_{\ell=1}^{k-1}(k-\ell)V_{k-2\ell}\left(\p{v_k}q\right)(\V)\nonumber\\
	&\hspace{1cm}-2\sum_{k=2}^\infty\sum_{\ell=1}^{k-1}(k-\ell)V_\ell V_{k-\ell}\left(\p{v_k}q\right)(\V)\label{eq:N0Z1Y2pos}
	\end{align}
	and
	\begin{align}
	\sum_{X\in\bSON}&\sum_{k=-\infty}^{-1}\left(\p{v_k}q\right)(\V)\cdot\wX^2 V_k \nonumber\\
	&=\frac{N-1}{N}\sum_{k=-\infty}^{-1} kV_k\left(\p{v_k}q\right)(\V)-\frac{2}{N}\sum_{k=-\infty}^{-2}\sum_{\ell=k+1}^{-1}(k-\ell)V_{k-2\ell}\left(\p{v_k}q\right)(\V)\nonumber\\
	&\hspace{1cm}-2\sum_{k=-\infty}^{-2}\sum_{\ell=k+1}^{-1}(-k+\ell)V_\ell V_{-k-\ell}\left(\p{v_k}q\right)(\V).\label{eq:N0Z1Y2neg}
	\end{align}
	
	Summing the second term in (\ref{eq:X2qV}) over $X\in\bSON$ and using (\ref{eq:XtrAmXAp}) yields
	\begin{align}
	\sum_{X\in\bSON}&\sum_{|j|,|k|\geq 1}\left(\pp{v_j}{v_k}q\right)(\V)\cdot \wX V_j\cdot \wX V_k\nonumber\\
	&=\frac{1}{N^2}\sum_{|j|,|k|\geq 1}jk(V_{k-j}-V_{k+j})\left(\pp{v_j}{v_k}q\right)(\V)\label{eq:K2}
	\end{align}
	Adding (\ref{eq:N0Z1Y2pos}), (\ref{eq:N0Z1Y2neg}), and (\ref{eq:K2}) together and multiplying by $W_m$, we get
	\begin{align}
	W_m\cdot (\LSON q(\V))&=\left[\left(-\frac{N-1}{N}\mathcal{N}_0+\frac{2}{N}\mathcal{Z}_1-2\mathcal{Y}_2+\frac{1}{N^2}\mathcal{K}_2\right)P\right]_N \label{eq:WmDqV}
	\end{align}
	
	Combining (\ref{eq:DWmqV}), (\ref{eq:XWmXqV}), and (\ref{eq:WmDqV}) proves (\ref{eq:intertwine1}) for $\SO(N,\R)$ ($\beta=1$). Equation (\ref{eq:intertwine1exp}) now follows from Corollary (\ref{cor:DNinvariant}). 
\end{proof}

\begin{remark} \label{rem:intertwineU(N)}
	A similar intertwining formula is proven in \cite{DHK:U(N)} for the $\U(N)$ case, where for any $P\in\Cuuv$,
	\begin{align} \label{eq:intertwineU(N)} 
	\Delta_{\U(N)}P_N&=\left[\left(\mathcal{L}_0-\frac{1}{N^2}(2\mathcal{K}_1^-+\mathcal{K}_2^-)\right)P\right]_N,
	\end{align}
	and, subsequently, for $\tau\in\C$,
	\begin{align} 
	e^{\frac{\tau}{2}\Delta_{\U(N)}}P_N&=[e^{\frac{\tau}{2}(\mathcal{L}_0-\frac{1}{N^2}(2\mathcal{K}_1^-+\mathcal{K}_2^-))}P]_N.\label{eq:intertwineUNexp}
	\end{align}
\end{remark}

\begin{remark} \label{rem:SBpowerseries}
	In this paper, we are primarily concerned with computing the Segal-Bargmann transform on trace polynomials. In this setting, if $D$ is a left-invariant differential operator on $K$, where $K$ is a compact type matrix Lie group, then we can compute the action of $e^D$ on a trace polynomial $f$ via the power series
	\begin{equation}
	(e^{D}f)(x)=\left(\sum_{n=0}^\infty\frac{1}{n!}D^nf\right)(x).
	\end{equation}
	In particular, if $D=\frac{\tau}{2}\Delta_K$, we can define $e^{\frac{\tau}{2}\Delta_K}f$ as a function on $K$ using the power series definition. In \cite{DHK:complextimeSB}, it is shown that $e^{\frac{\tau}{2}\Delta_K}f$ has a unique analytic continuation to $K_\C$. Moreover, this analytic continuation is equal to $B_{s,\tau}^Kf$. 
\end{remark}

\begin{proof}[Proof of Theorem \ref{thm:SBintertwine} for $\SO(N,\R)$] \label{pf:thm:SBintertwine}
	By the intertwining formula (\ref{eq:intertwine1exp}), we have \[e^{\frac{\tau}{2}\LSON}P_N=[e^{\frac{\tau}{2}\mathcal{D}_N^{(1)}}P]_N.\]  Corollary \ref{cor:DNinvariant} shows that $[e^{\frac{\tau}{2}\mathcal{D}_N^{(1)}}P]_N$ is a trace polynomial on $\SO(N,\R)$, so its analytic continuation to $\SO(N,\C)$ is given by the same trace polynomial function. Hence $[e^{\frac{\tau}{2}\mathcal{D}_N^{(1)}}P]_N$, viewed as a holomorphic function on $\SO(N,\C)$, is the analytic continuation of $e^{\frac{\tau}{2}\LSON}P_N$, and so is equal to $\mathbf{B}_{s,\tau}^{\SO(N,\R)} P_N$. 
\end{proof}

\subsection{Intertwining formulas for $\mathcal{A}^{\SO(N,\C)}_{s,\tau}$}

We now prove an intertwining formula for $\SO(N,\C)$ that is analogous to the intertwining formula for $\SO(N,\R)$ in Theorem \ref{thm:intertwine1}. To do so, we use the word polynomial space introduced in \cite[Notation 3.21]{DHK:U(N)}, which was used to prove a similar intertwining formula for $\GL_N(\C)=\U(N)_{\C}$. 

\begin{notation}
	For $m\in\N$, define $\mathscr{E}_m$ to be the set of words $\varepsilon:\{1,\hdots ,m\}\to\{\pm 1,\pm *\}$. We denote the length of a word $\varepsilon\in\mathscr{E}_m$ by $|\varepsilon|=m$ and set $\mathscr{E}=\cup_{m\geq 0}\mathscr{E}_m$. We define the {\bf word polynomial space} $\mathscr{W}$ to be $$\mathscr{W}=\C\left[\{v_{\varepsilon}\}_{\varepsilon\in\mathscr{E}}\right],$$ the space of polynomials in the commuting variables $\{v_{\varepsilon}\}_{\varepsilon\in\mathscr{E}_m}$. For $j,k\in\Z$ not both zero, we define the words
	\begin{align}
	\varepsilon(j,k)&=\overbrace{(\pm 1,\hdots ,\pm 1,}^{|j|\text{ times}}\overbrace{\pm *,\hdots ,\pm *)}^{|k|\text{ times}}\in\mathscr{E}_{|j|+|k|},
	\end{align}
	where the first $|j|$ slots contain $+1$ if $j>0$ and $-1$ if $j<0$, and the last $|k|$ slots contain $+*$ if $k>0$ and $-*$ if $k<0$. We set $v_{\varepsilon(0,0)}=1$. 
\end{notation}

\begin{notation} \label{not:wordfunctionalcalc}
	For $\varepsilon=(\varepsilon_1,\hdots ,\varepsilon_m)\in\mathscr{E}_m$ and $A\in\SO(N,\C)$, we define $A^\varepsilon=A^{\varepsilon_1}\cdots A^{\varepsilon_m}$, where $A^{+*}:=A^*$ and $A^{-*}:=(A^*)^{-1}$. If $P\in\mathscr{W}$, we define a function $P_N:\SO(N,\C)\to\C$ by $$P_N(A)=P(\mathbf{V}(A)),$$ where $$\mathbf{V}(A)=\{V_\varepsilon (A)\, :\, \varepsilon\in\mathscr{E}\}$$ and $$V_\varepsilon (A)=\tr(A^\varepsilon)=\tr(A^{\varepsilon_1}\cdots A^{\varepsilon_m}).$$
\end{notation}

\begin{notation}
	We define the inclusion maps $\iota,\iota^*:\C[\mathbf{v}]\hookrightarrow\mathscr{W}$, with $\iota$ linear and $\iota^*$ conjugate linear, by
	\begin{equation}
	\iota(v_k)=v_{\varepsilon(k,0)},\hspace{.5cm}\iota^*(v_k)=v_{\varepsilon(0,k)}.
	\end{equation}
	The maps $\iota$ and $\iota^*$ intertwine with the evaluation map in the following way: for $Q\in\C [\mathbf{v}]$,
	\begin{equation} \label{eq:intertwineword}
	[\iota(Q)]_N(A)=Q_N(A),\hspace{.5cm}[\iota^*(Q)]_N(A)=Q_N(A)^*.
	\end{equation}
\end{notation}

\begin{definition}
	The {\bf trace degree} of a monomial $\prod_{j=1}^mv_{\varepsilon_j}^{k_j}\in\mathscr{W}$ is defined to be $$\deg\left(\prod_{j=1}^mv_{\varepsilon_j}^{k_j}\right) =\sum_{j=1}^m |k_j||\varepsilon_j|,$$
	and the trace degree of an arbitrary element of $\mathscr{W}$ is the highest trace degree of any of its monomial terms. 
	
	For each $m\in\N$, we let $\mathscr{W}_m$ denote the finite dimensional subspace of $\mathscr{W}$ $$\mathscr{W}_m=\{P\in\mathscr{W}\,:\,\deg(P)\leq m\}\subseteq\C[\{v_{\varepsilon}\}_{|\varepsilon|\leq m}],$$ so that $\mathscr{W}=\bigcup_{m=1}^\infty\mathscr{W}_m$. 
\end{definition}

\begin{theorem} \label{thm:intertwiningIIpoly}
	Fix $s\in\R$ and $\tau=t+i\theta\in\C$. There are collections $\left\{Q_\varepsilon^{s,\tau}\, :\,\varepsilon\in\mathscr{E}\right\}$, $\left\{R_\varepsilon^{s,\tau}\, :\,\varepsilon\in\mathscr{E}\right\}$, and $\left\{S_{\varepsilon,\delta}^{s,\tau}\, :\,\varepsilon,\delta\in\mathscr{E}\right\}$ in $\mathscr{W}$ such that:
	\begin{enumerate}
		\item for each $\varepsilon\in\mathscr{E}$, $Q_\varepsilon^{s,\tau}$ and $R_\varepsilon^{s,\tau}$ are certain finite sums of monomials of trace degree $|\varepsilon|$ such that
		\begin{equation} \label{eq:ASONstVe}
			\ASONst V_\varepsilon=[Q_\varepsilon^{s,\tau}]_N+\frac{1}{N}[R_{\varepsilon}^{s,\tau}]_N,
		\end{equation}
		\item for $\varepsilon,\delta\in\mathscr{E}$, $S_{\varepsilon,\delta}^{s,\tau}$ is a certain finite sum of monomials of trace degree $|\varepsilon|+|\delta|$ such that
		\begin{equation} \label{eq:Redst}
		\sum_{\ell=1}^{\dSON}\left[\left(s-\frac{t}{2}\right)(\wX_\ell V_\varepsilon)(\wX_\ell V_\delta)+\frac{t}{2}(\wY_\ell V_\varepsilon)(\wY_\ell V_\delta)-\theta(\wX_\ell V_\varepsilon)(\wY_\ell V_\delta)\right]=\frac{1}{N^2}[S_{\varepsilon ,\delta}^{s,\tau}]_N.
		\end{equation}
	\end{enumerate}
\end{theorem}

For the proof below, we use the following conventions: let $\bSON=\{X_\ell\}_{\ell=1}^{\dSON}$ denote an orthonormal basis for $\so(N,\R)$, with $\beta_+=\bSON$ and $\beta_-=i\bSON$. If $X\in\MNC$ and $A\in\SO(N,\C)$,
\begin{equation}
(AX)^1:=AX,\hspace{.5cm}(AX)^{-1}:=-XA^{-1},\hspace{.5cm}(AX)^*:=X^*A^*,\hspace{.5cm}(AX)^{-*}:=-A^*X^*.
\end{equation}
In addition, we will be liberal in our use of $\pm$ to denote a sign that may vary for different terms and on different sides of an equation, since we will not require a precise formula for our word polynomials beyond what is described in the theorem statement.

\begin{proof}
	Fix a word $\varepsilon=(\varepsilon_1,\hdots ,\varepsilon_m)\in\mathscr{E}$. Then for each $X\in\beta_{\pm}$ and $A\in\SO(N,\C)$, 
	\begin{equation}
	(\wX V_\varepsilon)(A)=\sum_{j=1}^m\tr(A^{\varepsilon_1}\cdots (AX)^{\varepsilon_j}\cdots A^{\varepsilon_m}),
	\end{equation}
	so
	\begin{align}
	(\wX^2 V_\varepsilon)(A)=&\sum_{j=1}^m\tr(A^{\varepsilon_1}\cdots (AX^2)^{\varepsilon_j}\cdots A^{\varepsilon_m}) \label{eq:trAXX}\\
	&+2\sum_{1\leq j<k\leq m}\tr(A^{\varepsilon_1}\cdots (AX)^{\varepsilon_j}\cdots (AX)^{\varepsilon_k}\cdots A^{\varepsilon_m}) \label{eq:trAXAX}
	\end{align}
	
	Applying magic formula (\ref{eq:magic1}) to sum each term in (\ref{eq:trAXX}) over $\beta_{\pm}$, we have
	\begin{equation}
	\sum_{X\in\beta_\pm}\tr(A^{\varepsilon_1}\cdots (AX^2)^{\varepsilon_j}\cdots A^{\varepsilon_m})=\pm\frac{N-1}{N}\tr (A^{\varepsilon_1}\cdots A^{\varepsilon_j}\cdots A^{\varepsilon_m})=\pm\frac{N-1}{N}V_\varepsilon (A).
	\end{equation}
	Summing over $1\leq j\leq m$ now gives
	\begin{equation}
	\sum_{X\in\beta_\pm}\sum_{j=1}^m\tr(A^{\varepsilon_1}\cdots (AX^2)^{\varepsilon_j}\cdots A^{\varepsilon_m})=\frac{N-1}{N}n_\pm(\varepsilon)V_\varepsilon(A),
	\end{equation}
	where $n_\pm(\varepsilon)\in\Z$ and $|n_\pm (\varepsilon)|\leq|\varepsilon|$. For (\ref{eq:trAXAX}), we can express each term in the sum as
	\begin{equation} \label{eq:trAXAXterm}
	\tr(A^{\varepsilon_1}\cdots (AX)^{\varepsilon_j}\cdots (AX)^{\varepsilon_k}\cdots A^{\varepsilon_m})=\pm\tr(A^\ejkzero XA^\ejkone XA^\ejktwo),
	\end{equation}
	so $\ejkzero$, $\ejkone$, and $\ejktwo$ are substrings of $\varepsilon$ such that $\ejkzero\ejkone\ejktwo=\varepsilon$. Summing (\ref{eq:trAXAXterm}) over $X\in\beta_\pm$ by magic formula (\ref{eq:magic2}), we have
	\begin{align}
	\sum_{X\in\beta_\pm} \tr(A^{\varepsilon_1}\cdots &(AX)^{\varepsilon_j}\cdots (AX)^{\varepsilon_k}\cdots A^{\varepsilon_m})\nonumber\\
	&=\pm\left[\frac{1}{N}\tr(A^\ejkzero (A^\ejkone)^{-1}A^\ejktwo)-\tr(A^\ejkzero A^\ejktwo)\tr(A^\ejkone)\right]\\
	&=\pm\frac{1}{N}V_\ejkthree (A) \pm V_{\ejkzero\ejktwo}(A)V_\ejkone(A),
	\end{align}
	where $\ejkthree$ is the word of length $|\varepsilon|$ such that $V_\ejkthree(A)=\tr(A^\ejkzero (A^\ejkone)^{-1}A^\ejktwo)$. Thus summing (\ref{eq:trAXAX}) over $X\in\beta_\pm$ gives
	\begin{equation}
	\begin{aligned}
	\sum_{X\in\beta_\pm}\sum_{1\leq j<k\leq m}&\tr(A^{\varepsilon_1}\cdots (AX)^{\varepsilon_j}\cdots (AX)^{\varepsilon_k}\cdots A^{\varepsilon_m})\nonumber\\
	&=\frac{1}{N}\sum_{1\leq j<k\leq m}\pm V_{\ejkthree}(A)+\sum_{1\leq j<k\leq m}\pm V_{\ejkzero\ejktwo}(A)V_{\ejkone}(A).\label{eq:Qepsilon}
	\end{aligned}
	\end{equation}
	We define word polynomials corresponding to (\ref{eq:Qepsilon}) by
	\begin{align}
	Q_{\varepsilon,0}^\pm&=\sum_{1\leq j<k\leq m}\pm v_{\ejkthree}\\
	Q_{\varepsilon,1}^\pm&=\sum_{1\leq j<k\leq m}\pm v_{\ejkzero\ejktwo}v_{\ejkone}.
	\end{align}
	
	Now if $X\in\beta_+$ and $Y=iX\in\beta_-$, then
	\begin{equation}
	(\wY V_\varepsilon)(A)=\sum_{j=1}^m\tr(A^{\varepsilon_1}\cdots (AY)^{\varepsilon_j}\cdots A^{\varepsilon_m})=i\sum_{j=1}^m\pm \tr(A^{\varepsilon_1}\cdots (AX)^{\varepsilon_j}\cdots A^{\varepsilon_m}),
	\end{equation}
	and
	\begin{align}
	(\wX\wY V_\varepsilon)(A)=&i\sum_{j=1}^m\pm \tr(A^{\varepsilon_1}\cdots (AX^2)^{\varepsilon_j}\cdots A^{\varepsilon_m}) \label{eq:itrAXX}\\
	&+2i\sum_{1\leq j<k\leq m}\pm \tr(A^{\varepsilon_1}\cdots (AX)^{\varepsilon_j}\cdots (AX)^{\varepsilon_k}\cdots A^{\varepsilon_m}).\label{eq:itrAXAX}
	\end{align}
	A similar argument to the above shows that summing (\ref{eq:itrAXX}) over $X\in\beta_+$ gives
	\begin{equation}
	\sum_{X\in\beta_+}\sum_{j=1}^m \pm \tr(A^{\varepsilon_1}\cdots (AX^2)^{\varepsilon_j}\cdots A^{\varepsilon_m})=\frac{N-1}{N}\eta(\varepsilon)V_\varepsilon(A)
	\end{equation}
	where $\eta(\varepsilon)\in\Z$ and $|\eta(\varepsilon)|\leq |\varepsilon|$. In addition, summing (\ref{eq:itrAXAX}) over $X\in\beta_+$, we have
	\begin{equation}
	\begin{aligned}
	\sum_{X\in\beta_+}\sum_{1\leq j<k\leq m}\pm & \tr(A^{\varepsilon_1}\cdots (AX)^{\varepsilon_j}\cdots (AX)^{\varepsilon_k}\cdots A^{\varepsilon_m}) \\
	&=\frac{1}{N}\sum_{1\leq j<k\leq m}V_{\ejkthree}(A)+\sum_{1\leq j<k\leq m}\pm V_{\ejkzero\ejktwo}(A)V_{\ejkone}(A).
	\end{aligned}
	\end{equation} 
	We define corresponding word polynomials
	\begin{align}
	Q_{\varepsilon,2}&=i\sum_{1\leq j<k\leq m}\pm v_{\ejkthree}\\
	Q_{\varepsilon,3}&=i\sum_{1\leq j<k\leq m}\pm v_{\ejkzero\ejktwo}v_{\ejkone}.
	\end{align}
	Putting this together, we define
	\begin{align}
	Q_\varepsilon^{s,\tau}&=\left(s-\frac{t}{2}\right)\left(n_+(\varepsilon)v_\varepsilon +  2Q_{\varepsilon,1}^+\right)+\frac{t}{2}\left(n_-(\varepsilon)v_\varepsilon+2Q_{\varepsilon,1}^-\right)-\theta\left(i\eta(\varepsilon)+2Q_{\varepsilon,3}\right),\\
	R_\varepsilon^{s,\tau}&=\left(s-\frac{t}{2}\right)\left(-n_+(\varepsilon)v_\varepsilon +  2Q_{\varepsilon,0}^+\right)+\frac{t}{2}\left(-n_-(\varepsilon)v_\varepsilon+2Q_{\varepsilon,0}^-\right)-\theta\left(-i\eta(\varepsilon)+2Q_{\varepsilon,2}\right).
	\end{align}
	By construction, $Q_\varepsilon^{s,\tau}$ and $R_\varepsilon^{s,\tau}$ are of the required form and satisfy (\ref{eq:ASONstVe}). 
	
	For part (2), fix $\delta\in\mathscr{E}$ with $n=|\delta|$. For each $X\in\beta_\pm$, 
	\begin{equation}
	(\wX V_\delta)(A)(\wX V_\varepsilon)(A)=\sum_{j=1}^m\sum_{k=1}^n\tr(A^{\varepsilon_1}\cdots (AX)^{\varepsilon_j}\cdots A^{\varepsilon_m})\tr(A^{\delta_1}\cdots (AX)^{\delta_k}\cdots A^{\delta_n}).
	\end{equation}
	Using the invariance of the trace under cyclic permutations, we can write the terms in the above sum as $$\pm \tr(XA^{\varepsilon(j)})\tr(XA^{\delta(k)}),$$
	where $\varepsilon(j)$ and $\delta(k)$ are particular cyclic permutations of $\varepsilon$ and $\delta$. Summing over $X\in\beta_\pm$ and applying magic formula (\ref{eq:magic4}), along with the fact that $X\in\so(N,\C)$ is skew-symmetric,
	\begin{align}
	\sum_{X\in\beta_\pm}(\wX V_\delta)(A)(\wX V_\varepsilon)(A)&=\frac{1}{N^2}\sum_{j=1}^m\sum_{k=1}^n\pm\left(\tr((A^{\varepsilon(j)})^{-1}A^{\delta (k)})-\tr(A^{\varepsilon(j)}A^{\delta(k)}\right) \nonumber\\
	&=\frac{1}{N^2}\sum_{j=1}^m\sum_{k=1}^n\pm (V_{{\varepsilon_(j)^\intercal\delta(k)}}(A)-V_{\varepsilon(j)\delta(k)}(A)),
	\end{align} 
	where $\varepsilon(j)^\intercal$ is the word satisfying $V_{\varepsilon(j)^\intercal}(A)=V_{\varepsilon(j)}(A)^{-1}$. (For example, if $\varepsilon(j)=(-*,-1,*,1,1)$, then $\varepsilon(j)^\intercal=(-1,-1,-*,1,*)$.)	We define the associated word polynomials
	\begin{equation} \label{eq:Spmed}
	S_{\varepsilon,\delta}^\pm=\sum_{j=1}^m\sum_{k=1}^n\pm (v_{{\varepsilon_(j)^\intercal\delta(k)}}-v_{\varepsilon(j)\delta(k)}).
	\end{equation}
	
	Next, if $X\in\beta_+$ and $Y=iX\in\beta_-$, then
	\begin{equation}
	(\wX V_\varepsilon)(A)(\wY V_\delta)(A)=i\sum_{j=1}^m\sum_{k=1}^n\pm \tr(XA^{\varepsilon(j)})\tr(XA^{\delta(k)}).
	\end{equation}
	Using magic formula (\ref{eq:magic4}) to sum over $X\in\beta_+$ gives
	\begin{align}
	\sum_{\ell=1}^{\dSON}(\wX_\ell V_\varepsilon)(A)(\wY_\ell V_\delta)(A)&=\frac{i}{N^2}\sum_{j=1}^m\sum_{k=1}^n\pm (V_{{\varepsilon_(j)^\intercal\delta(k)}}(A)-V_{\varepsilon(j)\delta(k)}(A)).
	\end{align}
	Hence we define the associated word polynomial
	\begin{equation}\label{eq:Sprimeed}
	S'_{\varepsilon,\delta}=i\sum_{j=1}^m\sum_{k=1}^n\pm (v_{{\varepsilon_(j)^\intercal\delta(k)}}-v_{\varepsilon(j)\delta(k)})
	\end{equation}
	(where the signs in (\ref{eq:Sprimeed}) do not necessarily correspond to those in (\ref{eq:Spmed})). Finally, we define
	\begin{equation}
	S^{s,\tau}_{\varepsilon,\delta}=\left(s-\frac{t}{2}\right)S_{\varepsilon,\delta}^++\frac{t}{2}S_{\varepsilon,\delta}^--\theta S'_{\varepsilon,\delta},
	\end{equation}
	where we have constructed $S_{\varepsilon,\delta}^{s,\tau}$ so that it is a sum of monomials of trace degree $|\varepsilon|+|\delta|$ and satisfies (\ref{eq:Redst}). 
\end{proof}

\begin{theorem}[Intertwining formula for $\ASONst$] \label{thm:intertwiningII}
	Fix $s\in\R$ and $\tau=t+i\theta\in\C$, and let $\left\{Q_\varepsilon^{s,\tau}\, :\,\varepsilon\in\mathscr{E}\right\}$, $\left\{R_\varepsilon^{s,\tau}\, :\,\varepsilon\in\mathscr{E}\right\}$, and $\left\{S_{\varepsilon,\delta}^{s,\tau}\, :\,\varepsilon,\delta\in\mathscr{E}\right\}$ be as in Theorem \ref{thm:intertwiningIIpoly}. Define first and second order differential operators
	\begin{align}
	\wL_0^{s,\tau}&=\frac{1}{2}\sum_{\varepsilon\in\mathscr{E}}Q_{\varepsilon}^{s,\tau}\p{v_\varepsilon}\\
	\wL_1^{s,\tau}&=\frac{1}{2}\sum_{\varepsilon\in\mathscr{E}}R_{\varepsilon}^{s,\tau}\p{v_\varepsilon}\\
	\wL_2^{s,\tau}&=\frac{1}{2}\sum_{\varepsilon,\delta\in\mathscr{E}}S_{\varepsilon,\delta}^{s,\tau}\pp{v_\varepsilon}{v_\delta}
	\end{align}
	Then for all $N\in\N$ and $P\in\mathscr{W}$, 
	\begin{equation} \label{eq:intertwine2}
	\frac{1}{2}\ASONst P_N=\left[\left(\wL_0^{s,\tau}+\frac{1}{N}\wL_1^{s,\tau}+\frac{1}{N^2}\wL_2^{s,\tau}\right)P\right]_N.
	\end{equation}
\end{theorem}

\begin{proof}
	For each $X\in\M_N(\C)$, we apply the chain rule to get
	\begin{align*}
	\wX^2P_N&=\sum_{\varepsilon\in\mathscr{E}}\wX\left[\left(\dP\right)(\V)\cdot (\wX V_\varepsilon)\right]\\
	&=\sum_{\varepsilon\in\mathscr{E}}\left(\dP\right) (\V)\cdot (\wX^2 V_\varepsilon)+\sum_{\varepsilon,\delta\in\mathscr{E}}\left( \ddP \right)(\V)\cdot (\wX V_\varepsilon)(\wX V_\delta).
	\end{align*} 
	Hence
	\begin{align*}
	\ASONst P_N =\sum_{\varepsilon\in\mathscr{E}}&\left(\dP\right)(\V)\cdot\ASONst V_\varepsilon\\
	+&\sum_{\varepsilon,\delta\in\mathscr{E}}\left(\ddP\right)(\V)\sum_{\ell=1}^{\dSON}\bigg[\left(s-\frac{t}{2}\right)(\wX_\ell V_\varepsilon)(\wX_\ell V_\delta)\\
	&\hspace{5cm}+\frac{t}{2}(\wY_\ell V_\varepsilon)(\wY_\ell V_\delta)-\theta(\wX_\ell V_\varepsilon)(\wY_\ell V_\delta)\bigg].
	\end{align*}
	The result now follows from Theorem \ref{thm:intertwiningIIpoly}. 
\end{proof}

The following result will be useful for the proof of Theorem \ref{thm:freeSBlimit}. It appears in \cite{DHK:U(N)} for the $\GL(N,\C)$ case, with the proof virtually unchanged for $\SO(N,\C)$. 

\begin{lemma}[{\cite[Lemma 3.28]{DHK:U(N)}}] \label{lem:sesquilinearform}
	There exists a sesquilinear form (with the second argument conjugate linear) $$\mathcal{B}:\Cuuv\times\Cuuv\to\mathscr{W}$$ such that, for all $P,Q\in\Cuuv$, we have $\deg(\mathcal{B}(P,Q))=\deg(P)+\deg(Q)$ and $$[\mathcal{B}(P,Q)]_N(A)=\tr [P_N(A)Q_N(A)^*]\hspace{.5cm}\mbox{for all }A\in\SO(N,\C).$$
\end{lemma}

\section{Limit theorems for the Segal-Bargmann transform on $\SO(N,\R)$} \label{sec:limittheoremsSONR}

We now use the results from the previous section to identify the limit of the Segal-Bargmann transform on $\SO(N,\R)$ as $N\to\infty$. In Section \ref{subsec:concmeasures}, we prove a few preliminary results regarding the way in which the heat kernel measures $\rho_s^{\SO(N,\R)}$ and $\mu_{s,\tau}^{\SO(N,\C)}$ concentrate their mass as $N\to\infty$. Then, in Section \ref{subsec:freeSBtransf}, we prove the $\SO(N,\R)$ version of Theorem \ref{thm:freeSBlimit}, the main limit theorem for $\mathbf{B}_{s,\tau}^{\SO(N,\R)}$. 

\subsection{Concentration of measures} \label{subsec:concmeasures}

The following lemma is an essential component of our main limit theorems. It is a known result (see \cite[Corollary 6.2.32]{HJ:topics}); we include a direct proof here for convenience. 

\begin{lemma} \label{lem:expmatrixnorm}
	Let $X,Y\in M(N,\C)$ and suppose $\|\cdot\|$ is a submultiplicative matrix norm. Then $$\|e^{X+Y}-e^X\|\leq \|Y\|e^{\|X\|}e^{\|Y\|}.$$
\end{lemma}

\begin{proof}
	For $n\geq 0$, note that $$\|(X+Y)^n-X^n\|\leq (\|X\|+\|Y\|)^n-\|X\|^n,$$ where the inequality follows by expanding $(X+Y)^n$, which includes an $X^n$, then applying the submultiplicativity of the matrix norm, and then recombining terms. Hence
	\begin{align*}
	\|e^{X+Y}-e^X\|&=\left\|\sum_{n=0}^\infty \frac{(X+Y)^n}{n!}-\sum_{n=0}^\infty \frac{X^n}{n!} \right\|\\
	&\leq\sum_{n=0}^\infty\frac{1}{n!}\|(X+Y)^n-X^n\|\\
	&\leq\sum_{n=0}^\infty \frac{1}{n!}[(\|X\|+\|Y\|)^n-\|X\|^n]\\
	&=e^{\|X\|+\|Y\|}-e^{\|X\|}.
	\end{align*}
	It remains to show that $e^{\|X\|+\|Y\|}-e^{\|X\|}\leq \|Y\|e^{\|X\|}e^{\|Y\|}$, which is equivalent to showing that $e^{y}-1\leq ye^y$ for $y\geq 0$. Rearranging, this inequality is equivalent to $1-y\leq e^{-y}$, which holds, in fact, for all $y\in\R$: by Bernoulli's inequality, 
	\begin{equation*}
	e^{-y}=\lim_{n\to\infty}\left(1-\frac{y}{n}\right)^n\geq 1-y.\qedhere
	\end{equation*}
\end{proof}

\begin{cor} \label{cor:expmatrixnorm}
	Let $V$ be a finite-dimensional normed vector space over $\C$ and suppose that that $L_0$, $L_1$, and $L_2$ are operators on $V$. Then there exist constants $C_1,C_2,C_3<\infty$ depending on $L_0, L_1, L_2,\|\cdot\|_V$ such that for $N$ sufficiently large,
	\begin{gather}
	\|e^{L_0+\frac{1}{N}L_1+\frac{1}{N^2}L_2}-e^{L_0}\|_{\End(V)}\leq \frac{C_1}{N},\label{eq:expopnorm}\\
	\|e^{L_0+\frac{1}{N}L_1+\frac{1}{N^2}L_2}-e^{L_0+\frac{1}{N}L_1}\|_{\End(V)}\leq \frac{C_2}{N^2}\label{eq:expopnorm2}\\
	\|e^{L_0+\frac{1}{N}L_1}-e^{L_0}\|_{\End(V)}\leq \frac{C_3}{N}\label{eq:expopnorm3}
	\end{gather}
	where $\|\cdot\|_{\End(V)}$ is the operator norm on $\End(V)$ induced by $\|\cdot\|_V$. Hence if $\varphi\in V^*$ is a linear functional, then for $N$ sufficiently large,
	\begin{gather}
	|\varphi(e^{L_0+\frac{1}{N}L_0+\frac{1}{N^2}L_2}v)-\varphi(e^{L_0}v)|\leq \frac{C_1}{N}\|\varphi\|_{V^*}\|v||_V, \label{eq:expfuncnorm}\\
	|\varphi(e^{L_0+\frac{1}{N}L_1+\frac{1}{N^2}L_2}v)-\varphi(e^{L_0+\frac{1}{N}L_1}v)|\leq \frac{C_2}{N^2}\|\varphi\|_{V^*}\|v||_V \label{eq:expfuncnorm2}\\
	|\varphi(e^{L_0+\frac{1}{N}L_1}v)-\varphi(e^{L_0}v)|\leq \frac{C_3}{N}\|\varphi\|_{V^*}\|v||_V \label{eq:expfuncnorm3}
	\end{gather}
	where $\|\cdot\|_{V^*}$ is the dual norm on $V^*$. 
\end{cor}

\begin{proof}
	We set $X=L_0$ and $Y=\frac{1}{N}L_1+\frac{1}{N^2}L_2$ as in Lemma \ref{lem:expmatrixnorm}, so 
	\begin{align*}
	\|e^{L_0+\frac{1}{N}L_1+\frac{1}{N^2}L_2}-e^{L_0}\|_{\End(V)}&\leq\frac{1}{N}\| L_1\|_{\End(V)} e^{\|L_0\|_{\End (V)}}e^{\left\|\frac{1}{N}L_1+\frac{1}{N^2}L_2\right\|_{\End(V)}}\\
	&\leq\frac{1}{N}\| L_1\|_{\End(V)} e^{\|L_0\|_{\End (V)}}e^{\frac{1}{N}\left\|L_1\right\|_{\End(V)}+\frac{1}{N^2}\left\|L_2\right\|_{\End(V)}}.
	\end{align*}
	Choose $N_0\in\N$ such that $e^{\frac{1}{N_0}\left\|L_1\right\|_{\End(V)}+\frac{1}{N_0^2}\left\|L_2\right\|_{\End(V)}}\leq 2$. Then (\ref{eq:expopnorm}) holds for all $N\geq N_0$ by setting $C_1=2\|L_1\|_{\End(V)}e^{\|L_0\|_{\End(V)}}$. Equation (\ref{eq:expfuncnorm}) then follows. The proofs of \eqref{eq:expopnorm2}, \eqref{eq:expopnorm3}, and subsequently \eqref{eq:expfuncnorm2} and \eqref{eq:expfuncnorm3}, are similar. 
\end{proof}

We now introduce notation and establish a few preliminary results which we will use for proving Theorem \ref{thm:freeSBlimit} for $\SO(N,\R)$ in the next section.

\begin{definition}
	Define a family of holomorphic functions $\{\nu_k\}_{k\in\Z}$ on $\C$ by setting $\nu_0(\tau)\equiv 1$ and for $k\neq 0$, 
	\begin{equation} \label{eq:nuk}
	\nu_k(\tau)=e^{-\frac{|k|}{2}\tau}\sum_{j=0}^{|k|-1}\frac{(-\tau)^j}{j!}|k|^{j-1}\binom{|k|}{j+1}.
	\end{equation}
	For $\tau\in\C$, define the {\bf trace evaluation map} $\pi_{\tau}:\Cuuv\to\C[u,u^{-1}]$ by
	\begin{equation}
	(\pi_\tau P)(u)=P(u;\mathbf{v})|_{v_k=\nu_k(\tau),k\neq 0}.
	\end{equation}
\end{definition}

For a concrete example, if $P(u;\mathbf{v})=v_3v_{-7}^2u^6+9v_1v_{-4}^5$, then $$(\pi_\tau P)(u)=\nu_3(\tau)\nu_{-7}(\tau)^2u^6+9\nu_1(\tau)\nu_{-4}(\tau)^5.$$

The following result is a key tool in proving our main limit theorems.

\begin{theorem}[{Biane, \cite{Biane:freeBm}}] \label{thm:biane}
	For each $s>0$ and $k\in\Z$, $$\lim_{N\to\infty}\int_{\U(N)}\tr (U^k)\, \rho_s^{\U(N)}(U)\, dU=\nu_k(s).$$ 
\end{theorem}

For $P\in\Cuuv$, let $P(u;\mathbf{1})=P(u;\mathbf{v})_{v_k=1,k\neq 0}$. Using the intertwining formula (\ref{eq:intertwineUNexp}) for the $\U(N)$ case and Theorem \ref{thm:biane}, we have
\begin{align}
\nu_k(s)&=\lim_{N\to\infty}\int_{\U(N)}\tr (U^k)\, \rho_s^{\U(N)}(U)\, dU=\lim_{N\to\infty}(e^{\frac{s}{2}\Delta_{\U(N)}}\tr[(\cdot)^k])(I_N)\\
&=\lim_{N\to\infty}(e^{\frac{s}{2}(\mathcal{L}_0-\frac{1}{N^2}(2\mathcal{K}_1^-+\mathcal{K}_2^-))}v_k)(\mathbf{1}).
\end{align}
Since $\mathcal{L}_0$, $\mathcal{K}_1^-$, and $\mathcal{K}_2^-$ are linear operators which preserve the finite-dimensional vector space $\C_k[u,u^{-1};\mathbf{v}]$, we consider $(e^{\frac{s}{2}(\mathcal{L}_0-\frac{1}{N^2}(2\mathcal{K}_1^-+\mathcal{K}_2^-))}v_k)(\mathbf{1})$ to be the evaluation of the linear functional $\varphi(Q):=Q(\mathbf{1})$ on the finite-dimensional space $\C_m[\mathbf{v}]$. Hence Corollary \ref{cor:expmatrixnorm} applies, so taking the limit as $N\to\infty$, we have
\begin{equation} \label{eq:bianev2}
\nu_k(s)=(e^{\frac{s}{2}\mathcal{L}_0}v_k)(\mathbf{1}). 
\end{equation}

This yields the following version of Theorem \ref{thm:biane} for the $\SO(N,\R)$ case, which is due to L\'{e}vy \cite{Levy:master}. 

\begin{lemma}[L\'{e}vy, \cite{Levy:master}] \label{lem:bianeSON}
	For $s>0$ and $k\in\Z$, 
	\begin{equation}
	\lim_{N\to\infty}\int_{\SO(N,\R)}\tr (A^k)\, \rho_s^{\SO(N,\R)}(A)\, dA=\nu_k\left(s\right).
	\end{equation}
\end{lemma}

\begin{proof}
	We have
	\begin{align*}
	\lim_{N\to\infty}\int_{\SO(N,\R)}\tr (A^k)\, \rho_s^{\SO(N,\R)}(A)\, dA&=\lim_{N\to\infty}(e^{\frac{s}{2}\Delta_{\SO(N,\R)}}\tr[(\cdot)^k])(I_N)\\
	&=\lim_{N\to\infty}(e^{\frac{s}{2}(\mathcal{L}_0+\frac{1}{N}\mathcal{L}_1^{(1)}+\frac{1}{N^2}\mathcal{L}_2^{(1)})}v_k)(\mathbf{1})=(e^{\frac{s}{2}\mathcal{L}_0}v_k)(\mathbf{1})\\
	&=\nu_k\left(s\right),
	\end{align*}
	where we have again applied Corollary \ref{cor:expmatrixnorm} as in (\ref{eq:bianev2}). 
\end{proof}

\begin{lemma} \label{lem:trevalQ} 
	For $Q\in\C [\mathbf{v}]$ and $\tau\in\C$, we have
	\begin{equation}
	\left(e^{\frac{\tau}{2}\mathcal{L}_0}Q\right)(\mathbf{1})=\pi_{\tau}Q.
	\end{equation}
\end{lemma}

\begin{proof}
	By Theorem \ref{thm:partialprodrule}, $e^{\frac{\tau}{2}\mathcal{L}_0}$ is a homomorphism of $\C[\mathbf{v}]$. Hence it suffices to show that
	\begin{equation} \label{eq:trevalQ}
	\left(e^{\tau\mathcal{L}_0}v_k\right)(\mathbf{1})=\pi_{\tau}(v_k)=\nu_k\left(\tau\right)
	\end{equation}
	for each $k\in\Z$. To see that it holds for all $\tau\in\C$, we expand the left-hand side as a power series in $\tau$. Fix a norm $\|\cdot\|_{W_k}$ on the finite-dimensional vector space $W_k:=\C_k[u,u^{-1};\mathbf{v}]$. Let $\|\cdot\|_{\End (W_k)}$ denote the operator norm on $\End(W_k)$ induced by $\|\cdot\|_{W_k}$, and let $\|\cdot\|_{W_k^*}$ denote the corresponding dual norm on $W_k$. Let $\varphi$ be the linear functional acting on $W_k$ defined by $\varphi(P)=P(\mathbf{1})$. Then for any $\tau\in\C$, 
	\begin{align*}
	\left|\left(e^{\frac{\tau}{2}\mathcal{L}_0}v_k\right)(\mathbf{1})\right|&=\left|\sum_{n=0}^\infty\left( \frac{1}{n!}\left(\frac{\tau}{2}\right)^n(\mathcal{L}_0)^nv_k\right)(\mathbf{1})\right|=\left| \sum_{n=0}^\infty \frac{1}{n!}\left(\frac{\tau}{2}\right)^n\varphi((\mathcal{L}_0)^nv_k)\right|\\
	&\leq\sum_{n=0}^\infty\frac{1}{n!}\left(\frac{|\tau|}{2}\right)^n\|\varphi\|_{W_k^*}\|\mathcal{L}_0\|^n_{\End(W_k)}\|v_k\|_{W_k},
	\end{align*}
	and the right-hand side converges by the ratio test. Thus the function $\tau\mapsto\left(e^{\frac{\tau}{2}\mathcal{L}_0}v_k\right)(\mathbf{1})$ is analytic on the complex plane. By (\ref{eq:bianev2}), this function agrees with the entire function $\tau\mapsto\nu_k\left(\tau\right)$ for $\tau\in\R$, and so it also agrees for all $\tau\in\C$.
\end{proof}

\begin{lemma} \label{lem:treval2} 
	Let $s>0$ and $\tau=t+i\theta\in\mathbb{D}(s,s)$. For any $Q\in\C[\mathbf{v}]$, 
	\begin{gather}
	\left(e^{\wL_0^{s,\tau}}\iota(Q)\right)(\mathbf{1})=\pi_{s-\tau}Q,\\
	\left(e^{\wL_0^{s,\tau}}\iota ^*(Q)\right)(\mathbf{1})=\overline{\pi_{s-\tau}Q}. \label{eq:treval2star}
	\end{gather}
\end{lemma}

\begin{proof}
	First, observe that if $f:\SO(N,\C)\to\MNC$ is holomorphic, then $\widetilde{JX}f=i\widetilde{X}f$ for all $X\in\so(N,\R)$. Thus
	\begin{equation*}
	\ASONst f|_{\SO(N,\R)}=\sum_{X\in\bSON}\left[\left(s-\frac{t}{2}\right)\wX^2-\frac{t}{2}\wX^2-i\theta\wX^2\right]f=(s-\tau)\LSON f.
	\end{equation*}
	In particular, since $Q_N$ is a trace polynomial, 
	\begin{equation} \label{eq:expQN}
	e^{\frac{1}{2}\ASONst}Q_N=e^{\frac{1}{2}(s-\tau)\LSON}Q_N.
	\end{equation}
	Applying intertwining formulas (\ref{eq:intertwine2}) and (\ref{eq:intertwine1exp}) to the left side and using (\ref{eq:intertwineword}) gives
	\begin{equation}
	\left[e^{\wL_0^{s,\tau}+\frac{1}{N}\wL_1^{s,\tau}+\frac{1}{N^2}\wL_2^{s,\tau}}\iota(Q)\right]_N=[e^{\frac{1}{2}(s-\tau)\mathcal{D}_N^{(1)}}Q]_N.
	\end{equation}
	Hence for $Q\in\C[\mathbf{v}]$, 
	\begin{equation}
	\left(e^{\wL_0^{s,\tau}+\frac{1}{N}\wL_1^{s,\tau}+\frac{1}{N^2}\wL_2^{s,\tau}}\iota(Q)\right)(\mathbf{1})=\left(e^{\frac{1}{2}(s-\tau)(\mathcal{L}_0+\frac{1}{N}\mathcal{L}_1^{(1)}+\frac{1}{N^2}\mathcal{L}_2^{(1)})}Q\right)(\mathbf{1}).
	\end{equation}
	If $\deg Q=m$, then we can view the left and right sides as the evaluation of a linear functional $\varphi(P)=P(\mathbf{1})$ on the finite-dimensional spaces $\mathscr{W}_m$ and $\C_m[\mathbf{v}]$, respectively. Thus we may apply Corollary \ref{cor:expmatrixnorm} to take the limit as $N\to\infty$, which yields
	\begin{equation}
	\left(e^{\wL_0^{s,\tau}}\iota(Q)\right)(\mathbf{1})=\left(e^{\frac{1}{2}(s-\tau)\mathcal{L}_0}Q\right)(\mathbf{1}).
	\end{equation}
	By Lemma \ref{lem:trevalQ}, the expression on the right is $\pi_{s-\tau}Q$. Equation \eqref{eq:treval2star} is proved similarly, using the fact that if $f:\SO(N,\C)\to\M_N(\C)$ is antiholomorphic, then $\widetilde{JX}f=-i\widetilde{X}f$ for all $X\in\so(N,\R)$.
\end{proof}

\begin{theorem} \label{thm:bianeSONC}
	Let $s>0$, $\tau\in\mathbb{D}(s,s)$, and $k\in\Z$. Then
	\begin{equation}
	\lim_{N\to\infty}\int_{\SO(N,\C)}\tr(A^k)\,\mu_{s,\tau}^{\SO(N,\C)}(A)\, dA=\nu_k\left(s-\tau\right).
	\end{equation}
\end{theorem}

\begin{proof}
	We use (\ref{eq:heatKC}) and intertwining formula (\ref{eq:intertwine2}) to get
	\begin{align}
	\int_{\SO(N,\C)}\tr(A^k)\,\mu_{s,\tau}^{\SO(N,\C)}(A)\, dA&=\left(e^{\frac{1}{2}\ASONst}\tr[(\cdot)^k]\right)(I_N)\nonumber\\
	&=\left(e^{\wL_0^{s,\tau}+\frac{1}{N}\wL_1^{s,\tau}+\frac{1}{N^2}\wL_2^{s,\tau}}\iota(v_k)\right)(\mathbf{1}).\nonumber
	\end{align}
	Viewing the last expression as a linear functional acting on $\iota(v_k)\in\mathscr{W}_k$, we apply Corollary \ref{cor:expmatrixnorm} to take the limit as $N\to\infty$. Combining this with Lemma \ref{lem:treval2}, we have
	\begin{align}
	\lim_{N\to\infty}\int_{\SO(N,\C)}\tr(A^k)\,\mu_{s,\tau}^{\SO(N,\C)}(A)\, dA&=\left(e^{\wL_0^{s,\tau}}\iota(v_k)\right)(\mathbf{1})=\pi_{s-\tau}v_k=\nu_k\left(s-\tau\right).\nonumber\qedhere
	\end{align}
\end{proof}

\subsection{The free Segal-Bargmann transform for $\SO(N,\R)$} \label{subsec:freeSBtransf}

In this section, we prove our main result on the large-$N$ limit of the Segal-Bargmann transform for $\SO(N,\R)$. The free Segal-Bargmann transform $\Gst$ and the free inverse Segal-Bargmann transform $\Hst$ appearing in Theorem \ref{thm:freeSBlimit} have explicit formulas, which we first describe.  

\begin{definition}
	Let $\mathcal{L}_0$ be the pseudodifferential operator on $\Cuuv$ from Definition \ref{def:Doperators}. We define the {\bf free Segal-Bargmann transform} to be the map $\Gst:\C[u,u^{-1}]\to\C[u,u^{-1}]$ given by
	\begin{equation}
	\Gst=\pi_{s-\tau}\circ e^{\frac{\tau}{2}\mathcal{L}_0}, \label{eq:Gst}
	\end{equation}
	and we define the {\bf free inverse Segal-Bargmann transform} to be the map $\Hst:\C[u,u^{-1}]\to\C[u,u^{-1}]$ given by
	\begin{equation}
	\Hst=\pi_{s}\circ e^{- \frac{\tau}{2}\mathcal{L}_0}. \label{eq:Hst}
	\end{equation}
\end{definition}

Before proving Theorem \ref{thm:freeSBlimit}, we require the following result. 

\begin{theorem} \label{thm:tracelimit} 
	Let $s>0$ and $\tau=t+i\theta\in\mathbb{D}(s,s)$. For any $P\in\Cuuv$, 
	\begin{align}
	&\|P_N-[\pi_{s}P]_N\|^2_{L^2(\rho_{s}^{\SO(N,\R)})}=O\left(\frac{1}{N^2}\right),\hspace{.5cm}\mbox{and} \label{eq:tracelimitSON}\\
	&\|P_N-[\pi_{s-\tau}P]_N\|^2_{L^2(\mu_{s,\tau}^{\SO(N,\C)})}=O\left(\frac{1}{N^2}\right). \label{eq:tracelimitSONC}
	\end{align}
\end{theorem}

This result shows that with respect to the heat kernel measure, the space of trace polynomials concentrates onto the space of polynomials as $N\to\infty$. 

\begin{proof}
	We first show (\ref{eq:tracelimitSONC}). It suffices to prove the result for polynomials of the form $P(u;\mathbf{v})=u^kQ(\mathbf{v})$ for $k\in\Z$ and $Q\in\C[\mathbf{v}]$; (\ref{eq:tracelimitSONC}) then holds on all of $\Cuuv$ by the triangle inequality. For such a polynomial, $$P(u;\mathbf{v})-\pi_{s-\tau} P(u;\mathbf{v})=u^k[Q(\mathbf{v})-\pi_{s-\tau} Q]=u^kR_{s-\tau},$$
	where we let $R_{s-\tau} =Q-\pi_{s-\tau} Q$. Hence for $A\in\SO(N,\C)$,
	\begin{align}
	\|P_N(A)-[\pi_{s-\tau} P]_N(A)\|^2_{\MNC}&=\tr\left(A^k[R_{s-\tau}]_N(A)[R_{s-\tau}]_N(A)^*A^{*k}\right) \nonumber\\
	&=\tr(A^kA^{*k})[R_{s-\tau}]_N(A)[R_{s-\tau}]_N(A)^* \nonumber\\
	&=[v_{\varepsilon(k,k)}\iota(R_{s-\tau})\iota^*(R_{s-\tau})]_N(A). \label{eq:Rhst}
	\end{align}
	Along with intertwining formula (\ref{eq:intertwine2}), this allows us to compute
	\begin{align}
	\| P_N-[\pi_{s-\tau} P]_N\|^2_{L^2(\mu_{s,\tau}^{\SO(N,\C)})}&=e^{\frac{1}{2}\ASONst}\left(\| P_N-[\pi_{s-\tau} P]_N\|^2_{\MNC}\right)(I_N) \label{eq:APNist}\\
	&=\left(e^{\wL_0^{s,\tau}+\frac{1}{N}\wL_1^{s,\tau}+\frac{1}{N^2}\wL_2^{s,\tau}}(v_{\varepsilon(k,k)}\iota(R_{s-\tau})\iota^*(R_{s-\tau}))\right)(\mathbf{1})\label{eq:PNpist}
	\end{align}
	By the triangle inequality, the last line is bounded by
	\begin{align}
	&\left|\left(e^{\DNst}(v_{\varepsilon(k,k)}\iota(R_{s-\tau})\iota^*(R_{s-\tau}))\right)(\mathbf{1})-\left(e^{\LL}(v_{\varepsilon(k,k)}\iota(R_{s-\tau})\iota^*(R_{s-\tau}))\right)(\mathbf{1})\right| \label{eq:DNLzero}\\
	&\hspace{1cm}+\left|\left(e^{\wL_0^{s,\tau}+\frac{1}{N}\wL_1^{s,\tau}}(v_{\varepsilon(k,k)}\iota(R_{s-\tau})\iota^*(R_{s-\tau}))\right)(\mathbf{1})\right|.\label{eq:LLhom}
	\end{align}
	If $m=\deg v_{\varepsilon(k,k)}\iota(R_{s-\tau})\iota^*(R_{s-\tau})$, we can interpret \eqref{eq:DNLzero} as the evaluation of the linear functional $\varphi(R)=R(\mathbf{1})$ on $\mathscr{W}_{m}$, so by Corollary \ref{cor:expmatrixnorm}, 
	\begin{align} \label{eq:DNLzerobound}
	&\left|\left(e^{\DNst}(v_{\varepsilon(k,k)}\iota(R_{s-\tau})\iota^*(R_{s-\tau}))\right)(\mathbf{1})-\left(e^{\LL}(v_{\varepsilon(k,k)}\iota(R_{s-\tau})\iota^*(R_{s-\tau}))\right)(\mathbf{1})\right|\nonumber\\
	&\hspace{1cm}=O\left(\frac{1}{N^2}\right).
	\end{align}
	To bound \eqref{eq:LLhom}, recall from Theorem \ref{thm:intertwiningII} that  $\wL_0^{s,\tau}+\frac{1}{N}\wL_1^{s,\tau}$ is a first-order differential operator on $\mathscr{W}_{m}$, so $e^{\wL_0^{s,\tau}+\frac{1}{N}\wL_1^{s,\tau}}$ is an algebra homomorphism, i.e.
	\begin{align} 
	e^{\wL_0^{s,\tau}+\frac{1}{N}\wL_1^{s,\tau}}&(v_{\varepsilon(k,k)}\iota(R_{s-\tau})\iota^*(R_{s-\tau}))\nonumber\\
	&=e^\LL v_{\varepsilon(k,k)}\cdot e^\LL \iota(R_{s-\tau})\cdot e^\LL \iota^*(R_{s-\tau}).\label{eq:LLhomexpand}
	\end{align}
	For any $T\in\mathscr{W}$, 
	\begin{equation} \label{eq:LLhombound}
	\left|(e^{\LL}T)(\mathbf{1})\right|\leq \left| (e^\LL T)(\mathbf{1}) - (e^{\wL_0^{s,\tau}}T)(\mathbf{1})\right| + \left| (e^{\wL_0^{s,\tau}}T)(\mathbf{1})\right|
	\end{equation}
	By Corollary \ref{cor:expmatrixnorm}, the first term on the right side of \eqref{eq:LLhombound} is $O(1/N)$, while the second term is a constant not depending on $N$. Moreover, by Lemma \ref{lem:treval2}, 
	\begin{equation} \label{eq:treval2app}
	(e^{\wL_0^{s,\tau}}\iota(R_{s-\tau}))(\mathbf{1})=(e^{\wL_0^{s,\tau}}\iota^*(R_{s-\tau}))(\mathbf{1})=0. 
	\end{equation}
	Applying \eqref{eq:LLhombound} to each term in the product in \eqref{eq:LLhomexpand} and using \eqref{eq:treval2app}, we have that \eqref{eq:LLhom} is $O(1/N^2)$. Combining this with \eqref{eq:DNLzerobound} shows \eqref{eq:tracelimitSONC}. 
	
	Finally, observe that $\frac{s}{2}\LSON=\frac{1}{2}\mathcal{A}_{s,0}^{\SO(N,\C)}$. Setting $\tau=0$ in (\ref{eq:APNist}) and restricting to $\SO(N,\R)$ shows (\ref{eq:tracelimitSON}). 
\end{proof}

We can now prove Theorem \ref{thm:freeSBlimit} for $\SO(N,\R)$. 

\begin{proof}[Proof of Theorem \ref{thm:freeSBlimit} for $\SO(N,\R)$] \label{pf:thm:freeSBlimit}
	Let $f\in\C [u,u^{-1}]$. Using Theorem \ref{thm:SBintertwine} to rewrite $\BSONst f_N$ and applying the triangle inequality, we have
	\begin{align}
	\|\BSONst &f_N-[\Gst f]_N\|_\LmuSONst\nonumber\\
	&=\| [\exptDN f]_N-[\pi_{s-\tau}\circ\exptLz f]_N\|_\LmuSONst\nonumber\\
	&\leq \|[\exptDN f]_N-[\exptLz f]_N\|_\LmuSONst \\
	&\hspace{1.5cm}+\|[\exptLz f]_N-[\pi_{s-\tau}\circ\exptLz f]_N\|_\LmuSONst.\label{eq:SBlimittriangle}
	\end{align}
	By Theorem \ref{thm:tracelimit}, (\ref{eq:SBlimittriangle}) is $O(1/N)$, so it remains to show that
	\begin{equation} \label{eq:freeSBmaininequality}
	\|[\exptDN f]_N-[\exptLz f]_N\|^2_\LmuSONst =O\left(\frac{1}{N^2}\right).
	\end{equation}
	To this end, let $m=\deg f$ and $R^{(N)}=\exptDN f-\exptLz f\in\Cmuuv$. Recalling (\ref{eq:LtmuNstnorm}) and the sesquilinear form $\mathcal{B}$ of Lemma \ref{lem:sesquilinearform}, we have
	\begin{align}
	\|[\exptDN f]_N-[\exptLz f]_N\|^2_\LmuSONst &=\|[\RN]_N\|^2_\LmuSONst\nonumber\\
	&=e^{\frac{1}{2}\mathcal{A}^N_{s,\tau}}\left(\|[\RN]_N\|^2_{\MNC}\right)(I_N) \nonumber\\
	&=\left(e^{\DNst}\mathcal{B}(\RN,\RN)\right)(\mathbf{1}). \label{eq:RNLmunorm}
	\end{align}
	Fix norms $\|\cdot\|_{\mathscr{W}_{2m}}$ and $\|\cdot\|_{\Cmuuv}$ on $\mathscr{W}_{2m}$ and $\Cmuuv$, respectively. By Corollary \ref{cor:expmatrixnorm} applied to the linear functional $\varphi(P)=P(\mathbf{1})$ on $\mathscr{W}_{2m}$, there exists a constant $C=C(m,s,\tau)$ (not dependent on $N$) such that
	\begin{equation} \label{eq:DNLzbound}
	\left|\left(e^{\DNst}\mathcal{B}(\RN,\RN)\right)(\mathbf{1})-\left(e^{\Lzst}\mathcal{B}(\RN,\RN)\right)(\mathbf{1})\right|\leq\frac{C}{N}||\mathcal{B}(\RN,\RN)\|_{\mathscr{W}_{2m}}.
	\end{equation}
	Next, using the linear functional $\psi(P)=\left(\exptLz P)\right)(\mathbf{1})$ on $\mathscr{W}_{2m}$, we have
	\begin{equation} \label{eq:Lzbound}
	\left|\left(e^{\Lzst}\mathcal{B}(\RN,\RN)\right)(\mathbf{1})\right|\leq \|\psi\|_{\mathscr{W}_{2m}^*}\|\mathcal{B}(\RN,\RN)\|_{\mathscr{W}_{2m}}.
	\end{equation}
	Putting (\ref{eq:freeSBmaininequality}), (\ref{eq:DNLzbound}), and (\ref{eq:Lzbound}) together yields
	\begin{equation} \label{eq:DNLzboundII}
	\|[\exptDN f]_N-[\exptLz f]_N\|^2_\LmuSONst \leq \left(\|\psi\|_{\mathscr{W}_{2m}^*}+\frac{C}{N}\right)\|\mathcal{B}(\RN,\RN)\|_{\mathscr{W}_{2m}}.
	\end{equation}
	Since $\mathcal{B}:\Cmuuv\times\Cmuuv\to\mathscr{W}_{2m}$ is a sesquilinear form with finite domain and range, there exists a constant $C'=C'(m)$ (not dependent on $N$) such that
	\begin{equation} \label{eq:BPQbound}
	\|\mathcal{B}(P,Q)\|_{\mathscr{W}_{2m}}\leq C'\|P\|_{\Cmuuv}\|Q\|_{\Cmuuv}\hspace{.5cm}\mbox{for all }P,Q\in\Cmuuv.
	\end{equation}
	A final application of Corollary \ref{cor:expmatrixnorm} shows that there exists a constant $C''$ depending on $\mathcal{L}_0$, $\mathcal{L}_1^{(1)}$, $\mathcal{L}_2^{(1)}$, $\tau$, $\|\cdot\|_{\Cmuuv}$, and $f$, but again, not on $N$, such that 
	\begin{equation}
	\|\RN\|_{\Cmuuv}=\|\exptDN f-\exptLz f\|_{\Cmuuv}\leq\frac{C''}{N}.
	\end{equation}
	Hence
	\begin{equation} \label{eq:BRNRNbound}
	\|\mathcal{B}(\RN,\RN)\|_{\mathscr{W}_{2m}}\leq\|\RN\|_{\Cmuuv}^2\leq\frac{(C'')^2}{N^2}.
	\end{equation}
	Combining (\ref{eq:DNLzboundII}) and (\ref{eq:BRNRNbound}) proves (\ref{eq:freeSBmaininequality}). 
	
	To prove (\ref{eq:freeinverseSBlimit}), we observe that $(\BSONst)^{-1}f_N|_{\SO(N,\R)}=e^{-\frac{\tau}{2}\LSON}f_N$. By a similar triangle inequality argument using (\ref{eq:tracelimitSON}), it suffices to show
	\begin{equation} \label{eq:freeSBmaininequalityII}
	\|[e^{-\frac{\tau}{2}\mathcal{D}_N^{(1)}}f]_N-[e^{-\frac{\tau}{2}\mathcal{L}_0}f]_N\|^2_{L^2(\rho_s^{\SO(N,\R)})}=O\left(\frac{1}{N^2}\right).
	\end{equation}
	Replacing $\tau$ with $-\tau$ in the definition of $\RN$ and $(s,\tau)$ with $(s,0)$ from (\ref{eq:RNLmunorm}) onwards, the same argument as above proves (\ref{eq:freeSBmaininequalityII}). 
	
	For uniqueness, we first define seminorms on $\Cuuv$ by
	\begin{alignat}{2}
	&\|P\|_{s,N}^{(1)}=\|P_N\|_{L^2(\rho_s^{\SO(N,\R)})},\hspace{1cm}&&\|P\|_{s,N}^{(2)}=\|P_N\|_{L^2(\rho_s^{\U(N)})},\\
	&\|P\|_{s,\tau,N}^{(1)}=\|P_N\|_{L^2(\mu_{s,\tau}^{\SO(N,\C)})},\hspace{1cm}&&\|P\|_{s,\tau,N}^{(2)}=\|P_N\|_{L^2(\mu_{s,\tau}^{\GL_N(\C)})}.
	\end{alignat}
	In addition, for $\beta=1,2$, define the seminorms
	\begin{align}
	\|P\|_s^{(\beta)}&=\lim_{N\to\infty}\|P\|_{s,N}^{(\beta)} \label{eq:seminormlim1}\\
	\|P\|_{s,\tau}^{(\beta)}&=\lim_{N\to\infty}\|P\|_{s,\tau,N}^{(\beta)}\label{eq:seminormlim2}
	\end{align}
	In \cite[Lemma 4.7]{DHK:U(N)}, it was shown that for $\beta=2$, the seminorms \eqref{eq:seminormlim1} and \eqref{eq:seminormlim2} are norms when restricted to $\C[u,u^{-1}]$. Since
	\begin{align}
	\|P\|_s^{(1)}&=\|P\|^{(2)}_{s} \label{eq:seminorms}
	\end{align}
	it follows that the seminorm \eqref{eq:seminormlim1} is also a norm on $\C[u,u^{-1}]$ for $\beta=1$. An argument completely analogous to the one used in \cite[Lemma 4.7]{DHK:U(N)} to prove that $\|P\|_{s,\tau}^{(2)}$ is a norm on $\C[u,u^{-1}]$ shows that $\|P\|_{s,\tau}^{(1)}$ is a norm on $\C[u,u^{-1}]$. Hence, if $g_{s,\tau},g_{s,\tau}'\in\C[u,u^{-1}]$ both satisfy
	\begin{equation*}
	\|\BSONst f_N-[g_{s,\tau}]_N\|^2_{L^2(\mu_{s,\tau}^{\SO(N,\C)})}=O\left(\frac{1}{N^2}\right)=\|\BSONst f_N-[g_{s,\tau}']_N\|^2_{L^2(\mu_{s,\tau}^{\SO(N,\C)})},
	\end{equation*}
	the triangle inequality implies that
	\begin{equation}
	\|g_{s,\tau}-g_{s,\tau}'\|_{L^2(\mu_{s,\tau}^{\SO(N,\C)})}=O\left(\frac{1}{N^2}\right).
	\end{equation}
	Taking $N\to\infty$ shows that $\|g_{s,\tau}-g_{s,\tau}'\|_{s,\tau}^{(1)}=0$. Since $\|\cdot\|_{s,\tau}^{(1)}$ is a norm on $\C[u,u^{-1}]$, it follows that $g_{s,\tau}=g_{s,\tau}'$. A similar argument shows that $\Hst f$ is the unique polynomial satisfying \eqref{eq:freeinverseSBlimit}. 
\end{proof} 

\begin{example} \label{ex:BSONcomputation}
	Set $P(u;\mathbf{v})=u^2\in\C[u,u^{-1}]$. Then we can show that
	\begin{align*}
	\mathbf{B}_{s,\tau}^{\SO(N,\R)}P_N(A)&=e^{\frac{\tau}{2}\Delta_{\SO(N,\R)}}P_N(A)\\
	&=
	\frac{1}{N}(1-e^{-\tau})I_N+\frac{1}{2}e^{-\tau}(1+e^{\frac{2\tau}{N}})A^2-\frac{N}{2}e^{-\tau}(-1+e^{\frac{2\tau}{N}})A\tr A,
	\end{align*}
	via the same method used in \cite[Example 3.5]{DHK:U(N)}. By the intertwining formula $e^{\frac{\tau}{2}\Delta_{\SO(N,\R)}}P_N=[e^{\frac{\tau}{2}\mathcal{D}_N^{(1)}}P]_N$, this corresponds to the trace polynomial 
	\begin{align*}
	(e^{\frac{\tau}{2}\mathcal{D}_N^{(1)}}P)(u;\mathbf{v})&=\frac{1}{N}(1-e^{-\tau})+\frac{1}{2}e^{-\tau}(1+e^{\frac{2\tau}{N}})u^2-\frac{N}{2}e^{-\tau}(-1+e^{\frac{2\tau}{N}})uv_1\\
	&=e^{-\tau}(u^2-\tau uv_1)+O\left(\frac{1}{N^2}\right).
	\end{align*}
	By Theorem \ref{thm:tracelimit}, we compute $\mathscr{G}_{s,\tau}P$ by evaluating the traces in $e^{\frac{\tau}{2}\mathcal{D}_N^{(1)}}P$ at the moments $\nu_{s-\tau}$ and taking the large-$N$ limit of the resulting polynomial. Since $\nu_1(s-\tau)=e^{-(s-\tau)/2}$, this yields
	\begin{equation*}
	\mathscr{G}_{s,\tau}P(u)=e^{-\tau}(u^2-\tau e^{-(s-\tau)/2}u).
	\end{equation*}
	In \cite[Example 3.5]{DHK:U(N)}, it was shown that for the unitary case (with $\tau=t>0$), 
	\begin{align*}
	(e^{\frac{t}{2}\mathcal{D}_N^{(2')}}P)(u;\mathbf{v})&=e^{-t}\cosh (t/N)u^2-Ne^{-t}\sinh (t/N)uv_1\\
	&=e^{-\tau}(u^2-\tau uv_1)+O\left(\frac{1}{N^2}\right),
	\end{align*}
	which also results in the same expression for $\mathscr{G}_{s,t}P$. 
\end{example}


\section{The Segal-Bargmann transform on $\Sp(N)$} \label{sec:SBtransformSpN}

In this section, we prove the $\Sp(N)$ version of Theorems \ref{thm:intertwine1}, \ref{thm:SBintertwine}, and \ref{thm:freeSBlimit}. which comprise our main results. We begin with a slight digression and show that $\Sp(N)$ can also be realized as a subset of $N\times N$ quaternion matrices. This allows us to easily identify a particular orthonormal basis of $\symp(N)\subseteq\M_{2N}(\C)$ which we use to compute a set of magic formulas used in the proofs of our intertwining formulas for $e^{\frac{\tau}{2}\Delta_{\Sp(N)}}$ and $e^{\frac{1}{2}\mathcal{A}_{s,\tau}^{\Sp(N,\C)}}$. These magic formulas share key similarities with the magic formulas for $\SO(N,\R)$. Consequently, many of the results for the $\Sp(N)$ case can proven using techniques similar to those used in the $\SO(N,\R)$ case, and so we do not always provide the full details. We conclude by showing that when $\Sp(N)$ is realized as a subset of $M_N(\HH)$, a version of the magic formulas and intertwining formulas for $\Delta_{\Sp(N)}$ also hold. 

\subsection{Two realizations of the compact symplectic group} \label{subsec:SpNrealizations}

The {\bf quaternion algebra} $\HH$ is the four-dimensional associative algebra over $\R$ spanned by $\idH$, $\ii$, $\jj$, and $\kk$ satisfying the relations $$\ii^2=\jj^2=\kk^2=-\idH$$ and 
\begin{alignat*}{2}
\ii\jj&=\kk,\hspace{1cm}\jj\ii&&=-\kk,\\
\jj\kk&=\ii,\hspace{1cm}\kk\jj&&=-\ii,\\
\kk\ii&=\jj,\hspace{1cm}\ii\kk&&=-\jj.
\end{alignat*}

We denote the quaternion conjugate of $q\in\HH$ by $q^*$. That is, if $q=a\idH+b\ii +c\jj+d\kk$, then $q^*=a\idH-b\ii -c\jj-d\kk$. If $A\in \M_N(\HH)$, we define the adjoint of $A$ to be the matrix $A^*$ defined by $(A^*)_{i,j}=(A_{j,i})^*$. 
For each $N\geq 1$, we define $\wSpN$ as
\begin{equation}
\wSpN=\{A\in M_N(\HH)\, :\, A^*A=I_N\}. \label{def:wSpN}
\end{equation}

Its Lie algebra is
\begin{equation}
\wsympN=\{X\in\M_N(\HH)\, :\, X^*=-X\},\label{eq:inproSpNquat}
\end{equation}
which we endow with the real inner product
\begin{equation} \label{eq:innerprodwSpN}
\inpro{X,Y}_{\wsympN}=2N\Real\Tr (X^*Y)\hspace{.5cm}\mbox{for all }X,Y\in\wsympN.
\end{equation}
Let $$C_N=\left\{\frac{1}{\sqrt{4N}}(E_{a,b}+E_{b,a})\right\}_{1\leq a<b\leq N}\cup\left\{\frac{1}{\sqrt{2N}} E_{a,a}\right\}_{1\leq a\leq N},$$
where $E_{a,b}\subseteq M_N(\HH)$ is the $N\times N$ matrix with a $1$ in the $(a,b)$ entry and zeros elsewhere. An orthonormal basis for $\wsympN$ with respect to this inner product is
\begin{equation} \label{eq:ONBspN}
\beta_{\wsympN}=\left\{\frac{1}{\sqrt{4N}}(E_{a,b}-E_{b,a})\right\}_{1\leq a<b\leq N}\cup \ii C_N\cup \jj C_N\cup\kk C_N.
\end{equation}

From \eqref{def:wSpN}, we see that $\wSpN$ is the unitary group over the quaternions. We now explicitly construct a (real) Lie group isomorphism between $\wSpN$ and $\Sp(N)$. While $\wSpN$ is, in some sense, the more natural realization of the compact symplectic group, for our purposes it is easier to work with the definition given in \eqref{def:SpN}. In particular, the complexification of the compact symplectic group is more readily understood when $\Sp(N)$ is realized as a subset of $\M_{2N}(\C)$ rather than of $\M_N(\HH)$. 

The material presented below is standard (cf. \cite{Stillwell:naive}); we include it here for convenience. 

First, observe that we can realize $\HH$ as a subalgebra of $M_2(\C)$. We do so by letting $\psi:\HH\to\M_2(\C)$ denote the injective homomorphism 
\begin{equation} \label{eq:quaternion.matrix.abcd}
\HH\ni q=a\idH +b\ii + c\jj +d\kk\hspace{.3cm}\mapsto\hspace{.3cm}\left[
\begin{array}{rr}
a+di & -b-ci \\
b-ci & a-di \\
\end{array}
\right]\in\M_{2}(\C).
\end{equation}
The map $\psi$ preserves conjugation: $\psi(q^*)=\psi(q)^*$ for all $q\in\HH$, where we take the quaternion conjugate on the left and the complex conjugate transpose on the right. Moreover, note that we can write a matrix of the form (\ref{eq:quaternion.matrix.abcd}) as
\begin{equation} \label{eq:quaternion.matrix.alpha.beta}
\left[
\begin{array}{rr}
\alpha & -\beta \\
\overline{\beta} & \overline{\alpha} \\
\end{array}
\right]
\end{equation}
for some $\alpha,\beta\in\C$, and any matrix of this form corresponds to a $q\in\HH$. 

Now define a map $\Phi:\M_N(\HH)\to\M_{2N}(\C)$ by $$\Phi([q_{i,j}]_{1\leq i,j\leq N})=[\psi(q_{i,j})]_{1\leq i,j\leq N}.$$ By the properties of block multiplication, $\Phi$ is an algebra homomorphism. Since $\psi$ is one-to-one, it follows that $\Phi$ is also one-to-one. Moreover, since $\psi$ preserves conjugation, so does $\Phi$: if $A\in\M_N(\HH)$, then $\Phi(A^*)=\Phi(A)^*$, where again we take the quaternion conjugate transpose on the left and the complex conjugate transpose on the right. 

Consequently, if $A\in\wSpN\subseteq\M_N(\HH)$, so that $A^*A=I_N$, then $$\Phi(A)^*\Phi(A)=\Phi(A^*)\Phi(A)=\Phi(A^*A)=I_{2N}.$$ Hence
\begin{equation}
\Phi(\wSpN)=\{Z\in M_{2N}(\C)\, :\, \mbox{$Z^*Z=I_{2N}$ and $\exists$ $A\in\M_N(\HH)$ such that $\Phi(A)=Z$}\}. \label{eq:PhiSpNimage}
\end{equation}
We would like to show that $\Phi(\wSpN)=\Sp(N)$, which would imply that $\Sp(N)\cong\wSpN$ as groups. Recall that since $A\in\Sp(N)$ is unitary and $\Omega^{-1}=-\Omega$, we can rewrite the definition of $\Sp(N)$ as follows:
\begin{equation}
\Sp(N)=\{Z\in M_{2N}(\C)\, :\, Z^*Z=I_{2N},-\Omega Z\Omega=\overline{Z}\}. \label{eq:SpNdefinition}
\end{equation}

Comparing \eqref{eq:PhiSpNimage} and \ref{eq:SpNdefinition}, we see that it suffices to show that
\begin{equation}
Z=\Phi(A)\mbox{ for some }A\in\M_N(\HH)\hspace{.3cm}\Leftrightarrow \hspace{.3cm}-\Omega Z\Omega=\overline{Z}.
\end{equation}
We need the following lemma. 
\begin{lemma} \label{lem:complexquaternion}
	Let $B\in M_2(\C)$. Then $-\Omega_0B\Omega_0=\overline{B}$ if and only if $B=\left[
	\begin{array}{rr}
	\alpha & -\beta \\
	\overline{\beta} & \overline{\alpha} \\
	\end{array}
	\right]$ for some $\alpha,\beta\in\C$. 
\end{lemma}

\begin{proof}
	Direct computation. 
\end{proof}

Consequently, for $Z\in M_{2N}(\C)$,
\begin{equation}
[-\Omega Z\Omega]^{i,j}=\sum_{1\leq k,\ell\leq N}[\Omega]^{i,k}[Z]_{k,\ell}[\Omega]^{\ell,j}=-\Omega_0[Z]^{i,j}\Omega_0,
\end{equation}
where $[Z]^{i,j}$ denotes the $(i,j)$th $2\times 2$ block of $Z$. By Lemma \ref{lem:complexquaternion}, we see that for each $1\leq i,j\leq N$, there exists $q_{i,j}\in\HH$ such that  $\psi(q_{i,j})=[Z]^{i,j}$ precisely when $-\Omega Z\Omega=\overline{Z}$. 

\begin{theorem}
	The map $\Phi|_{\wSpN}:\wSpN\to \Sp(N)$ is a Lie group isomorphism. 
\end{theorem}

\begin{proof}
	The preceding discussion shows that $\Phi|_{\wSpN}$ is a group isomorphism. It remains to be shown that $\Phi|_{\wSpN}$ and $\Phi|_{\wSpN}^{-1}$ are smooth, but this is clear from the definition of $\Phi$. 
\end{proof}

The inner products \eqref{eq:inproSpNcomplex} and \eqref{eq:innerprodwSpN} for $\symp(N)$ and $\wsympN$, resp., are related by $\Phi$. For $X\in\M_N(\HH)$, note that
\begin{equation}
\Real\Tr X=\frac{1}{2}\Real\Tr\Phi(X)=\frac{1}{2}\Tr\Phi(X).
\end{equation}
Hence for $X,Y\in\wsympN$, 
\begin{align}
\inpro{X,Y}_{\wsympN}&=2N\Real\Tr(XY^*)=N\Tr(\Phi(XY^*))\nonumber\\
&=N\Tr(\Phi(X)\Phi(Y)^*)=\inpro{\Phi(X),\Phi(Y)}_{\symp(N)}.\label{eq:PhiTrintertwine}
\end{align}
In particular, if $\beta_{\wsympN}$ is an orthonormal basis for $\wsympN\subseteq\M_N(\HH)$ with respect to \eqref{eq:inproSpNquat}, then $\Phi(\beta_{\wsympN})$ is an orthonormal basis for $\symp(N)\subseteq\M_{2N}(\C)$ with respect to \eqref{eq:inproSpNcomplex}.

\subsection{Magic formulas and derivative formulas}

\begin{prop}[Magic formulas for $\Sp(N)$] \label{prop:magicSpNv2}
	Let $\beta_{\symp(N)}$ be any orthonormal basis for $\symp(N)$ with respect to the inner product \eqref{eq:inproSpNcomplex}. For any $A,B\in\M_{2N}(\C)$, 
	\begin{align}
	\sum_{X\in\beta_{\symp(N)}} X^2&=-\frac{2N+1}{2N}I_{2N}=-I_{2N}-\frac{1}{2N}I_{2N} \label{eq:magicSpN1v2}\\
	\sum_{X\in\beta_{\symp(N)}} XAX&=-\frac{1}{2N}\Omega A^\intercal\Omega^{-1}-\widetilde{\tr}(A)I_{2N} \label{eq:magicSpN2v2}\\
	\sum_{X\in\beta_{\symp(N)}} \widetilde{\tr}(XA)X&=\frac{1}{4N^2}(\Omega A^\intercal \Omega^{-1}-A)\label{eq:magicSpN3v2}\\
	\sum_{X\in\beta_{\symp(N)}}\widetilde{\tr}(XA)\widetilde{\tr}(XB)&=\frac{1}{4N^2}(\widetilde{\tr}(\Omega A^\intercal\Omega B)-\widetilde{\tr}(AB))\label{eq:magicSpN4v2}
	\end{align}
\end{prop}

\begin{proof}
	As with the magic formulas for $\SO(N,\R)$, the quantities on the left hand side of (\ref{eq:magicSpN1v2}), (\ref{eq:magicSpN2v2}), (\ref{eq:magicSpN3v2}), and (\ref{eq:magicSpN4v2}) are independent of choice of orthonormal basis. Thus we may compute with respect to the orthonormal basis $\beta_{\symp(N)}:=\Phi(\beta_{\wsympN})$, where $\beta_{\wsympN}$ is the orthonormal basis described in \eqref{eq:ONBspN}. 
	
	We will use the following notation: for $A\in\M_{2N}(\C)$, let $[A]^{i,j}$ denote the $(i,j)$th $2\times 2$ submatrix, and let $A^{i,j}_{k,\ell}$ denote the $(k,\ell)$th entry of $[A]^{i,j}$. If $E_{a,b}\subseteq\M_N(\HH)$ is an elementary matrix and $\gamma\in\{\idH,\ii, \jj,\kk\}\subseteq\HH$, 
	\begin{equation}
	F_{a,b}^\gamma:=\Phi(\gamma E_{a,b})\in M_{2N}(\C).
	\end{equation}
	Observe that $[\Fabg]^{i,j}$ contains all zeros unless $(i,j)=(a,b)$, in which case it is equal to one of $$\psi(\idH)=\left[
	\begin{array}{rr}
	1 & 0 \\
	0 & 1 \\
	\end{array}
	\right], \hspace{.15cm}\psi(\ii)=\left[
	\begin{array}{rr}
	0 & -1 \\
	1 & 0 \\
	\end{array}
	\right],\hspace{.15cm}\psi(\jj)=\left[
	\begin{array}{rr}
	0 & -i \\
	-i & 0 \\
	\end{array}
	\right],\hspace{.15cm}\psi(\kk)=\left[
	\begin{array}{rr}
	i & 0 \\
	0 & -i \\
	\end{array}
	\right],$$ depending on whether $\gamma$ equals $\idH, \ii,\jj$, or $\kk$, respectively.
	
	Equation (\ref{eq:magicSpN1v2}) follows from (\ref{eq:magicSpN2v2}). For \eqref{eq:magicSpN2v2}, we have
	\begin{align} 
	\sum_{X\in\beta_{\symp(N)}}XAX&=\frac{1}{4N}\sum_{1\leq a,b\leq N}\sum_{\gamma\in\{\idH,\ii,\jj,\kk\}} \Fabg A\Fabg\nonumber\\
	&\hspace{.5cm}-\frac{1}{4N}\sum_{1\leq a,b\leq N}\left(\FabidH A\FbaidH -\sum_{\gamma\in\{\ii,\jj,\kk\}}\Fabg A\Fbag\right) \label{eq:XAXSpN1}
	\end{align}
	For $1\leq a,b,c,d,i,j\leq N$, we have
	\begin{align}
	[\Fabg A F_{c,d}^\gamma]^{i,j}&=\sum_{1\leq k,\ell\leq N}[\Fabg]^{i,k}[A]^{k,\ell}[F_{c,d}^\gamma]^{\ell,j}=[\Fabg]^{i,b}[A]^{b,c}[F_{c,d}^\gamma]^{c,j}\\
	&=\begin{cases}
	0_2&\mbox{$i\neq a$ or $j\neq d$}\\
	\psi(\gamma)[A]^{b,c}\psi(\gamma)&\mbox{$i=a$ and $j=d$}
	\end{cases}
	\end{align}
	
	A straightforward computation shows that
	\begin{align}
	&\psi(\idH)[A]^{b,c}\psi(\idH)=\left[
	\begin{array}{rr}
	A^{b,c}_{1,1} & A^{b,c}_{1,2} \\
	A^{b,c}_{2,1} & A^{b,c}_{2,2} \\
	\end{array}
	\right],\hspace{.3cm}\psi(\ii)[A]^{b,c}\psi(\ii)=\left[
	\begin{array}{rr}
	-A^{b,c}_{2,2} & A^{b,c}_{2,1} \\
	A^{b,c}_{1,2} & -A^{b,c}_{1,1} \\
	\end{array}
	\right]\\
	&\psi(\jj)[A]^{b,c}\psi(\jj)=\left[
	\begin{array}{rr}
	-A^{b,c}_{2,2} & -A^{b,c}_{2,1} \\
	-A^{b,c}_{1,2} & -A^{b,c}_{1,1} \\
	\end{array}
	\right],\hspace{.3cm}\psi(\kk)[A]^{b,c}\psi(\kk)\left[
	\begin{array}{rr}
	-A^{b,c}_{1,1} & A^{b,c}_{1,2} \\
	A^{b,c}_{2,1} & -A^{b,c}_{2,2} \\
	\end{array}
	\right]
	\end{align}
	Hence for the first term on the right hand side of \eqref{eq:XAXSpN1}, 
	\begin{align}
	\left[\frac{1}{4N}\sum_{1\leq a,b\leq N}\sum_{\gamma\in\{\idH,\ii,\jj,\kk\}} \Fabg A\Fabg\right]^{i,j}&=\frac{1}{4N}\left[\sum_{\gamma\in\{\idH,\ii,\jj,\kk\}}F_{i,j}^\gamma AF_{i,j}^\gamma\right]^{i,j}\nonumber\\
	&=\frac{1}{2N}\left[
	\begin{array}{rr}
	-A^{j,i}_{2,2} & A^{j,i}_{1,2} \\
	A^{j,i}_{2,1} & -A^{j,i}_{1,1} \\
	\end{array}
	\right].
	\end{align}
	For the second term on the right hand side of \eqref{eq:XAXSpN1}, the $(i,j)$th block is $0_2$ for $i\neq j$. For $i=j$, 
	\begin{align}
	&\left[\frac{1}{4N}\sum_{1\leq a,b\leq N}\left(\FabidH A\FbaidH -\sum_{\gamma\in\{\ii,\jj,\kk\}}\Fabg A\Fbag\right)\right]^{i,j}\nonumber\\
	&\hspace{1cm}=
	\frac{1}{4N}\left[\sum_{b=1}^N\left(F_{i,b}^{\idH}AF_{b,i}^{\idH}-\sum_{\gamma\in\{\ii,\jj,\kk\}}F_{i,b}^\gamma AF_{b,i}^\gamma\right)\right]^{i,j}\\
	&\hspace{1cm}=\frac{1}{4N}\sum_{b=1}^N\left([A]^{b,b}-\sum_{\gamma\in\{\ii,\jj,\kk\}}\psi(\gamma)[A]^{b,b}\psi(\gamma)\right)\\
	&\hspace{1cm}=\frac{1}{2N}\sum_{b=1}^N\left[
	\begin{array}{rr}
	A^{b,b}_{1,1} + A^{b,b}_{2,2} & 0 \\
	0 & A^{b,b}_{1,1} + A^{b,b}_{2,2} \\
	\end{array}
	\right]\\
	&\hspace{1cm}=\frac{1}{2N}\Tr(A)I_2.
	\end{align}
	Putting this together, we have
	\begin{align}
	\left[\sum_{X\in\beta_{\symp(N)}}XAX\right]^{i,j}&=
	\frac{1}{2N}
	\begin{cases}
	\left[
	\begin{array}{rr}
	-A^{j,i}_{2,2} & A^{j,i}_{1,2} \\
	A^{j,i}_{2,1} & -A^{j,i}_{1,1} \\
	\end{array}
	\right]&i\neq j\\[12pt]
	\left[
	\begin{array}{rr}
	-A^{j,i}_{2,2} & A^{j,i}_{1,2} \\
	A^{j,i}_{2,1} & -A^{j,i}_{1,1} \\
	\end{array}
	\right]-\Tr(A)I_2&i=j.
	\end{cases} \label{eq:XAXSpN2}
	\end{align}
	On the other hand, we can compute
	\begin{align}
	\left[-\frac{1}{2N}\Omega A^\intercal\Omega^{-1} -\wtr (A)I_{2N}\right]^{i,j}&=-\frac{1}{2N}\begin{cases}
	\Omega_0 ([A]^{j,i})^\intercal\Omega_0^{-1}&i\neq j,\\
	\Omega_0 ([A]^{j,i})^\intercal\Omega_0^{-1}-\Tr(A)I_2&i=j.
	\end{cases}\label{eq:XAXSpN3}
	\end{align}
	Multiplying out \eqref{eq:XAXSpN3} and comparing with \eqref{eq:XAXSpN2} proves \eqref{eq:magicSpN2v2}. 
	
	Next, for \eqref{eq:magicSpN3v2}, we again use the basis $\beta_{\symp(N)}$ to compute
	\begin{align} 
	\sum_{X\in\beta_{\symp(N)}}\wtr(XA)X&=\frac{1}{4N}\sum_{1\leq a,b\leq N}\sum_{\gamma\in\{\idH,\ii,\jj,\kk\}} \wtr(\Fabg A)\Fabg \label{eq:magic3SpNv2calc1}\\
	&\hspace{.5cm}-\frac{1}{4N}\sum_{1\leq a,b\leq N}\left(\wtr(\FabidH A)\FbaidH -\sum_{\gamma\in\{\ii,\jj,\kk\}}\wtr(\Fabg A)\Fbag\right). \label{eq:magic3SpNv2calc2}
	\end{align} 
	For $1\leq a,b\leq N$, a simple calculation shows that 
	\begin{align}
	&\wtr(\FabidH A)=\frac{1}{2N}(A_{1,1}^{b,a}+A_{2,2}^{b,a}),\hspace{.5cm}\wtr(\Fab^{\ii}A)=\frac{1}{2N}(A_{1,2}^{b,a}-A_{2,1}^{b,a})\\
	&\wtr (\Fab^{\jj}A)=-\frac{i}{2N}(A_{1,2}^{b,a}+A_{2,1}^{b,a}),\hspace{.5cm}\wtr(\Fab^{\kk}A)=\frac{i}{2N}(A_{1,1}^{b,a}-A_{2,2}^{b,a}),
	\end{align}
	The $(i,j)$th block of the right hand side of \eqref{eq:magic3SpNv2calc1} is thus
	\begin{align}
	\left[\frac{1}{4N}\sum_{1\leq a,b\leq N}\sum_{\gamma\in\{\idH,\ii,\jj,\kk\}} \wtr(\Fabg A)\Fabg\right]^{i,j}&=\frac{1}{4N}\left[\wtr (F_{i,j}^\idH A)F_{i,j}^\idH +\sum_{\gamma\in\{\ii,\jj,\kk\}}\wtr (F_{i,j}^\gamma A)F_{i,j}^\gamma\right]^{i,j}\nonumber\\
	&=\frac{1}{4N^2}\left[
	\begin{array}{rr}
	A_{2,2}^{j,i} & -A^{j,i}_{1,2} \\
	-A^{j,i}_{2,1} & A^{j,i}_{1,1} \\
	\end{array}
	\right].
	\end{align}
	
	Similarly, the $(i,j)$th block of \eqref{eq:magic3SpNv2calc2} is
	\begin{align}
	&\left[\frac{1}{4N}\sum_{1\leq a,b\leq N}\left(\wtr(\FabidH A)\FbaidH -\sum_{\gamma\in\{\ii,\jj,\kk\}}\wtr(\Fabg A)\Fbag\right)\right]^{i,j}\nonumber\\
	&\hspace{1cm}=\frac{1}{4N}\left[\wtr (F_{j,i}^\idH A)F_{i,j}^\idH - \sum_{\gamma\in\{\ii,\jj,\kk\}}\wtr(F_{j,i}^\gamma A)F_{i,j}^\gamma\right]^{i,j}\nonumber\\
	&\hspace{1cm}=\frac{1}{4N^2}\left[
	\begin{array}{rr}
	A_{1,1}^{i,j} & A^{i,j}_{1,2} \\
	A^{i,j}_{2,1} & A^{i,j}_{2,2} \\
	\end{array}
	\right].
	\end{align}
	Combining our computations, 
	\begin{equation}
	\left[\sum_{X\in\beta_{\symp(N)}}\wtr (XA)X\right]^{i,j}=\frac{1}{4N^2}\left[
	\begin{array}{rr}
	A_{2,2}^{j,i}-A_{1,1}^{i,j} & -A^{j,i}_{1,2}-A_{1,2}^{i,j} \\
	-A^{j,i}_{2,1}-A_{i,j}^{i,j} & A^{j,i}_{1,1} -A_{2,2}^{i,j}\\
	\end{array}
	\right] \label{eq:trXAXSpNcalc}
	\end{equation}
	Again, we can expand
	\begin{align}
	\left[\frac{1}{4N^2}(\Omega A^\intercal\Omega^{-1}-A)\right]^{i,j}=\frac{1}{4N^2}[-\Omega_0([A]^{j,i})^\intercal \Omega_0^{-1}-A]^{i,j}
	\end{align}
	to see that this is equivalent to \eqref{eq:trXAXSpNcalc}. This proves \eqref{eq:magicSpN3v2}; multiplying on the right by $B\in M_{2N}(\C)$ and taking $\wtr$ proves \eqref{eq:magicSpN4v2}. 
\end{proof}

While Proposition \ref{prop:magicSpNv2} holds for all $A,B\in\M_{2N}(\C)$, we now restrict to the case in which $A$ and $B$ are in $\Sp(N)$ or $\Sp(N,\C)$. Note that for $A\in\Sp(N,\C)$, $\Omega A^\intercal\Omega^{-1}=A^{-1}$. 

The magic formulas for $\Sp(N)$ allow us to compute the following derivative formulas, as in the $\SO(N,\R)$ case. 

\begin{prop}[Derivative formulas for $\Sp(N)$] \label{prop:SpNderivatives}
	The following identities hold on $\Sp(N)$ and $\Sp(N,\C)$:
	\begin{gather}
	\LSpN A^m=\Ind_{m\geq 2}\left[-\frac{1}{N}\sum_{j=1}^{m-1}(m-j)A^{m-2j}-2\sum_{j=1}^{m-1}(m-j)\wtr(A^j)A^{m-j}\right]\nonumber\\
	+\left(-1-\frac{1}{2N}\right)mA^m,\hspace{.5cm}m\geq 0\label{eq:DAmposSpN}\\
	\LSpN A^m=\Ind_{m\leq -2}\left[-\frac{1}{N}\sum_{j=m+1}^{-1}(-m+j)A^{m-2j}-2\sum_{j=m+1}^{-1}(-m+j)\wtr(A^{j})A^{m-j}\right]\nonumber\\
	-\left(-1-\frac{1}{2N}\right)mA^m,\hspace{.5cm}m<0\label{eq:DAmnegSpN}\\
	\sum_{X\in\bSpN}\wX \wtr (A^m)\cdot\wX A^p=\frac{mp}{4N^2}(A^{p-m}-A^{p+m}),\hspace{.5cm}m,p\in\Z. \label{eq:cross.sum.spN}
	\end{gather}
\end{prop}

\begin{proof}
	For $m\geq 0$, we use \eqref{eq:XXAm} and magic formulas \eqref{eq:magic1SpN} and \eqref{eq:magic2SpN} to sum over all $X\in\beta_{\Sp(N)}$:
	\begin{align}
	\Delta_{\Sp(N)}A^m&=2\Ind_{m\geq 2}\sum_{X\in\bSON}\sum_{\substack{k,j\neq 0\\k+j+\ell =m}}A^kXA^jXA^\ell +\sum_{X\in\bSON}\sum_{j=1}^mA^jX^2A^{m-j}\\
	&=2\Ind_{m\geq 2}\sum_{\substack{k,j\neq 0\\k+j+\ell =m}}\left[-\frac{1}{2N}A^k(A^j)^* A^\ell-\wtr(A^j)A^{k+\ell}\right]-\frac{m(2N+1)}{2N}A^m\\
	&=\Ind_{m\geq 2}\left[-\frac{1}{N}\sum_{j=1}^{m-1}(m-j)A^{m-2j}-2\sum_{j=1}^{m-1}(m-j)\wtr(A^j)A^{m-j}\right]\nonumber\\
	&\hspace{.5cm}+\left(-1-\frac{1}{2N}\right)mA^m.
	\end{align}
	
	For $m<0$, we use \eqref{eq:XAmneg} and magic formulas \eqref{eq:magic1SpN} and \eqref{eq:magic2SpN} to sum over all $X\in\beta_{\Sp(N)}$:
	\begin{align*}
	\Delta_{\Sp(N)}A^m&=2\Ind_{m\geq 2}\sum_{X\in\bSON}\sum_{\substack{k,\ell<0,j\leq 0\\j+k+\ell =m}}A^kXA^jXA^\ell +\sum_{X\in\bSON}\sum_{j=m+1}^0A^jX^2A^{m-j}\\
	&=2\Ind_{m\geq 2}\sum_{\substack{k,\ell<0,j\leq 0\\j+k+\ell =m}}\left[-\frac{1}{2N}A^k(A^j)^* A^\ell-\wtr(A^j)A^{k+\ell}\right]+\frac{m(2N+1)}{2N}A^m\\
	&=\Ind_{m\leq -2}\left[-\frac{1}{N}\sum_{j=m+1}^{-1}(-m+j)A^{m-2j}-2\sum_{j=m+1}^{-1}(-m+j)\wtr(A^{j})A^{m-j}\right]\\
	&\hspace{.5cm}-\left(-1-\frac{1}{2N}\right)mA^m.
	\end{align*} 
	
	Finally, comparing the magic formulas \eqref{eq:magic3} and \eqref{eq:magicSpN3v2}, we see that the proof of \eqref{eq:cross.sum.spN} is identical to that of \eqref{eq:XtrAmXAp}. 
\end{proof}

\subsection{Intertwining formulas for $\Delta_{\Sp(N)}$}

The derivative formulas allow us to prove the intertwining formula for $\Delta_{\Sp(N)}$. 

\begin{proof}[Proof of Theorem \ref{thm:intertwine1} for $\Sp(N)$] \label{pf:thm:intertwine1SpN}
	We proceed as in the proof of Theorem \ref{thm:intertwine1} for the $\SO(N,\R)$ case. Again, it suffices to show that $\Delta_{\Sp(N)}P_N=[\mathcal{D}_N^{(4)}P]_N$ holds for $P(u;\mathbf{v})=u^mq(\mathbf{v})$, where $m\in\Z\setminus\{0\}$ and $q\in\C[\mathbf{v}]$. Then $P_N=W_m\cdot q(\mathbf{V})$, and by the product rule, 
	\begin{align}
	\Delta_{\Sp(N)}P_N&=(\Delta_{\Sp(N)}W_m)\cdot q(\V)+2\sum_{X\in\beta_{\mathfrak{sp}(N)}}\wX W_m\cdot \wX q(\V)+W_m\cdot (\Delta_{\Sp(N)}q(\V)) \label{eq:LaplacianPNSpN}
	\end{align}
	
	For the first term in (\ref{eq:LaplacianPNSpN}), we consider the cases $m\geq 0$ and $m<0$ separately. For $m\geq 0$, we apply (\ref{eq:DAmposSpN}) to get
	\begin{align}
	(\Delta_{\Sp(N)} W_m)\cdot q(\V)&=-\frac{m(2N+1)}{2N}W_m\cdot q(\V)\nonumber\\
	&\hspace{.5cm}+\Ind_{m\geq 2}\bigg[-\frac{1}{N}\sum_{k=1}^{m-1}(m-k)W_{m-2k}-2\sum_{k=1}^{m-1}(m-k)V_kW_{m-k}\bigg]q(\V)\nonumber\\
	&=-\frac{2N+1}{2N}\left[u\p{u}\mathcal{R}^+P\right]_N\nonumber\\
	&\hspace{.5cm}+\frac{1}{2N}\left[-2\left(\sum_{k=1}^\infty u^{-k+1}\mathcal{R}^+\p{u}\mathcal{M}_{u^{-k}}\mathcal{R}^+\right)P\right]_N\nonumber\\
	&\hspace{.5cm}-2\left[\left(\sum_{k=1}^\infty v_ku\mathcal{R}^+\p{u}\mathcal{M}_{u^{-k}}\mathcal{R}^+\right)P\right]_N\nonumber\\
	&=-\frac{2N+1}{2N}\left[u\p{u}\mathcal{R}^+P\right]_N-\frac{1}{N}[\mathcal{Z}_2^+P]_N-2[\mathcal{Y}_1^+P]_N.\label{eq:DWmqVposSpN}
	\end{align}
	Similarly, for $m<0$, we apply (\ref{eq:DAmnegSpN}) to get 
	\begin{align*}
	(\Delta_{\Sp(N)} W_m)\cdot q(\V)&=\frac{2N+1}{2N}\left[u\p{u}\mathcal{R}^-P\right]_N+\frac{1}{N}[\mathcal{Z}_2^-P]_N+2[\mathcal{Y}_1^-P]_N.
	\end{align*}
	Since $\mathcal{Y}_1^+,\mathcal{Z}_2^+$ annihilates $\C[u^{-1}]$ while $\mathcal{Y}_1^-,\mathcal{Z}_2^-$ annihilates $\C[u]$, we have that for all $m\in\Z$, 
	\begin{align}
	(\Delta_{\Sp(N)} W_m)\cdot q(\V)&=-\frac{2N+1}{2N}[\mathcal{N}_1P]_N-\frac{1}{N}[\mathcal{Z}_2P]_N-2[\mathcal{Y}_1P]_N.\label{eq:DWmqVSpN}
	\end{align}
	
	For the middle term in \eqref{eq:LaplacianPNSpN}, we note the similarity between the derivative formulas \eqref{eq:XtrAmXAp} and \eqref{eq:cross.sum.spN} for the $\SO(N,\R)$ and $\Sp(N)$ cases, resp. Hence, as in \eqref{eq:XWmXqV}, 
	\begin{align}
	\sum_{X\in\beta_{\symp(N)}}\wX W_m\cdot\wX q(\V)&=\frac{1}{4N^2}[\K_1 P]_N.\label{eq:XWmXqVSpN}
	\end{align}
	
	For the last term in (\ref{eq:LaplacianPNSpN}), we have, for each $X\in\beta_{\symp(N)}$, 
	\begin{align}
	\wX^2 q(\V)&=\sum_{|k|\geq 1}\left(\p{v_k}q\right)(\V)\cdot\wX^2 V_k+\sum_{|j|,|k|\geq 1}\left(\pp{v_j}{v_k}q\right)(\V)\cdot \wX V_j\cdot \wX V_k.\label{eq:X2qVSpN},
	\end{align}
	as with \eqref{eq:X2qV}. 
	
	In summing the first term in (\ref{eq:X2qVSpN}) over $X\in\beta_{\symp(N)}$, we break up the sum for positive and negative $k$. Using (\ref{eq:DAmposSpN}) and (\ref{eq:DAmnegSpN}), we have
	\begin{align}
	\sum_{X\in\beta_{\symp(N)}}\sum_{k=1}^\infty\left(\p{v_k}q\right)(\V)\cdot\wX^2 V_k
	=&-\frac{2N+1}{2N}\sum_{k=1}^\infty kV_k\left(\p{v_k}q\right)(\V)\nonumber\\
	&-\frac{1}{N}\sum_{k=2}^\infty\sum_{\ell=1}^{k-1}(k-\ell)V_{k-2\ell}\left(\p{v_k}q\right)(\V)\nonumber\\
	&-2\sum_{k=2}^\infty\sum_{\ell=1}^{k-1}(k-\ell)V_\ell V_{k-\ell}\left(\p{v_k}q\right)(\V)\label{eq:N0Z1Y2posSpN}
	\end{align}
	and
	\begin{align}
	\sum_{X\in\beta_{\symp(N)}}\sum_{k=-\infty}^{-1}\left(\p{v_k}q\right)(\V)\cdot\wX^2 V_k
	=&\frac{2N+1}{2N}\sum_{k=-\infty}^{-1} kV_k\left(\p{v_k}q\right)(\V)\nonumber\\
	&-\frac{1}{N}\sum_{k=-\infty}^{-2}\sum_{\ell=k+1}^{-1}(k-\ell)V_{k-2\ell}\left(\p{v_k}q\right)(\V)\nonumber\\
	&-2\sum_{k=-\infty}^{-2}\sum_{\ell=k+1}^{-1}(-k+\ell)V_\ell V_{-k-\ell}\left(\p{v_k}q\right)(\V)\label{eq:N0Z1Y2negSpN}
	\end{align}
	
	Summing the second term in (\ref{eq:X2qVSpN}) over $X\in\beta_{\symp(N)}$ and using magic formula (\ref{eq:cross.sum.spN}) yields
	\begin{align}
	\sum_{X\in\beta_{\symp(N)}}\sum_{|j|,|k|\geq 1}&\left(\pp{v_j}{v_k}q\right)(\V)\cdot \wX V_j\cdot \wX V_k\nonumber\\
	&=\frac{1}{4N^2}\sum_{|j|,|k|\geq 1}jk(V_{k-j}-V_{k+j})\left(\pp{v_j}{v_k}q\right)(\V).\label{eq:K2SpN}
	\end{align}
	Adding (\ref{eq:N0Z1Y2posSpN}), (\ref{eq:N0Z1Y2negSpN}), and (\ref{eq:K2SpN}) together and multiplying by $W_m$, we get
	\begin{align}
	W_m\cdot (\Delta_{\Sp(N)} q(\V))&=\left[\left(-\frac{2N+1}{2N}\mathcal{N}_0-\frac{1}{N}\mathcal{Z}_1-2\mathcal{Y}_2+\frac{1}{4N^2}\mathcal{K}_2\right)P\right]_N. \label{eq:WmDqVSpN}
	\end{align}
	
	Combining (\ref{eq:DWmqVSpN}), (\ref{eq:XWmXqVSpN}), and (\ref{eq:WmDqVSpN}) proves (\ref{eq:intertwine1}) for $\Sp(N)$. Equation (\ref{eq:intertwine1exp}) now follows.
\end{proof}

\begin{proof}[Proof of Theorem \ref{thm:SBintertwine} for $\Sp(N)$] \label{pf:thm:SBintertwineSpN}
	The proof for $\Sp(N)$ is entirely analogous to the proof for $\SO(N,\R)$. Using the intertwining formula for $e^{\frac{t}{2}\Delta_{\Sp(N)}}$  (\ref{eq:intertwine1exp}), we have $e^{\frac{\tau}{2}\Delta_{\Sp(N)}}P_N=[e^{\frac{\tau}{2}\mathcal{D}_N^{(4)}}P]_N$.  Corollary \ref{cor:DNinvariant} shows that $[e^{\frac{\tau}{2}\mathcal{D}_N^{(4)}}P]_N$ is a well-defined trace polynomial on $\Sp(N)$, so its analytic continuation to $\Sp(N,\C)$ is given by the same trace polynomial function. Hence $[e^{\frac{\tau}{2}\mathcal{D}_N^{(4)}}P]_N$, viewed as a holomorphic function on $\Sp(N,\C)$, is the analytic continuation of $e^{\frac{\tau}{2}\Delta_{\Sp(N)}}P_N$, and so is equal to $\mathbf{B}_{s,\tau}^{\Sp(N)} P_N$. 
\end{proof}

\subsection{Intertwining formulas for $\mathcal{A}_{s,\tau}^{\Sp(N,\C)}$}

\begin{theorem} \label{thm:intertwiningIIpolySpN}
	Fix $s\in\R$ and $\tau=t+i\theta\in\C$. There are collections $\{T_{\varepsilon}^{s,\tau}\,:\,\varepsilon\in\mathscr{E}\}$, $\{U_{\varepsilon}^{s,\tau}\,:\,\varepsilon\in\mathscr{E}\}$ such that for each $\varepsilon\in\mathscr{E}$, $T_\varepsilon^{s,\tau}$ and $U_\varepsilon^{s,\tau}$ are certain finite sums of monomials of trace degree $|\varepsilon|$ such that
		\begin{equation} \label{eq:ASpstVe}
		\ASpst V_\varepsilon=[T_\varepsilon^{s,\tau}]_N+\frac{1}{N}[U_{\varepsilon}^{s,\tau}]_N,
		\end{equation}
		Moreover, let $\{S_{\varepsilon,\delta}^{s,\tau}\,:\,\varepsilon,\delta\in\mathscr{E}\}\subseteq\mathscr{W}$ be the collection of word polynomials from Theorem \ref{thm:intertwiningIIpoly}. Then for $\varepsilon,\delta\in\mathscr{E}$,
		\begin{equation} \label{eq:RedstSpN}
		\sum_{\ell=1}^{d_{\symp(N)}}\left[\left(s-\frac{t}{2}\right)(\wX_\ell V_\varepsilon)(\wX_\ell V_\delta)+\frac{t}{2}(\wY_\ell V_\varepsilon)(\wY_\ell V_\delta)-\theta(\wX_\ell V_\varepsilon)(\wY_\ell V_\delta)\right]=\frac{1}{N^2}[S_{\varepsilon ,\delta}^{s,\tau}]_N,
		\end{equation}
		where $\beta_{\symp(N)}=\{X_\ell\}_{\ell=1}^{d_{\symp(N)}}$ and $Y_\ell=iX_\ell$. 
\end{theorem}

For the proof, we will employ the conventions, mutandis mutatis, as those used in the proof of Theorem \ref{thm:intertwiningIIpoly} (see the remarks following the theorem statement). 

\begin{proof}
	Fix a word $\varepsilon=(\varepsilon_1,\hdots ,\varepsilon_m)\in\mathscr{E}$. Then for each $X\in\beta_{\pm}$ and $A\in\Sp(N)$, we apply the product rule twice to get
	\begin{align}
	(\wX^2 V_\varepsilon)(A)=&\sum_{j=1}^m\tr(A^{\varepsilon_1}\cdots (AX^2)^{\varepsilon_j}\cdots A^{\varepsilon_m}) \label{eq:trAXXSpN}\\
	&+2\sum_{1\leq j<k\leq m}\tr(A^{\varepsilon_1}\cdots (AX)^{\varepsilon_j}\cdots (AX)^{\varepsilon_k}\cdots A^{\varepsilon_m}). \label{eq:trAXAXSpN}
	\end{align}
	
	Applying magic formula \eqref{eq:magicSpN1v2} to each term in (\ref{eq:trAXAXSpN}), we have
	\begin{equation}
	\sum_{X\in\beta_\pm}\tr(A^{\varepsilon_1}\cdots (AX^2)^{\varepsilon_j}\cdots A^{\varepsilon_m})=\pm\frac{2N+1}{2N}\tr (A^{\varepsilon_1}\cdots A^{\varepsilon_j}\cdots A^{\varepsilon_m})=\pm\frac{2N+1}{2N}V_\varepsilon (A).
	\end{equation}
	
	Summing over $1\leq j\leq m$ now gives
	\begin{equation}
	\sum_{X\in\beta_\pm}\sum_{j=1}^m\tr(A^{\varepsilon_1}\cdots (AX^2)^{\varepsilon_j}\cdots A^{\varepsilon_m})=\frac{2N+1}{2N}n_\pm(\varepsilon)V_\varepsilon(A),
	\end{equation}
	where $n_\pm(\varepsilon)\in\Z$ and $|n_\pm (\varepsilon)|\leq|\varepsilon|$ (where the integers $n_\pm (\varepsilon)$ are not necessarily the same as in the $\SO(N,\R)$ case). For (\ref{eq:trAXAXSpN}), we can express each term in the sum as
	\begin{equation} \label{eq:trAXAXtermSpN}
	\tr(A^{\varepsilon_1}\cdots (AX)^{\varepsilon_j}\cdots (AX)^{\varepsilon_k}\cdots A^{\varepsilon_m})=\pm\tr(A^\ejkzero XA^\ejkone XA^\ejktwo),
	\end{equation}
	so $\ejkzero$, $\ejkone$, and $\ejktwo$ are substrings of $\varepsilon$ such that $\ejkzero\ejkone\ejktwo=\varepsilon$. Summing (\ref{eq:trAXAXtermSpN}) over $X\in\beta_\pm$ by magic formula (\ref{eq:magicSpN2v2}), we have
	\begin{align}
	\sum_{X\in\beta_\pm} \tr(A^{\varepsilon_1}\cdots (AX)^{\varepsilon_j}\cdots (AX)^{\varepsilon_k}\cdots A^{\varepsilon_m})&=\pm\bigg[-\frac{1}{2N}\tr(A^\ejkzero (A^\ejkone)^{-1}A^\ejktwo)\nonumber\\
	&\hspace{1.5cm}-\tr(A^\ejkzero A^\ejktwo)\tr(A^\ejkone)\bigg]\nonumber\\
	&=\pm\frac{1}{2N}V_\ejkthree (A) \pm V_{\ejkzero\ejktwo}(A)V_\ejkone(A),
	\end{align}
	where $\ejkthree$ is the word of length $|\varepsilon|$ such that $V_\ejkthree(A)=\tr(A^\ejkzero (A^\ejkone)^{-1}A^\ejktwo)$. Thus summing (\ref{eq:trAXAXSpN}) over $X\in\beta_\pm$ gives
	\begin{align}
	\sum_{X\in\beta_\pm}\sum_{1\leq j<k\leq m}&\tr(A^{\varepsilon_1}\cdots (AX)^{\varepsilon_j}\cdots (AX)^{\varepsilon_k}\cdots A^{\varepsilon_m})\nonumber\\
	&=\frac{1}{2N}\sum_{1\leq j<k\leq m}\pm V_{\ejkthree}(A)+\sum_{1\leq j<k\leq m}\pm V_{\ejkone}(A)V_{\ejkzero\ejktwo}(A).\label{eq:QepsilonSpN}
	\end{align}
	We define word polynomials corresponding to (\ref{eq:QepsilonSpN}) by
	\begin{align}
	T_{\varepsilon,0}^\pm&=\sum_{1\leq j<k\leq m}\pm v_{\ejkthree}\\
	T_{\varepsilon,1}^\pm&=\sum_{1\leq j<k\leq m}\pm v_{\ejkzero\ejktwo}v_{\ejkone}.
	\end{align}
	
	Now if $X\in\beta_+$ and $Y=iX\in\beta_-$, then
	\begin{align}
	(\wX\wY V_\varepsilon)(A)=&\frac{i}{2}\sum_{j=1}^m\pm \tr(A^{\varepsilon_1}\cdots (AX^2)^{\varepsilon_j}\cdots A^{\varepsilon_m}) \label{eq:itrAXXSpN}\\
	&+i\sum_{1\leq j<k\leq m}\pm \tr(A^{\varepsilon_1}\cdots (AX)^{\varepsilon_j}\cdots (AX)^{\varepsilon_k}\cdots A^{\varepsilon_m}).\label{eq:itrAXAXSpN}
	\end{align}
	Summing (\ref{eq:itrAXXSpN}) over $X\in\beta_+$ and applying \eqref{eq:magic1SpN} gives
	\begin{equation}
	\sum_{X\in\beta_+}\sum_{j=1}^m \pm \tr(A^{\varepsilon_1}\cdots (AX^2)^{\varepsilon_j}\cdots A^{\varepsilon_m})=\frac{2N+1}{2N}\eta(\varepsilon)V_\varepsilon(A)
	\end{equation}
	where $\eta(\varepsilon)\in\Z$ and $|\eta(\varepsilon)|\leq |\varepsilon|$. In addition, summing (\ref{eq:itrAXAXSpN}) over $X\in\beta_+$, we have
	\begin{align}
	\sum_{X\in\beta_+}\sum_{1\leq j<k\leq m}&\pm \tr(A^{\varepsilon_1}\cdots (AX)^{\varepsilon_j}\cdots (AX)^{\varepsilon_k}\cdots A^{\varepsilon_m})\nonumber\\
	&=-\frac{1}{2N}\sum_{1\leq j<k\leq m}V_{\ejkthree}(A)-\sum_{1\leq j<k\leq m}\pm V_{\ejkzero\ejktwo}(A)V_{\ejkone}(A).
	\end{align}
	
	We define corresponding word polynomials
	\begin{align}
	T_{\varepsilon,2}&=i\sum_{1\leq j<k\leq m}\pm v_{\ejkthree}\\
	T_{\varepsilon,3}&=i\sum_{1\leq j<k\leq m}\pm v_{\ejkzero\ejktwo}v_{\ejkone}.
	\end{align}
	Putting this together, we define
	\begin{align}
	T_\varepsilon^{s,\tau}&=\left(s-\frac{t}{2}\right)\left(n_+(\varepsilon)v_\varepsilon +  2T_{\varepsilon,1}^+\right)+\frac{t}{2}\left(n_-(\varepsilon)v_\varepsilon+2T_{\varepsilon,1}^-\right)-\theta\left(i\eta(\varepsilon)+2T_{\varepsilon,3}\right),\\
	U_\varepsilon^{s,\tau}&=\left(s-\frac{t}{2}\right)\left(-\frac{n_+(\varepsilon)}{2}v_\varepsilon +  T_{\varepsilon,0}^+\right)+\frac{t}{2}\left(-\frac{n_-(\varepsilon)}{2}v_\varepsilon+T_{\varepsilon,0}^-\right)-\theta\left(-i\frac{\eta(\varepsilon)}{2}+T_{\varepsilon,2}\right).
	\end{align}
	By construction, $T_\varepsilon^{s,\tau}$ and $U_\varepsilon^{s,\tau}$ are of the required form and satisfy (\ref{eq:ASpstVe}). 
	
	The proof of \eqref{eq:RedstSpN} is essentially identical to that of Theorem \ref{thm:intertwiningIIpoly}(2); this is due to the similarity of magic formulas \eqref{eq:magic3} and \eqref{eq:magic3SpN} for $\SO(N,\R)$ and $\Sp(N)$. 
\end{proof}

\begin{theorem}[Intertwining formula for $\ASpst$] \label{thm:ASpstintertwine}
	Fix $s\in\R$ and $\tau=t+i\theta\in\C$. Let $\left\{T_\varepsilon^{s,\tau}\, :\,\varepsilon\in\mathscr{E}\right\}$, $\left\{U_\varepsilon^{s,\tau}\, :\,\varepsilon\in\mathscr{E}\right\}$, and $\left\{S_{\varepsilon,\delta}^{s,\tau}\, :\,\varepsilon,\delta\in\mathscr{E}\right\}$ be as in Theorem \ref{thm:intertwiningIIpolySpN}, and $\wL_2^{s,\tau}$ be as in Theorem \ref{thm:intertwiningII}. Define first and second order differential operators
	\begin{align}
	\mathcal{C}_0^{s,\tau}&=\frac{1}{2}\sum_{\varepsilon\in\mathscr{E}}T_{\varepsilon}^{s,\tau}\p{v_\varepsilon}\\
	\mathcal{C}_1^{s,\tau}&=\frac{1}{2}\sum_{\varepsilon\in\mathscr{E}}U_{\varepsilon}^{s,\tau}\p{v_\varepsilon}
	\end{align}
	Then for all $N\in\N$ and $P\in\mathscr{W}$, 
	\begin{equation} \label{eq:intertwine2SpN}
	\frac{1}{2}\ASpst P_N=\left[\left(\mathcal{C}_0^{s,\tau}+\frac{1}{N}\mathcal{C}_1^{s,\tau}+\frac{1}{N^2}\wL_2^{s,\tau}\right)P\right]_N.
	\end{equation}
\end{theorem}

\begin{proof}
	For each $X\in\symp(N)$, we apply the chain rule to get
	\begin{align*}
	\wX^2P_N&=\sum_{\varepsilon\in\mathscr{E}}\wX\left[\left(\dP\right)(\V)\cdot (\wX V_\varepsilon)\right]\\
	&=\sum_{\varepsilon\in\mathscr{E}}\left(\dP\right) (\V)\cdot (\wX^2 V_\varepsilon)+\sum_{\varepsilon,\delta\in\mathscr{E}}\left( \ddP \right)(\V)\cdot (\wX V_\varepsilon)(\wX V_\delta).
	\end{align*} 
	Hence
	\begin{align*}
	\ASpst P_N=&\sum_{\varepsilon\in\mathscr{E}}\left(\dP\right)(\V)\cdot\ASpst V_\varepsilon\\
	&+\sum_{\varepsilon,\delta\in\mathscr{E}}\left(\ddP\right)(\V)\sum_{\ell=1}^{\dSON}\bigg[\left(s-\frac{t}{2}\right)(\wX_\ell V_\varepsilon)(\wX_\ell V_\delta)\\
	&\hspace{5cm}+\frac{t}{2}(\wY_\ell V_\varepsilon)(\wY_\ell V_\delta)-\theta(\wX_\ell V_\varepsilon)(\wY_\ell V_\delta)\bigg].
	\end{align*}
	The result now follows from Theorem \ref{thm:intertwiningIIpolySpN}. 
\end{proof}

\subsection{Limit theorems for the Segal-Bargmann transform on $\Sp(N)$} \label{pf:thm:freeSBlimitSpN}

In this section, we outline the proof of Theorem \ref{thm:freeSBlimit} for $\Sp(N)$. Since the techniques used are similar to those used in Section \ref{sec:limittheoremsSONR} to prove the $\SO(N,\R)$ case, we do not provide the full details. 

The following two lemmas are the analogues of Lemmas \ref{lem:bianeSON}, \ref{lem:trevalQ}, and \ref{lem:treval2} for $\Sp(N)$. The proofs are very similar, and so are not included: the key concentration result is again Lemma \ref{cor:expmatrixnorm}, and the only change required is to replace intertwining formulas \eqref{eq:intertwine1} and \eqref{eq:intertwine2} for $\SO(N,\R)$ with intertwining formulas \eqref{eq:intertwine1} and \eqref{eq:intertwine2SpN} for $\Sp(N)$ wherever applicable.

\begin{lemma}[L\'{e}vy, \cite{Levy:master}]
	For $s>0$ and $k\in\Z$, 
	\begin{equation}
	\lim_{N\to\infty}\int_{\Sp(N)}\tr (A^k)\, \rho_s^{\Sp(N)}(A)\, dA=\nu_k\left(s\right).
	\end{equation}
\end{lemma}

\begin{lemma} \label{lem:treval2SpN}
	Let $s>0$ and $\tau=t+i\theta\in\mathbb{D}(s,s)$. For any $Q\in\C [\mathbf{v}]$, 
	\begin{equation}
	\left(e^{\mathcal{C}_0^{s,\tau}}\iota(Q)\right)(\mathbf{1})=\pi_{s-\tau}Q. 
	\end{equation}
\end{lemma}

As in the proof of Theorem \ref{thm:freeSBlimit} for $\SO(N,\R)$, we first require the following analogue of Theorem \ref{thm:tracelimit}.

\begin{theorem} \label{thm:tracelimitSpN} 
	Let $s>0$ and $\tau=t+i\theta\in\mathbb{D}(s,s)$. For any $P\in\Cuuv$, 
	\begin{align}
	&\|P_N-[\pi_{s}P]_N\|^2_{L^2(\rho_{s}^{\Sp(N)})}=O\left(\frac{1}{N^2}\right),\hspace{.5cm}\mbox{and} \label{eq:tracelimitSpN}\\
	&\|P_N-[\pi_{s-\tau}P]_N\|^2_{L^2(\mu_{s,\tau}^{\Sp(N,\C)})}=O\left(\frac{1}{N^2}\right). \label{eq:tracelimitSpNC}
	\end{align}
\end{theorem}

\begin{proof}
	The proof is entirely analogous to the proof of Theorem \ref{thm:tracelimit}. Again, we replace intertwining formula \eqref{eq:intertwine2} for $\ASONst$ with intertwining formula \eqref{eq:intertwine2SpN} for $\ASpst$ wherever necessary. To show the $\Sp(N)$ version of \eqref{eq:treval2app} (where $\wL_0^{s,\tau}$ is replaced with $\mathcal{C}_0^{s,\tau}$), we use Lemma \ref{lem:treval2SpN} in place of Lemma \ref{lem:treval2}. The remainder of the proof is the same. 
\end{proof}

\begin{proof}[Proof of Theorem \ref{thm:SBintertwine} for $\Sp(N)$]
	The proof of the limit results \eqref{eq:freeSBlimit} and \eqref{eq:freeinverseSBlimit} is very similar to the $\SO(N,\R)$ case (see p. \pageref{pf:thm:freeSBlimit}), where we now use the polynomial $S^{(N)}=e^{\frac{\tau}{2}\mathcal{D}_N^{(4)}}f-e^{\frac{\tau}{2}\mathcal{L}_0}f\in C_m[u,u^{-1};\mathbf{v}]$ in place of $R^{(N)}$ and intertwining formula \eqref{eq:intertwine2SpN} in place of \eqref{eq:intertwine2} in \eqref{eq:RNLmunorm}.
	
	To show uniqueness, we define seminorms on $\Cuuv$ by
	\begin{alignat}{2}
	&\|P\|_{s,N}^{(4)}=\|P_N\|_{L^2(\rho_s^{\Sp(N)})},\hspace{1cm}&&\|P\|_{s,\tau,N}^{(4)}=\|P_N\|_{L^2(\mu_{s,\tau}^{\Sp(N,\C)})},\\
	&\|P\|_s^{(4)}=\lim_{N\to\infty}\|P\|_{s,N}^{(4)},\hspace{1cm}&&\|P\|_{s,\tau}^{(4)}=\lim_{N\to\infty}\|P\|_{s,\tau,N}^{(4)}.
	\end{alignat}
	Noting that
	\begin{equation}
	\|P\|_s^{(4)}=\|P\|_{s}^{(2)},
	\end{equation}
	the remainder of the proof proceeds as in the proof of uniqueness in Theorem \ref{thm:SBintertwine} for $\SO(N,\R)$.
\end{proof}

\subsection{Extending the magic formulas and intertwining formulas to $\wSpN$} \label{sec:SpNquaternion}

We conclude this section by showing a version of the magic formulas and intertwining results for $\Sp(N)$ hold for $\wSpN\subseteq\M_N(\HH)$. We have the following key relation: 
\begin{equation} \label{eq:SpNtraceintertwine}
\widetilde{\tr}(\Phi(A))=\Real\tr(A),\hspace{.3cm}A\in\M_N(\HH),
\end{equation}
which can be seen by the definition of $\psi$ (see \eqref{eq:quaternion.matrix.abcd}).

\begin{prop}[Magic formulas for $\wSpN$] \label{prop:magicSpN}
	Let $\beta_{\wsympN}$ be any orthonormal basis for $\wsympN$ with respect to the inner product (\ref{eq:innerprodwSpN}). For any $A,B\in M_N(\HH)$, 
	\begin{align}
	\sum_{X\in\beta_{\wsympN}}X^2&=-\frac{2N+1}{2N}I_N=-I_N-\frac{1}{2N}I_N \label{eq:magic1SpN}\\
	\sum_{X\in\beta_{\wsympN}}XAX&=-\frac{1}{2N}A^*-\Real\tr (A)I_N\label{eq:magic2SpN}\\
	\sum_{X\in\beta_{\wsympN}}\Real\tr(XA)X&=\frac{1}{4N^2}(A^*-A)\label{eq:magic3SpN}\\
	\sum_{X\in\beta_{\wsympN}}\Real\tr(XA)\Real\tr(XB)&=\frac{1}{4N^2}(\Real\tr (A^*B)-\Real\tr(AB))\label{eq:magic4SpN}
	\end{align}
\end{prop}

\begin{proof}
	We begin with the observation that for $A\in\M_N(\HH)$, 
	\begin{equation}
	\Omega\Phi(A)^\intercal\Omega^{-1}=\Phi(A^*). \label{eq:PhiAT}
	\end{equation}
	To see why this is the case, note that Lemma \ref{lem:complexquaternion} directly implies that for $q\in\HH$, 
	\begin{equation}
	\Omega_0\psi(q)^\intercal\Omega_0^{-1}=\psi(q^*). \nonumber
	\end{equation}
	Hence if $A=[q_{i,j}]$, 
	\begin{align}
	[\Omega\Phi(A)^\intercal\Omega^{-1}]^{i,j}&=\sum_{1\leq k,\ell\leq N}[\Omega]^{i,k}[\Phi(A)^\intercal]^{k,\ell}[\Omega]^{\ell,j}=[\Omega]^{i,i}[\Phi(A)^\intercal]^{i,j}[\Omega]^{j,j} \nonumber\\
	&=\Omega_0\psi(q_{j,i})^\intercal\Omega^{-1}_0=\psi(q_{j,i}^*)=[\Phi(A^*)]^{i,j}, \nonumber
	\end{align}
	which shows \eqref{eq:PhiAT}.
	
	Recall that if $\beta_{\wsympN}$ is an orthonormal basis for $\wsympN$, $\Phi(\beta_{\wsympN})$ is an orthonormal basis for $\symp(N)$. Let $A\in\M_N(\HH)$. Using  \eqref{eq:magicSpN2v2} and \eqref{eq:SpNtraceintertwine}, we have 
	\begin{align}
	\sum_{X\in\beta_{\wsympN}}\Phi(XAX)&=\sum_{X\in\beta_{\wsympN}}\Phi(X)\Phi(A)\Phi(X) \nonumber\\
	&=-\frac{1}{2N}\Omega\Phi(A)^\intercal\Omega^{-1}-\wtr (\Phi(A))I_{2N}\nonumber\\
	&=-\frac{1}{2N}\Phi(A^*)-\Real\tr(A)\Phi(I_N)\nonumber\\
	&=\Phi\left(-\frac{1}{2N}A^*-\Real\tr (A)I_N\right)\nonumber.
	\end{align}
	Since $\Phi$ is injective, \eqref{eq:magic2SpN} follows. The proof of \eqref{eq:magic3SpN} is similar, and \eqref{eq:magic1SpN} and \eqref{eq:magic4SpN} follow from \eqref{eq:magic2SpN} and \eqref{eq:magic3SpN}, resp. 
\end{proof}

An $\wSpN$ version of Theorem \ref{thm:intertwine1} also holds for an appropriately defined trace polynomial functional calculus: for $P\in\Cuuv$, we let $\widetilde{P}_N:\wSpN\to\M_N(\HH)$ be the function defined by
\begin{equation}
\widetilde{P}_N(A):=P(u;\mathbf{v})|_{u=A,v_k=\Real\tr (A^k),k\neq 0}.
\end{equation}

\begin{prop} \label{prop:SpNtraceintertwine}
	Let $P\in\Cuuv$ and $A\in\wSpN\subseteq\M_N(\HH)$. Then
	\begin{equation}
	P_N(\Phi(A))=\Phi(\widetilde{P}_N(A)).
	\end{equation}
\end{prop}

\begin{proof}
	It suffices to prove the result for $P(u;\mathbf{v})=u^mv_k^n$, with $m,n,k\in\Z$ and $k\neq 0$. Using \eqref{eq:SpNtraceintertwine}, we have
	\begin{align*}
	P_N(\Phi(A))&=\Phi(A)^m\widetilde{\tr}(\Phi(A)^k)^n=\Phi(A^m)(\Real\tr(A^k))^n=\Phi(A^m(\Real\tr(A^k))^n)=\Phi(\widetilde{P}_N(A)).\qedhere
	\end{align*}
\end{proof}

Using this proposition, we see that the intertwining formula for $\Sp(N)\subseteq\M_N(\HH)$ is a direct consequence of the intertwining formula \eqref{eq:intertwine1} for $\Sp(N)\subseteq\M_{2N}(\C)$.

\begin{theorem}[Intertwining formulas] \label{thm:intertwine1SpNv2}
	For any $P\in \Cuuv$, 
	\begin{equation}
	\Delta_{\wSpN} \widetilde{P}_N=\left[\mathcal{D}_N^{(4)}P\right]_N. \label{eq:intertwine1SpNv2}
	\end{equation}
	Moreover, for all $\tau\in\C$, 
	\begin{equation}
	e^{\frac{\tau}{2}\Delta_{\wSpN}}\widetilde{P}_N=[e^{\frac{\tau}{2}\mathcal{D}_N^{(4)}}P]_N.\label{eq:intertwine1SpNexpv2}
	\end{equation}
\end{theorem}

\begin{remark}
	We have seen that the magic formulas \eqref{prop:magicSpN} for $\wSpN$ can be derived from the magic formulas \eqref{prop:magicSpNv2} for $\Sp(N)$. However, the converse is not true. Suppose we only know that the magic formula \eqref{eq:magic2SpN} holds. Applying $\Phi$ to both sides yields
	\begin{equation}
	\sum_{X\in\beta_{\wsympN}}\Phi(X)\Phi(A)\Phi(X)=-\frac{1}{2N}\Phi(A)^*-\wtr (\Phi(A))I_{2N}.
	\end{equation}
	We cannot replace $\Phi(A)$ with an arbitrary $B\in\M_{2N}(\C)$ in the formula above. For example, consider $$B=\left[
	\begin{array}{cc}
	1 & 0 \\
	1 & 1 \\
	\end{array}
	\right]\in\Sp(1,\C).$$
	Then we can compute that
	\begin{align*}
	\sum_{Y\in\beta_{\symp(1)}}YBY&=\frac{1}{2}\left[
	\begin{array}{rr}
	-3 & 0 \\
	1 & -3 \\
	\end{array}
	\right]\neq \frac{1}{2}\left[
	\begin{array}{rr}
	-3 & -1 \\
	0 & -3 \\
	\end{array}
	\right]=-\frac{1}{2}B^*-\wtr (B)I_2.
	\end{align*}
	This is another reason why it is more convenient to work with $\Sp(N)$ rather than $\wSpN$. 
\end{remark}

\section{The Segal-Bargmann transform on $\SU(N)$}

In this final section, we analyze the Segal-Bargmann transform on the special unitary group $\SU(N)$. This case uses many of the techniques from the $\U(N)$ case, studied in \cite{DHK:U(N)}. Consequently, we outline the main ideas required for the proofs of these results and do not provide the full details. 

We first prove the intertwining formulas for $\SU(N)$. As in the $\SO(N,\R)$ and $\Sp(N)$ cases, the basis of these results is a set of magic formulas. We recall the following set of magic formulas for $\U(N)$, proven in \cite{DHK:U(N)}.

\begin{prop}[{\cite[Proposition 3.1]{DHK:U(N)}}]
	Let $\beta_{\uu(N)}$ be any orthonormal basis for $\uu(N)$ with respect to the inner product defined by $\inpro{X,Y}_{\uu(N)}=N^2\tr(XY^*)$. For any $A,B\in\MNC$,
	\begin{align}
	\sum_{X\in\beta_{\uu(N)}}X^2&=-I_N\label{eq:magic1UN}\\
	\sum_{X\in\beta_{\uu(N)}}XAX&=-\tr(A)I_N\label{eq:magic2UN}\\
	\sum_{X\in\beta_{\uu(N)}}\tr (XA)X&=-\frac{1}{N^2}A\label{eq:magic3UN}\\
	\sum_{X\in\beta_{\uu(N)}}\tr(XA)\tr(XB)&=-\frac{1}{N^2}\tr(AB)\label{eq:magic4UN}
	\end{align}
\end{prop}

The magic formulas for $\SU(N)$ follow easily from the magic formulas for $\U(N)$. 

\begin{prop}[Magic formulas for $\SU(N)$] \label{prop:magicSUN}
	Let $\beta_{\su(N)}$ be any orthonormal basis for $\su(N)$ with respect to the inner product \eqref{eq:inproSUN}. For any $A,B\in\MNC$,
	\begin{align}
	\sum_{X\in\beta_{\su(N)}}X^2&=\left(-1+\frac{1}{N^2}\right)I_N\label{eq:magic1SUN}\\
	\sum_{X\in\beta_{\su(N)}}XAX&=-\tr(A)I_N+\frac{1}{N^2}A\label{eq:magic2SUN}\\
	\sum_{X\in\beta_{\su(N)}}\tr (XA)X&=-\frac{1}{N^2}A+\frac{1}{N^2}\tr(A)I_N\label{eq:magic3SUN}\\
	\sum_{X\in\beta_{\su(N)}}\tr(XA)\tr(XB)&=-\frac{1}{N^2}\tr(AB)+\frac{1}{N^2}\tr(A)\tr(B)\label{eq:magic4SUN}
	\end{align}
\end{prop}

\begin{proof}
	The expressions on the left side of \eqref{eq:magic1SUN}, \eqref{eq:magic2SUN}, \eqref{eq:magic3SUN}, and \eqref{eq:magic4SUN} are independent of orthonormal basis chosen. Fix an orthonormal basis $\beta_{\su(N)}$ for $\su(N)$. Observe that the matrix $\frac{i}{N}I_N$ is an element of $\uu(N)$ of unit norm such that
	\begin{equation*}
	\inpro{\frac{i}{N}I_N,X}_{\uu(N)}=0\hspace{.5cm}\mbox{for all }X\in\su(N)\subseteq\uu(N).
	\end{equation*}
	Hence $\beta_{\uu(N)}=\beta_{\su(N)}\cup\{\frac{i}{N}I_N\}$ is an orthonormal basis for $\uu(N)$. Thus
	\begin{align*}
	\sum_{X\in\beta_{\su(N)}}X^2&=\sum_{X\in\beta_{\uu(N)}}X^2-\left(\frac{i}{N}I_N\right)^2=\left(-1+\frac{1}{N^2}\right)I_N,
	\end{align*}
	which proves \eqref{eq:magic1SUN}. The remaining magic formulas are proven similarly. 
\end{proof}

\begin{prop}[Derivative formulas for $\SU(N)$] \label{prop:derivativeSUN}
	The following identities hold on $\SU(N)$ and $\SL(N,\C)$:
	\begin{gather}
	\Delta_{\SU(N)}A^m=-2\Ind_{m\geq 2}\sum_{j=1}^{m-1}jA^j\tr(A^{m-j})-mA^m+\frac{m^2}{N^2}A^m,\hspace{.5cm}m\geq 0\label{eq:DAmposSUN}\\
	\Delta_{\SU(N)}A^m=2\Ind_{m\leq -2}\sum_{j=m+1}^{-1}jA^j\tr(A^{m-j})+mA^m+\frac{m^2}{N^2}A^m,\hspace{.5cm}m<0\label{eq:DAmnegSUN}\\
	\sum_{X\in\beta_{\su(N)}}\wX\tr(A^m)\cdot\wX A^p=\frac{mp}{N^2}(-A^{m+p}+A^p\tr(A^m)),\hspace{.5cm}m,p\in\Z.\label{eq:XtrAmXApSUN}
	\end{gather}		
\end{prop}

\begin{proof}
	The proof of these results is essentially identical to the proofs of the corresponding derivative formulas for $\U(N)$ in \cite[Theorem 3.3]{DHK:U(N)}; one need only to replace the magic formulas for $\U(N)$ by the magic formulas for $\SU(N)$ in the proof, wherever applicable. 
\end{proof}

The derivative formulas of Proposition \ref{prop:derivativeSUN} comprise the key ingredient for the intertwining formulas for $\Delta_{\SU(N)}$ and $\mathbf{B}_{s,\tau}^{\SU(N)}$ (Theorems \ref{thm:intertwine1} and \ref{thm:SBintertwine}). The proofs of these results are entirely analogous to the corresponding results for $\U(N)$ (cf. \cite[Theorems 1.8, 1.9]{DHK:U(N)}). The only change required in the proof is to replace the magic formulas and derivative formulas for $\U(N)$ by the corresponding formulas for $\SU(N)$ from Propositions \ref{prop:magicSUN} and \ref{prop:derivativeSUN}. By keeping track of these changes, we see that for $P\in\Cuuv$, the intertwining formula for $\Delta_{\SU(N)}$ is $$\Delta_{\SU(N)}P_N=[\mathcal{D}_N^{(2)}P]_N=\left[\left(\mathcal{L}_0-\frac{1}{N^2}(2\mathcal{K}_1^-+\mathcal{K}_2^--\mathcal{J})\right)P\right]_N.$$ The only difference between this and \eqref{eq:intertwineU(N)}, the intertwining formula for $\Delta_{\U(N)}$, is that for $\SU(N)$, the $1/N^2$ term on the right hand side contains the additional operator $\mathcal{J}$. This is a consequence of the close relationship between the magic formulas for $\U(N)$ and $\SU(N)$. \label{pf:thm:intertwine1SUN} \label{pf:thm:SBintertwineSUN}

Finally, our main result regarding the free Segal-Bargmann transform for $\SU(N)$, Theorem \ref{thm:freeSBlimit}, is proven in the same way as the corresponding result for $\U(N)$ (see \cite[Theorem 1.11]{DHK:U(N)}); again, the only change necessary is the substitution of the magic and derivative formulas for $\SU(N)$ in place of those for $\U(N)$ wherever required. \label{pf:thm:freeSBlimitSUN}

\subsection*{Acknowledgments}
The author thanks their PhD advisor Todd Kemp for suggesting the topic of the present paper, as well as his invaluable advice and extensive comments in completing this work. In addition, the author thanks the referee, whose suggestions greatly improved the presentation of the results.

\bibliographystyle{acm}
\bibliography{citations}

\end{document}